\documentclass[a4paper,11pt,onecolumn]{amsart}
\usepackage{amsfonts}

\usepackage{amscd}
\usepackage{amsmath}
\usepackage{amssymb}
\usepackage{amsthm}
\usepackage{epsf}
\usepackage{latexsym}
\usepackage{verbatim}
\usepackage{color}
\usepackage{todonotes}
\usepackage{enumitem}
\input epsf.tex
\usepackage{mathtools}
\usepackage{tikz-cd}

\usepackage{epstopdf}
\usepackage{graphicx}

\usepackage{centernot}
\usepackage{mathtools}
\usepackage{ stmaryrd }
\usepackage[all]{xy}

\usepackage{blkarray, bigstrut}

\usepackage{marginnote}

\newtheorem{theorem}{Theorem}[section]
\newtheorem{definition}[theorem] {Definition}
\newtheorem{lemma}[theorem]{Lemma}
\newtheorem{corollary}[theorem]{Corollary}
\newtheorem{remark}[theorem]{Remark}
\newtheorem{proposition}[theorem]{Proposition}

\theoremstyle{definition}

\newtheorem{example}[theorem]{Example}

\newcommand{\Relu}{\textrm{ReLU}}

\begin{document} 

\title{Functional Dimension of Feedfoward ReLU neural networks}

\author[J.E. Grigsby]{J. Elisenda Grigsby}
\thanks{JEG was partially supported by Simons grant 635578 and NSF grant number DMS - 2133822.}
\address{Boston College; Department of Mathematics; 522 Maloney Hall; Chestnut Hill, MA 02467}
\email{grigsbyj@bc.edu}

\author[K. Lindsey]{Kathryn Lindsey}
\thanks{KL was partially supported by NSF grant numbers DMS-1901247 and DMS - 2133822.}
\address{Boston College; Department of Mathematics; 567 Maloney Hall; Chestnut Hill, MA 02467}
\email{lindseka@bc.edu}

\author[R. Meyerhoff]{Robert Meyerhoff}
\thanks{RM was partially supported by NSF grant number DMS-1308642.}
\address{Boston College; Department of Mathematics; 569 Maloney Hall; Chestnut Hill, MA 02467}
\email{robert.meyerhoff@bc.edu}

\author[C. Wu]{Chenxi Wu}
\address{University of Wisconsin; Department of Mathematics;  517 Van Vleck Hall; Madison, WI 53706 }
\email{cwu367@math.wisc.edu}

\begin{abstract}  It is well-known that the parameterized family of functions representable by fully-connected feedforward neural networks with ReLU activation function is precisely the class of piecewise linear functions with finitely many pieces. It is less well-known that for every fixed architecture of ReLU neural network, the parameter space admits positive-dimensional spaces of symmetries, and hence the local functional dimension near any given parameter is lower than the parametric dimension. In this work we carefully define the notion of functional dimension, show that it is inhomogeneous across the parameter space of ReLU neural network functions, and continue an investigation -- initiated in \cite{RolnickKording} and \cite{PhuongLampert} -- into when the functional dimension achieves its theoretical maximum. We also study the quotient space and fibers of the realization map from parameter space to function space, supplying examples of fibers that are disconnected, fibers upon which functional dimension is non-constant, and fibers upon which the symmetry group acts non-transitively. 
\end{abstract}

\maketitle

\section{Introduction}
Given any parameterized family of mathematical objects, it is natural to ask how well the parameter space models the family. This can be viewed as a question about the realization map \[\rho: \{\mbox{Parameters}\} \longrightarrow \{\mbox{Objects}\}.\] What are the image and fibers of $\rho$?
In the current work, we examine this question for the class of feedforward ReLU neural network functions with domain $\mathbb{R}^{n_0}$ and codomain $\mathbb{R}^{n_m}$, for fixed $n_0, n_m \in \mathbb{N}$. It has been established \cite{arora2018} that this class coincides with the class of finite piecewise linear (PL) functions from $\mathbb{R}^{n_0}$ to $\mathbb{R}^{n_m}$. Recall that a feedforward ReLU neural network of architecture $(n_0, n_1, \ldots, n_m)$ is parameterized by $\mathbb{R}^D$, where $D = \sum_{i=1}^{m} n_{i}(n_{i-1} + 1)$. It follows that if one considers the direct sum, $\Omega$, of all parameter spaces $\mathbb{R}^D$, over all network architectures of feedforward ReLU neural networks with input dimension $n_0$ and output dimension $n_m$, the realization map \[\rho: \Omega \rightarrow \{\mbox{Finite PL functions } \mathbb{R}^{n_0} \rightarrow \mathbb{R}^{n_m}\}\] is surjective. 

The question of {\em reverse-engineering} the network architecture and parameters from a given finite PL function $F: \mathbb{R}^{n_0} \rightarrow \mathbb{R}^{n_m}$ is trickier. The idea of examining the quotient space of the parameter space of a multilayer perceptron by the realization map $\rho$ goes back at least as far as \cite{Amari1}, which refers to this quotient as the {\em neuromanifold}, and -- in follow-on work -- studies the generalization behavior and dynamics of gradient descent near its singularities.  For neural networks with sigmoidal activation functions, it was shown independently by \cite{FeffermanMarkel} for the $\tanh$ activation function and by \cite{AlbertiniSontag} for slightly more general sigmoidal activation functions\footnote{Specifically, \cite{AlbertiniSontag} requires that the activation function be smooth and have Taylor expansion around $0$ that agrees with $\tanh$ to second order.} that the architecture and defining parameter are uniquely determined by the function, up to a finite group of obvious symmetries.\footnote{These are the discrete symmetries: complementary sign flips and permutation of the neurons in a layer.} 

When the activation function is ReLU, it is well-known to the experts (cf. \cite{RolnickKording, PhuongLampert}) that $\rho$ is farther from injective. Indeed, there is an additional {\em positive-dimensional} space of symmetries coming from scaling/inverse-scaling the input/output to each neuron in the non-input layers. In \cite{RolnickKording}, Kording-Rolnick show that for most parameters in each depth $2$ architecture, reverse-engineering is still possible--modulo these known symmetries--by examining the geometry of the decomposition of the domain into linear regions. In \cite{PhuongLampert}, Phuong-Lampert use similar techniques to show that there exists a parameter in every non-widening architecture for which reverse-engineering is possible in the above sense.

The cornerstone observation of this paper is that the degree to which the realization map $\rho$ fails to be injective for ReLU neural network functions with fixed architecture is inhomogeneous across parameter space.  That is, in some regions of parameter space, a comparatively high-dimensional set of parameters all yield the same function, while in other regions of parameter space,  the set of parameters that correspond to a single function is lower-dimensional.   We introduce the notion of the \emph{functional dimension} of a parameter, $\theta$, in parameter space, $\textrm{dim}_{\textrm{fun}}(\theta)$; roughly speaking, $\textrm{dim}_{\textrm{fun}}(\theta)$ is the number of degrees of freedom attained in the space of functions by $\rho(\theta + \epsilon)$ under all infinitesimally small perturbations $\epsilon$ in parameter space.  Equivalently, if $D$ denotes the dimension of the parameter space, the difference $D- \textrm{dim}_{fun}(\theta)$ is a measure of the \emph{functional redundancy} at $\theta$, the number of linearly independent directions  $\epsilon$ in parameter space along which $\rho(\theta+\epsilon)$ is constant. 
 
We motivate the notion of functional dimension with a simple example.
 
 \begin{example}
Consider the architecture $(1,2,1)$.  The corresponding parameter space is $\mathcal{P} =\mathbb{R}^7$, and the realization map $\rho$ is given by 
$$\rho(a,b,c,d,e,f,g) \coloneqq x \mapsto \sigma \left( e \sigma(ax+b)+f \sigma (cx+d) + g \right) $$
Consider any parameter $\theta_0 = (a,\ldots,g) \in \mathcal{P}$ for which $a,b,c,e,f,g>0$ and  $d<0$.  Then $\rho(\theta_0)$ is the continuous piecewise affine-linear function that is the constant function $g$ on the interval  $(-\infty,-b/a]$, 
 has slope $ea$ on $[-b/a,-d/c]$, and slope $ea+fc$ on  $[-d/c,\infty)$.  Such functions are determined by 5 algebraically-independent scalars: two  non-differentiable bend points ($-b/a$ and $-d/c$), two slopes ($ea$ and $ea+fc$), and the vertical offset parameter ($g$).  We will say that $\textrm{dim}_{\textrm{fun}}(\theta_0) = 5$.  Correspondingly, in a small neighborhood $U \subset P$ of $\theta_0$, the \emph{fiber} 
 $$\rho^{-1} (\rho(\theta_0)) \coloneqq \{\theta \in \mathcal{P} \mid \rho(\theta) = \rho(\theta_0)\}$$
has dimension $\textrm{dim}(\mathcal{P}_{1,2,1}) - \textrm{dim}_{\textrm{fun}}(\theta_0) = 7-5 = 2$.

 In contrast, consider a parameter $\theta_1 \in \mathcal{P}$ with $a,b,d,e,g < 0$ and $c,f>0$.  Then $\rho(\theta_1)$ is the continuous function that is $0$ on the interval $(-\infty,  (-g-fd)/(fc)]$ and has slope $fc$ on the interval $[(-g-fd)/(fc),\infty)$.  Such a function is determined by 2 scalars: the  bend point $(-g-fd)/(fc)$ and the slope $fc$.  Thus we will say that $\textrm{dim}_{\textrm{fun}}(\theta_1) = 2$; there is a $5$-dimensional fiber consisting of parameters that all determine the function $\rho(\theta_1)$ passing through $\theta_1$. 
 \end{example}
 
\noindent After offering a precise definition of functional dimension (Definition \ref{def:functionaldimension})--well-defined at smooth points of the parametric family--we establish the existence of an upper bound on the functional dimension that is (except in the depth $1$ case) 
strictly less than $D$, the parametric dimension. This is precisely the upper bound that comes from incorporating the known positive-dimensional symmetries appearing in \cite{RolnickKording, PhuongLampert}. 
 
The gap between parametric and functional dimension -- as well as the variance of this gap across parameter space -- should have implications for training and convergence of feedforward ReLU networks via gradient descent \cite{TragerKohnBruna}, particularly for an overparameterized network (one for which the parametric dimension exceeds the size of the training set). Note that the functional dimension is a maximum over what we call the {\em batch} functional dimension relative to a finite data set (batch) in the domain. The maximum is attained when the batch is {\em decisive} (Definition \ref{d:decisive}) -- informally, when it covers a sufficiently large portion of the domain -- and this is unlikely to happen when our batch is small compared to the number of parameters. In the heavily overparameterized setting the parametric dimension far exceeds the batch functional dimension, which is bounded above by the product of the output dimension and the batch size. Moreover, each gradient descent update during training is confined to a subspace of the tangent space of parameter space that is bounded above by the batch functional dimension, where the batch is the training set, cf. Definition \ref{def:batchfundim} and the remarks immediately following. Note that this subspace, and its dimension, also evolves during training as the parameters are updated.

We also note that it has been established (see \cite{Belkin} and \cite{Cooper}) that for an overparameterized network with a smooth activation function, the global minimum of the MSE loss function is $0$ and the locus of global minima is positive dimensional (and indeed often nonconvex). Moreover, in \cite{Belkin}, the authors describe a local convexity condition they call the P\L$^*$ condition\footnote{Here ``P\L" stands for ``Polyak-{\L}ojasiewicz" and not ``piecewise-linear". We hope this will not cause too much confusion.} that guarantees convergence of gradient descent, for a training batch of size $N$, to a global minimum in a ball of radius $\mathcal{O}(1/N)$ around a parameter $\theta \in \mathbb{R}^D$ when a version of the neural tangent kernel at $\theta$  that we refer to here as the {\em batch neural tangent kernel}  has no zero eigenvalues. Since the  batch  neural tangent kernel is always a positive semi-definite matrix by construction, this occurs  in the overparameterized setting  precisely when the  batch  neural tangent kernel at $\theta$  relative to a training set (batch) of size $N$ has (maximal possible) rank $N$. 

We note that the results in \cite{Belkin} require a differentiable activation function in order to ensure that the neural tangent kernel is well-defined for all parameters. Here we study ReLU, which is non-differentiable at $0$. This implies that for almost all parameters, ReLU neural network functions are differentiable only in the complement of a codimension $1$ set \cite{GrigsbyLindsey}. In the current work, we define the notion of an {\em ordinary} parameter (Definition \ref{def:ordinary}) for a ReLU neural network architecture and a {\em parametrically smooth} point (Definition \ref{def:parametricallysmooth}) of the associated function. 
For ordinary parameters and finite data sets consisting of parametrically smooth points for these parameters, the neural tangent kernel is well-defined.  We define a generalization of the neural tangent kernel, the batch neural tangent kernel (Definition \ref{def:batchNTK}).  It follows immediately that batch functional dimension is equal to the rank of the batch neural tangent kernels, and  nearly immediately (Theorem \ref{t:ntkDim}) that functional dimension is the supremum of the ranks of the batch neural tangent kernels.

In \cite{Belkin}, the authors ask whether some version of the relationship between the P\L$^*$ condition and local convergence of gradient descent generalizes to ReLU neural networks. It would be interesting to investigate precisely how this convergence depends on the batch functional dimension of a parameter at initialization, and on the probability distribution from which the batch is drawn. 

Section \ref{s:continuity} investigates the continuity of functional dimension on the open, full-measure set of parameters where functional dimension is defined.  Section \ref{s:fibers} formally defines the symmetry group and fibers of the realization map.  It gives examples of fibers that are disconnected, fibers on which functional dimension is non-constant, and fibers upon which the symmetry group does not act transitively. 

\subsection{Related work} The notion of functional dimension we define and study here has appeared in the literature - in a supporting role - either prior to or concurrently (and independently) to the appearance of this work, largely as a tool for exploring questions of {\em identifiability} of a parameter from its function realization in neural network classes. The literature on this topic is vast and growing, so we apologize that our list of references is very likely incomplete. The notion of a {\em neuromanifold} - the image of the realization map for parameterized function classes used in learning - goes back at least to \cite{Amari1}, and one can regard the functional dimension of a parameter $\theta$ as the local dimension of the neuromanifold at $\theta$. In \cite{TragerKohnBruna} and \cite{KohnMontufarTrager}, the authors use algebraic geometry to study the neuromanifold and fibers of deep {\em linear} feedforward and convolutional neural network classes. In \cite{PetersenVoigt}, the authors study the topology of the neuromanifold for certain analytic activation functions (not including ReLU). In \cite{VenturiBB19}, the authors study the effects of parameter space symmetries on the optimization landscape for shallow (single hidden layer) networks with a variety of smooth activation functions, and along the way they define a global - not local - notion of complexity for these function classes that they call the intrinsic dimension. Work on these questions for ReLU network classes is more recent. In \cite{PhuongLampert, RolnickKording}, the authors tackle the question of identifiability and reverse-engineering a parameter from the data of the functions defined by the intermediate neurons in a ReLU network, and in \cite{StockGribonval} and \cite{BP1}, the authors define explicit imbeddings of the neuromanifold (relative to a batch of input points) into certain high-dimensional Euclidean spaces and give necessary and sufficient geometric conditions on the imbeddings for a parameter to be {\em locally} identifiable - that is, recoverable up to the positive scaling symmetry that appears in \cite{PhuongLampert, RolnickKording}. The geometric conditions appearing in \cite{StockGribonval} and \cite{BP1} implicitly involve batch functional dimension on a suitable batch.  Finally, we mention that in the years since our paper first appeared in 2022 and this 2025 revision in preparation for publication, a significant community at the intersection of mathematics and deep learning theory has formed around studying parameter-space symmetries, see for example \cite{GodfreyKvinge, ZhaoGWYD23} and \cite{ZhaoWaltersYu_survey} for a survey. This work most naturally fits into that intellectual community. Note that ReLU network symmetries are significantly more challenging to study formally than symmetries in architectures with smooth activation functions, since the complicated locus of non-differentiability needs to be incorporated into the story. The definitions we give here can be applied to a variety of activation functions, including other piecewise-linear activations like Leaky ReLU. However, the parameter-space symmetries for those activation functions will be significantly different, and so our results will not be applicable off-the-shelf for those. On the other hand, the techniques we develop should carry over nicely.

\tableofcontents

\subsection*{Acknowledgments}
The first two authors would like to thank Julian Asilis and Caleb Miller, who did some beautiful preliminary experimental investigations into functional dimension under the direction of our colleague Jean-Baptiste Tristan. The accompanying conversations in our weekly group meetings helped shape several of the ideas found here. We would also like to thank David Rolnick for interesting conversations and insights, and the anonymous referees for a number of excellent suggestions.


\section{Setup and background} \label{s:setup}
 
 \subsection{Feedforward ReLU neural networks, associated spaces, and the realization map}
 The \emph{rectified linear unit} or \emph{ramp} function is the function $\Relu:\mathbb{R} \to \mathbb{R}$ defined by $\Relu(x) = \max\{0,x\}$.  For any $n \in \mathbb{N}$, we denote by $\sigma$ the map  $\sigma:\mathbb{R}^n \to \mathbb{R}^n$ that applies $\Relu$ to each coordinate.  

  \begin{definition}\label{d:neuralnetwork} A \emph{feedforward ReLU neural network} of architecture \[(n_0,n_1,\ldots,n_m) \in \mathbb{N}^{m+1}\] is a finite, ordered collection of affine maps $A^{1},\ldots,A^{m}$ such that  $$A^{i}:\mathbb{R}^{n_{i-1}} \to \mathbb{R}^{n_{i}}$$ for  each $i = 1,\dots,m$.
   These affine maps determine a formal composition of maps  
  $$\bar{\rho}(\theta) \coloneqq \sigma \circ A^{m} \circ \sigma \circ A^{m-1} \circ \ldots \circ \sigma \circ A^{1},$$ which we call the associated \emph{marked neural network function}  and also determine a function  $F:\mathbb{R}^{n_0} \rightarrow \mathbb{R}^{n_m}$,  which we call the associated \emph{unmarked neural network function}, given by 
   $$F(x) \coloneqq  \sigma \circ A^{m} \circ \sigma \circ A^{m-1} \circ \ldots \circ \sigma \circ A^{1}(x).$$
\end{definition}

\begin{remark}
The word ``marked'' in ``marked neural network function'' indicates that $\bar{\rho}(\theta)$ retains the data of how the function is expressed as a sequence of nested ReLUs and affine maps involving the parameter coordinates. In contrast, the unmarked neural network function is simply a function from $\mathbb{R}^{n_0} \to \mathbb{R}^{n_m}$. \end{remark}

An affine map $A^{i}:\mathbb{R}^{n_{i-1}} \to \mathbb{R}^{n_{i}}$ is determined by a unique  $n_i \times (n_{i-1}+1)$ matrix $\left[A^{i}\right]$ that acts on the left:\begin{equation} \label{eq:matrixeq}
\left[A^{i}\right][ x  , 1] ^T = A^{i}(x)
\end{equation}
 for all $x \in \mathbb{R}^{n_{i-1}}$. (Here $[x , 1]$ denotes the length-$(n_{i-1}+1)$ vector formed by adding an additional final coordinate with the value $1$ to $x$.)  There is therefore a natural map from $\mathbb{R}^{\sum_{i=1}^m n_i(n_{i-1}+1)}$ to the set of neural networks of architecture $(n_0,\ldots,n_m)$; this realization map is the map that forms a neural network by using the coordinates of the point in $\mathbb{R}^{\sum_{i=1}^m n_i(n_{i-1}+1)}$ to determine affine maps $A^{1},\ldots,A^{m}$.

\begin{definition} For any network architecture  $(n_0,\dots,n_m)$, define the \emph{dimension of the parameter space} of neural networks of architecture $(n_0,\dots,n_m)$ to be the quantity 
$$D(n_0,\dots,n_m) \coloneqq \sum_{i=1}^m n_i(n_{i-1}+1)$$ and define the \emph{parameter space} of neural networks of architecture $(n_0,\ldots,n_m)$ to be the Euclidean space 
 $$\mathcal{P}_{n_0,\ldots,n_m} \coloneqq \mathbb{R}^{D(n_0,\ldots,n_m)}.$$ 
\end{definition}

\begin{definition} \ 
 \begin{enumerate} 
  \item 
  We denote by $\bar{\rho}$ the \emph{marked  realization map} that associates to a point in $\mathcal{P}_{n_0,\ldots,n_m}$ the corresponding marked neural network function.  
\item 
 We denote by $\rho$ the \emph{(unmarked) realization map}
 \[\rho:\mathcal{P}_{n_0,\ldots,n_m} \to C(\mathbb{R}^{n_0}, \mathbb{R}^{n_m})\] 
that sends a point $\theta \in \mathcal{P}_{n_0,\ldots,n_m}$ to the unmarked neural network map $F:\mathbb{R}^{n_0} \to \mathbb{R}^{n_m}$ associated to the parameter $\theta$.
\end{enumerate}
\end{definition}

\begin{definition} \label{def:moduliSpace}
For any network architecture $(n_0,\ldots,n_m)$, define the \emph{moduli space of neural networks of architecture $(n_0,\ldots,n_m)$} to be the set  
$$\mathcal{M}_{n_0,\ldots,n_m} \subset C(\mathbb{R}^{n_0}, \mathbb{R}^{n_m})$$ that is the image of the (unmarked)
realization map, i.e. 
$$\mathcal{M}_{n_0,\ldots,n_m} := \{\rho(\theta) : \theta \in \mathcal{P}_{n_0,\ldots,n_m}\}.$$
\end{definition}

\begin{remark}
Note that the sum of two functions in  $\mathcal{M}_{n_0,\ldots,n_m}$ is not necessarily in $\mathcal{M}_{n_0,\ldots,n_m}$.  Thus, moduli space is not a vector space.  The set  $\mathcal{M}_{n_0,\ldots,n_m}$ is, however, closed under scaling by nonnegative constants. 

\end{remark}

\subsection{Background on general polyhedral complexes}

 We briefly recall some relevant background about convex polytopes, polyhedral sets, and polyhedral complexes, referring the reader to \cite{Grunbaum, Grunert} for a more thorough treatment.

A {\em polyhedral set $\mathcal{P}$ in $\mathbb{R}^n$} is an intersection of finitely many closed affine half spaces $H_1^+, \ldots,  H_m^+ \subseteq \mathbb{R}^n.$ A {\em convex polytope} in $\mathbb{R}^n$ is a bounded polyhedral set. Note that a polyhedral set is an intersection of convex sets, and hence convex. 
A hyperplane $H$ in $\mathbb{R}^n$ is a {\em cutting hyperplane} of a polyhedral set $\mathcal{P}$ if $H \cap \textrm{interior}(\mathcal{P}) \neq \emptyset$; 
$H$ is a {\em supporting hyperplane} of $\mathcal{P}$ if $H$ does not cut $\mathcal{P}$ and $H \cap \mathcal{P} \neq \emptyset$. For any set $S \subset \mathbb{R}^n$, the {\em affine hull} of $S$, denoted $\mbox{aff}(S)$, is the intersection of all affine-linear subspaces of $\mathbb{R}^n$ containing $S$.  
  The \emph{dimension} $\textrm{dim}(\mathcal{P})$ of a polyhedral set $\mathcal{P}$ is that of its affine hull.

For a polyhedral set $\mathcal{P}$,  a subset $\mathcal{F} \subseteq \mathcal{P}$ is a {\em face} of $\mathcal{P}$ if either $\mathcal{F} = \emptyset$, $\mathcal{F} = \mathcal{P}$, or $\mathcal{F} = H \cap \mathcal{P}$ for some supporting hyperplane of $\mathcal{P}$.  $\emptyset$ and $\mathcal{P}$ are called the {\em improper} faces of $\mathcal{P}$, while all other faces are called {\em proper}.
 A {\em $k$--face} of $\mathcal{P}$ is a face of $\mathcal{P}$ that has dimension $k$. 
A {\em facet} of $\mathcal{P}$ is a $(\textrm{dim}(\mathcal{P})-1)$--face of $\mathcal{P}$.
 A {\em vertex} of $\mathcal{P}$ is a $0$--face of $\mathcal{P}$. 


 A {\em polyhedral complex} $\mathcal{C}$ of dimension $d$ ($\leq n$) in $\mathbb{R}^n$ is a {\em finite}  collection
 of polyhedral sets in $\mathbb{R}^n$ of dimension at most $d$ called the \emph{cells} of $\mathcal{C}$, such that i)  If $P \in \mathcal{C}$, then every face of $P$ is in $\mathcal{C}$, and ii)  if $P, Q \in \mathcal{C}$, then $P \cap Q$ is a single mutual face of $P$ and $Q$.  The \emph{domain} or \emph{underlying set} $|\mathcal{C}|$ of a polyhedral complex $\mathcal{C}$ is the union of its cells. If $|\mathcal{C} | = \mathbb{R}^n$, we call $\mathcal{C}$ a {\em polyhedral decomposition of $\mathbb{R}^n$}.

\subsection{The canonical polyhedral complex, generic and transversal neural networks}

 The \emph{affine solution set arrangement} associated to a layer map $\sigma \circ A :\mathbb{R}^i \to \mathbb{R}^j$ of a feedforward ReLU neural network is the finite ordered set $\{S_1,\ldots,S_j\}$, where each $S_k$, $1 \leq k \leq j$,  is the solution set 
$$S_k := \{\vec{x} \in \mathbb{R}^i :   [A]_k (\vec{x} | 1)^T= 0\},$$
where $[A]_k$ is the $k^{\textrm{th}}$ row of the matrix $[A]$.  Such an affine solution set arrangement $S = \{S_1, \ldots , S_j\}$ is said to be \emph{generic} if for all subsets $$\{S_{i_1} , \ldots , S_{i_p}\} \subseteq S,$$ it is the case that $S_{i_1} \cap \ldots \cap S_{i_p}$ is an affine-linear subspace of $\mathbb{R}^i$ of dimension $i - p$, where a negative-dimensional intersection is understood to be empty.  In the particular case that every solution set $S_k$, $1 \leq k \leq j$, is a hyperplane (i.e. has codimension $1$) in $\mathbb{R}^i$ we call $S = \{S_1, \ldots , S_j\}$ the \emph{hyperplane arrangement} associated to that layer map.  A hyperplane arrangement in $\mathbb{R}^i$ induces a polyhedral decomposition of $\mathbb{R}^i$.  

For a marked feedforward ReLU neural network map
\begin{equation} \label{eq:FdefNotation}
\bar{\rho}(s) = F: \mathbb{R}^{n_0}  \xrightarrow{F^{1}= \sigma \circ A^{1}} \mathbb{R}^{n_1}  \xrightarrow{F^{2}= \sigma \circ A^{2}} \ldots \xrightarrow{F^{m} = \sigma \circ A^{m}} \mathbb{R}^{n_m} 
\end{equation} the \emph{canonical polyhedral complex} $\mathcal{C}(\theta)$ (alternatively $\mathcal{C}(F)$)  is a polyhedral decomposition of $\mathbb{R}^{n_0}$ defined as follows.  For $i \in \{1,\ldots,m\}$, denote by $R^{i}$ the polyhedral complex on $\mathbb{R}^{n_{i-1}}$ induced by the hyperplane arrangement associated to the $i^{\textrm{th}}$ layer map, $F^{i}$.  Inductively define polyhedral complexes $\mathcal{C}(F^{1}),\ldots,\mathcal{C}(F^{m} \circ \dots \circ F^{1})$ on $\mathbb{R}^{n_0}$ as follows: Set $\mathcal{C}(F^{1}):= R^{1}$ and for $i = 2,\ldots,m$, set
\begin{multline*}
\mathcal{C}(F^{i} \circ \ldots \circ F^{1})
:= \\
 \left \{S \cap (F^{i-1} \circ \ldots \circ F^{1})^{-1}(Y) \mid S \in \mathcal{C}(F^{i-1} \circ \ldots \circ F^{1}), Y \in R^{i} \right \}.
 \end{multline*}
Set $\mathcal{C}(\theta) := \mathcal{C}(F^{m} \circ \ldots \circ F^{1})$. See \cite{Grunert, GrigsbyLindsey, Masden} to understand why this forms a polyhedral complex.

A layer map of a marked neural network is said to be \emph{generic} if the corresponding affine solution set arrangement is generic.  A parameter of the corresponding marked neural network whose layer maps are all generic is said to be \emph{generic}.

Let $\mathcal{C}$ be a polyhedral complex embedded in $\mathbb{R}^{n_0}$ and let $F:|\mathcal{C}| \to \mathbb{R}$ be a map that is affine-linear on cells of $\mathcal{C}$.  A point $x \in |\mathcal{C}|$ is said to have an \emph{$F$-nonconstant cellular neighborhood} (with respect to $\mathcal{C}$) if $F$ is nonconstant on each cell of $C \in \mathcal{C}$ that contains $x$.  A threshold $t \in \mathbb{R}$ is \emph{transversal} for the function $F$ and $\mathcal{C}$ if each point $p \in F^{-1}(\{t\})$ has an $F$-nonconstant cellular neighborhood. A threshold $t \in \mathbb{R}$ is a \emph{transversal threshold for a marked neural network with one-dimensional output}  $\bar{\rho}(\theta) = F:\mathbb{R}^{n_0} \to \mathbb{R}$
if $t$ is a transversal threshold for $F$ and its canonical polyhedral complex $\mathcal{C}(F)$. The $(i,j)th$ \emph{node map} of a parameter $s$ or marked neural network $F$ as in \eqref{eq:FdefNotation} is the map $$F_{i,j} \coloneqq \pi_j \circ F^{i} \circ \ldots \circ F^{1},$$ where $\pi_j$ denotes projection to the $j$th coordinate.  
A parameter $\theta$ (or marked neural network $\bar{\rho}(\theta)$)  is called \emph{transversal} if, for each $i \in \{1,\ldots,m\}$ and each $j \in \{1,\ldots,n_i\}$, $t=0$ is a transversal threshold for the node map $F_{i,j}: \mathbb{R}^{n_0} \to \mathbb{R}$ and the polyhedral complex $\mathcal{C}(F_{i,j})$. 

It was shown in \cite{GrigsbyLindsey} that the subset of any parameter space $\mathcal{P}_{n_0,\ldots,n_m}$ consisting of all generic and transversal parameters has full Lebesgue measure.  Moreover, \cite{GrigsbyLindsey} proved that if $\theta \in \mathcal{P}_{n_0, \ldots, n_m}$ is a generic and transversal parameter and $F = \bar{\rho}(\theta)$ is its marked realization map, then the $n_0$-cells of the canonical polyhedral complex, $\mathcal{C}(F)$, are the closures of the activation regions (defined in \S \ref{ss:ternarylabeling}), and the $(n_0-1)$-skeleton of $\mathcal{C}(F)$ is  the union of the zero sets of the node maps, \[\bigcup_{i,j} F^{-1}_{i,j}\{0\},\] which -- following \cite{HaninRolnick} -- we refer to as the \emph{bent hyperplane arrangement} associated to $F$. 

 \subsection{Ternary labeling} \label{ss:ternarylabeling}

The \emph{ternary labeling} of a point $x \in \mathbb{R}^{n_0}$ relative to the canonical polyhedral complex $\mathcal{C}(\bar{\rho}(\theta))$ 
is the sequence of ternary tuples 
\begin{equation} \label{eq:theta}
s_x := \left(s^{1}_x, \ldots, s^{m}_x\right) \in \{-1,0,1\}^{n_1 +\ldots + n_m}
\end{equation}
indicating the sign of the pre-activation output of each neuron of $\bar{\rho}(\theta)$ at $x$.

Explicitly, letting $F = \bar{\rho}(\theta)$ be as in \eqref{eq:FdefNotation}, for any input vector $x = x^{0} \in \mathbb{R}^{n_0}$, we will denote by 
\begin{equation} \label{eq:xell}
x^{\ell}:= (F^{\ell} \circ F^{\ell-1} \circ \ldots \circ F^{1})(x) \in \mathbb{R}^{n_\ell}
\end{equation}
the output of the first $\ell$ layer maps. Letting 
\begin{equation} \label{eq:zell}
y^{\ell}:= (A^{\ell} \circ F^{\ell-1} \circ \ldots \circ F^{1})(x) \in \mathbb{R}^{n_\ell}
\end{equation}
denote the pre-activation output of the first $\ell$ layer maps, the components of 
$$s^{\ell}_x \coloneqq \left(s^{\ell}_{x,1}, \,\, \ldots \,\,, s^{\ell}_{x,n_\ell}\right)$$
 are defined by $s^{\ell}_{x,i} = \mbox{sgn}(y_i^{\ell})$ (using the convention $\textrm{sgn}(0) = 0$).

An \emph{activation region} (for a parameter $\theta \in \mathcal{P}_{n_0,\ldots,n_m}$) is a maximal connected component of the set of input vectors $x \in \mathbb{R}^{n_0}$ at which the ternary labeling is constant and has no $0$s.   For generic, transversal networks, it follows from \cite{GrigsbyLindsey} that the activation regions of $\mathcal{C}(\bar{\rho}(\theta))$ are precisely the interiors of the $n_0$--cells of $\mathcal{C}(\bar{\rho}(\theta))$. Moreover, for all $\bar{\rho}(\theta)$ it follows immediately from the definitions that the ternary labeling is constant on the interior of each cell of $\mathcal{C}(\bar{\rho}(\theta))$, inducing a ternary labeling on the cells of $\mathcal{C}(\bar{\rho}(\theta))$, cf. \cite{Masden}.  If $s^{\ell}_{x,i} \leq 0$, we say that the $i$th neuron in the $\ell$th layer is \emph{off} or \emph{turned off} at $x$.

\begin{definition}
A neuron of $\bar{\rho}(\theta)$ is \emph{stably unactivated} if there exists an open neighborhood $U$ of $\theta$ in parameter space such that for every parameter $u \in U$ and every point $x \in \mathbb{R}^{n_0}$, that neuron of the marked function $\bar{\rho}(u)$ outputs a nonpositive number when $\bar{\rho}(u)$ is evaluated at $x$.  
\end{definition}

The machine learning literature sometimes refers to {\em stably unactivated} neurons as {\em dead}. The following notation will be useful for us. Let $A^{\ell} \in M_{n_{\ell} \times (n_{\ell -1 }+1)}$ be the matrix associated to the  affine part of the $\ell$th layer map $F^{\ell}: \mathbb{R}^{n_{\ell-1}} \to \mathbb{R}^{n_\ell}$. For $x=x^{0} \in \mathbb{R}^{n_0}$, let $s_{x} = (s_x^{1}, \ldots, s_x^{m})$ be the ternary labeling as in \eqref{eq:theta}.

\begin{itemize} 
	\item For any point $y \in \mathbb{R}^n$, denote by $\widehat{y}$ the \emph{augmented vector} $[y,1]^T \in \mathbb{R}^{n+1}$.
	\item Let $A^{\ell}_{s_x} \in M_{n_{\ell} \times (n_{\ell -1}+1)}$ denote the matrix obtained from $A^{\ell}$ by replacing the $i$th row with $\vec{0} \in \mathbb{R}^{n_{\ell -1} +1}$ if $s_{x,i}^{\ell} \leq 0$.
	\item For any matrix $B \in M_{m \times (n+1)}$ and $e_{n+1}^T = [0 \cdots 0 \,\, 1]$ the (transpose of the) $(n+1)$st coordinate vector in  $\mathbb{R}^{n+1}$, let $\widehat{B} \in M_{(m+1 )\times (n+1)}$ denote the \emph{augmented matrix} 
	$$\widehat{B} \coloneqq \left(\begin{array}{c} B\\ \tiny{e_{n+1}^T}\end{array}\right) \in M_{(m+1) \times (n+1)}.$$ 
	\item For any matrix $B$, let $B_{ij}$ denote the entry in the $i$th row and the $j$th column.
\end{itemize}

The following lemma is immediate.
\begin{lemma} \label{lem:terntuplecomp} Let \[F = F^{m} \circ \ldots \circ F^{1}: \mathbb{R}^{n_0} \rightarrow \mathbb{R}^{n_m}\] be a ReLU neural network, where $F^{i}$ has associated affine-linear map represented by the matrix $A^{i} \in M_{n_{i} \times (n_{i-1}+1)}$, and let $x = x^{0} \in \mathbb{R}^{n_0}$ be an input vector in the interior of a cell $C$ with associated ternary tuple $s_C = \left(s_x^{1}, \ldots, s_x^{m}\right)$. Then for each $\ell$, we have
 \[\widehat{x}^{\ell} =  \widehat{A}_{s_x^{\ell}}^{\ell}  \ldots \widehat{A}^{1}_{s_x^{1}} \widehat{x}^{0},\]
 where $x^{\ell}$ is as defined in \eqref{eq:xell}. 
\end{lemma}

 The following lemma will be useful for understanding the partial derivatives of the function $F$ with respect to the parameters. 

\begin{lemma} \label{lem:neuronoff} Suppose that the $i$th neuron in the $\ell$th layer is turned off for a data point $x \in \mathbb{R}^{n_0}$ (i.e. ${s}^{\ell}_{x,i} \leq 0$ at $x$). Then, when the relevant partial derivatives of $F$ are well-defined, we have: 
\begin{enumerate}
	\item $\frac{\partial F}{\partial A_{ij}^{\ell}} = 0$ at $x$ for all $j$, and
	\item $\frac{\partial F}{\partial  A_{ki}^{\ell+1}} = 0$ at $x$ for all $k$.
\end{enumerate}
\end{lemma}

\begin{proof}  The $i$th row of $A^{\ell}_{s_x^{\ell}}$, and hence the $i$th element of the output of the node map $F_\ell$, is $0$ at $x$. This implies that the value of $F$ on a point in this region has no dependence on any parameters in the $i$th row of $A^{\ell}$ or the $i$th column of $A^{\ell+1}$, and hence, the partial derivatives of the components of $F$  with respect to these parameters  are $0$ when they are well-defined.

\end{proof}

In Section \ref{sec:smoothpoints}, we give sufficient conditions ensuring that the partial derivatives in Lemma \ref{lem:neuronoff} are well-defined at $x$.

\subsection{Rank of a smooth map}
We briefly recall the definition of the rank of a smooth map and establish notation.  
For a differentiable map $f:\mathbb{R}^n \to \mathbb{R}^m$, the \emph{Jacobian matrix} of $f$ at a point $p \in \mathbb{R}^n$ is the $m \times n$ matrix of partial derivatives
$$\mathbf{J}(f) \vert_p = \left[ \frac{\partial f_i}{\partial x_j}(p) \right] $$
where $f_i$ denotes the $i$th coordinate function of $f$.  
If $F:M \to N$ is a smooth map between smooth manifolds,  the \emph{rank} of $F$ at a point $p \in M$, which we denote $\textrm{rank }F\vert_p$, is the rank of the linear map $F_*:T_pM \to T_{F(p)}N$, which coincides with the rank of the Jacobian matrix $\mathbf{J}F\vert_p$ with respect to any smooth coordinate chart.  
We will denote the rank in various equivalent ways, depending on what we want to emphasize:
$$\textrm{rank } F \vert_p = \textrm{rank }(\mathbf{J}F) \vert_p = \textrm{rank }\mathbf{J}F(p).$$

\section{Smoothness of the parameterized family $\mathcal{F}_{n_0,\ldots,n_m}$} \label{sec:smoothpoints}

Our definition of the functional dimension of a parameter $\theta$ (as well as the definition of the neural tangent kernel) involves taking partial derivatives of $\rho(\theta)(x)$, for fixed $x$ in the domain of $\rho(\theta)$, with respect to the coordinates of $\theta$.  Since the ReLU activation function is not everywhere differentiable, we must preface an investigation of such derivatives with a careful treatment of the locus of differentiability (with respect to both coordinates of $\theta$ and $x$) of $\rho(\theta)(x)$. 

 \begin{definition}
 The \emph{parameterized family} of neural networks of architecture $(n_0,\ldots,n_m)$ is the map $\mathcal{F}_{n_0,\ldots,n_m}: \mathcal{P}_{n_0,\ldots,n_m} \times \mathbb{R}^{n_0} \to \mathbb{R}^{n_m}$ defined by 
 $$\mathcal{F}_{n_0,\ldots,n_m}(\theta,x) = \rho(\theta)(x).$$
 \end{definition}

\subsection{The parameterized family is finitely piecewise polynomial}

This subsection defines the term finitely piecewise polynomial and proves (Theorem \ref{t:Fpiecewisepolynomial}) that 
$\mathcal{F}_{n_0,\ldots,n_m}$ is finitely piecewise polynomial (which is used in the proof of Theorem \ref{t:ordinaryparamsfull}). 

 \begin{definition}
 Define a \emph{piece} of $\mathbb{R}^n$ to be a subset of $\mathbb{R}^n$ that can be written as the closure of a nonempty, open, connected subset of $\mathbb{R}^n$.  
 \end{definition}

 \begin{definition}
 Let $k \in \mathbb{N}$.  A function $f:\mathbb{R}^{n} \to \mathbb{R}^{m}$ is said to be \emph{finitely piecewise polynomial} if there exist finitely many pieces 
 $X_1,\ldots,X_k$ in $\mathbb{R}^n$ such that $\mathbb{R}^n = \bigcup_{i=1}^k X_i$ and for each $1 \leq i \leq k$ and each $1 \leq j \leq m$, the $j^{\textrm{th}}$ coordinate function of the restriction of $f$ to $X_i$,  $\pi_j(f\vert_{X_i})$, is a polynomial.  
  \end{definition}
 
 \begin{lemma} \label{l:sufficientconditionnondegenerate}
Let $\theta \in \mathcal{P}_{n_0,\ldots,n_m}$ and let $x \in \mathbb{R}^{n_0}$.  Assume that there are no $0$'s in the ternary labeling at $x$ associated to $\bar{\rho}(\theta)$.  
Then there is an open neighborhood of $(\theta,x)$ in $\mathcal{P}_{n_0,\ldots,n_m} \times \mathbb{R}^{n_0}$ on which (each coordinate function of) $\mathcal{F}_{n_0,\ldots,n_m}$ is a polynomial.  
 \end{lemma}
 
 \begin{proof}
  This follows immediately from Lemma \ref{lem:terntuplecomp}, since the coordinates of a composition of linear transformations are polynomial in the entries of the matrices and the input coordinates.
 \end{proof}

 \begin{theorem} \label{t:Fpiecewisepolynomial}
 For any architecture $(n_0,\ldots,n_m)$, the parametrized family  $\mathcal{F}_{n_0,\ldots,n_m}$ is continuous and finitely piecewise polynomial. 
 
 Furthermore, we may take the pieces to be the closures of the connected components of  $(\mathcal{P}_{n_0,\ldots,n_m} \times \mathbb{R}^{n_0}) \setminus V$, where 
 $V \subset \mathcal{P}_{n_0,\ldots,n_m} \times \mathbb{R}^{n_0}$ is the set consisting of all points $(\theta,x)$ where  the ternary labeling at $x$ associated to $\bar{\rho}(\theta)$ contains at least one $0$. 
 \end{theorem}  

 \begin{proof}
 By the proof of \cite[Theorem 4]{GrigsbyLindsey}, the set $V$  is the vanishing set of a polynomial, and the complement of $V$ consists of finitely many connected components.  Let $\mathring{P}$ denote one of these (open) connected components. 
 
  Lemma \ref{l:sufficientconditionnondegenerate} guarantees that for every $(\theta,x) \in \mathring{P}$, there is an open neighborhood $U_\theta$ of $(\theta,x)$ in $\mathcal{P}_{n_0,\ldots,n_m} \times \mathbb{R}^{n_0}$ such that the restriction of  $\mathcal{F}_{n_0,\ldots,n_m}$ to $U_\theta$ is a polynomial.  If two polynomials agree on an open set, they are the same polynomial (by, for example, the Identity Theorem in complex analysis for analytic functions); thus all of the maps $\mathcal{F}_{n_0,\ldots,n_m}\vert_{U_\theta}$ are restrictions of a unique polynomial on $\mathring{P}$.  Hence,  the restriction of $\mathcal{F}_{n_0,\ldots,n_m}$ to the open set $ \mathring{P}$ is a polynomial.  It is a standard fact that the restriction of a polynomial on an open subset of Euclidean space admits a unique continuous extension to the boundary of the subset (stemming from the fact that a polynomial is locally uniformly continuous) -- and that extension is the polynomial (restricted to the closure of the subset).  Since $\mathcal{F}_{n_0,\ldots,n_m}$ is continuous, its restriction to the closure of $\mathring{P}$ is therefore a polynomial.  
\end{proof}

 \begin{corollary} \label{cor:NonDiffInBHA}
 A sufficient but not necessary condition for $(\theta,x)$ to be a smooth point for $\mathcal{F}_{n_0,\ldots,n_m}$ is that  the ternary labeling at $x$ has no $0$s with respect to $\bar{\rho}(\theta)$. 
  \end{corollary}
 
 \begin{proof}
Sufficiency of the condition is an immediate consequence of Theorem \ref{t:Fpiecewisepolynomial}. 
We give an example that shows that the condition is not necessary.  Consider the architecture $(1,1,1)$, which is parametrized by the components of $A^{1}, A^{2}$:
\begin{eqnarray*}
	A^{1} &=& \left[ \begin{array}{cc} w^{1} & b^{1}\end{array}\right]\\
	A^{2} &=& \left[\begin{array}{cc} w^{2} & b^{2}\end{array}\right].
\end{eqnarray*}
Let $\theta = \left(w^{1},b^{1},w^{2},b^{2}\right) = (1,0,1,-1)$. A straightforward calculation shows us that there are $3$ top-dimensional cells of the associated polyhedral decomposition of the domain, $\mathbb{R}$: \[\{(-\infty,0],[0,1],[1,\infty)\}.\] 
Direct calculation shows that $\rho(\theta)= 0$ when restricted to both adjacent cells $(-\infty,0]$ and $[0,1]$:
$$\rho(\theta)(x) = 
\begin{cases} 
0&  \textrm{ if } x \leq 1 \\ 
x-1 &  \textrm{ if } x \geq 1
\end{cases}$$
 Although the ternary labeling for $x=0$ with respect to $\theta = (1,0,1,-1)$ is $\left(s_x^{1}, s_x^{2}\right) = (0,-1)$, $(x,\theta)$ is a smooth point for $\mathcal{F}_{n_0,\ldots,n_m}$ since the output neuron is stably unactivated at $x = 0$, so $\rho(\theta)$ is locally constant (and hence smooth) with respect to both $x$ and the coordinates of $\theta$. 
 \end{proof}   

 \begin{remark}
Corollary \ref{cor:NonDiffInBHA} implies the fact, well-known to the experts, that the locus of non-differentiability/non-linearity of $F={\rho}(\theta)$ is contained in the union of the cells $C$ of $\mathcal{C}(F)$ for which the ternary label $s_C$ contains at least one $0$. Hanin-Rolnick,  in \cite{HaninRolnick}, call this latter locus the {\em bent hyperplane arrangement}, and \cite{TropGeometry} call this latter locus the {\em tropical hypersurface} of the corresponding tropical rational function. Also well-known--and illustrated by the example above--is the fact that the locus of non-differentiability is often a proper subset of the bent hyperplane arrangement. 
 \end{remark}
 
 \begin{remark}
 Theorem \ref{t:Fpiecewisepolynomial} seems also to be well-known to the experts, although we haven't found it stated formally in the literature. Note that each polynomial in the decomposition is {\em multi-affine-linear}. That is, it is affine-linear when all but one variable is held constant.
 \end{remark}

 \subsection{Parametrically smooth points} \label{ss:parametricallysmoothpoints}
 
 This subsection defines parametrically smooth points, ordinary parameters, and proves that the ordinary parameters form an open, dense, full-measure subset of parameter space (Theorem \ref{t:ordinaryparamsfull}).

      \begin{definition} \label{def:parametricallysmooth}
  Let $(n_0,\ldots,n_m)$ be an architecture and $\theta \in \mathcal{P}_{n_0,\ldots,n_m}$.  
  \begin{enumerate}
  \item A point $x \in \mathbb{R}^{n_0}$ is \emph{parametrically smooth for $\theta$} if $(\theta,x)$ is a smooth point for $\mathcal{F}_{n_0,\ldots,n_m}$. Note that {\em smooth} means $C^\infty$. See Remark \ref{rmk:Cinfty}.
  

  \item   A finite set $X \subset \mathbb{R}^{n_0}$ is  \emph{parametrically smooth for $\theta$} if every point $x \in X$ is parametrically smooth  for  $\theta$.
  \end{enumerate}
 
 \end{definition}

      \begin{remark} \label{r:nonordinary}
We give a simple example of a parameter that does not have any parametrically smooth points.  
Consider the architecture $(1,1)$ with the parametrized family of maps given by
$$\rho(a,b)(x) = \textrm{ReLU}(ax+b.)$$  We claim that the parameter $\theta=(0,0)$ admits no parametrically smooth points.  

Suppose $x > 0$.  Then 
$$\lim_{\epsilon \to 0^+} \frac{\rho(0+\epsilon,0)(x)}{\epsilon}  = \lim_{\epsilon \to 0^+} \frac{ \epsilon x}{\epsilon} = x > 0 $$ 
but 
$$\lim_{\epsilon \to 0^-} \frac{\rho(0+\epsilon,0)(x)}{\epsilon} = \lim_{\epsilon \to 0^-} \frac{0}{\epsilon} = 0, $$
so the derivative $ \left. \frac{\partial \rho(a,b)(x)}{\partial a} \right \vert_{(a,b) = s}$ does not exist.  

If $x < 0$,  a similar computation shows the derivative $ \left. \frac{\partial \rho(a,b)(x)}{\partial a} \right \vert_{(a,b)=s}$ does not exist.

Suppose $x = 0$.  Then 
$$\lim_{\epsilon \to 0^+} \frac{\rho(0,0+\epsilon)(x)}{\epsilon} = \lim_{\epsilon \to 0^+} \frac{\epsilon}{\epsilon} = 1$$
but 
$$\lim_{\epsilon \to 0^-} \frac{\rho(0,0+\epsilon)(x)}{\epsilon} =  \lim_{\epsilon \to 0^-} \frac{0}{\epsilon} = 0,$$
so the derivative $\left. \frac{\partial \rho(a,b)(x)}{\partial b} \right \vert_{(a,b) = s}$ does not exist. 

Thus, no point $x \in \mathbb{R}$ is parametrically smooth for $\theta$. 
\end{remark}

 Fortunately, parameters $\theta \in \mathcal{P}_{n_0,\ldots,n_m}$ that have no parametrically smooth points  are rare, as the next two lemmas show:
 
 \begin{lemma}\label{l:ordinaryparamsdense}
Fix any architecture $n_0,\ldots,n_m$.  Then for Lebesgue-almost every parameter $\theta \in \mathcal{P}_{n_0,\ldots,n_m}$, Lebesgue-almost every point $x \in \mathbb{R}^{n_0}$ is parametrically smooth for $\theta$.  
\end{lemma}

\begin{proof}
This is an immediate consequence of \cite[Theorem 4]{GrigsbyLindsey}, which asserts that the parametrized family $\mathcal{F}_{n_0,\ldots,n_m}$ is smooth on the complement of a set $Z \subset \mathcal{P}_{n_0,\ldots,n_m} \times \mathbb{R}^{n_0}$ that is the vanishing set of a polynomial, and hence is a null set for Lebesgue measure. 

If there were a positive measure set of parameters for each of which a positive measure set of points was non-parametrically smooth, it would follow that $\mathcal{F}_{n_0,\ldots,n_m}$ would fail to be smooth on a set of positive measure.  
\end{proof}

\begin{definition} \label{def:ordinary}
For any architecture $n_0,\ldots,n_m$, a parameter $\theta \in \mathcal{P}_{n_0,\ldots,n_m}$ is \emph{ordinary} if there exists at least one point $x \in \mathbb{R}^{n_0}$ that is parametrically smooth for $\theta$. 
\end{definition}

\begin{lemma} \label{l:generictransversalordinary}
Every generic, transversal parameter is ordinary. Furthermore, if $x$ is a point in the interior of a top-dimensional cell of the canonical polyhedral complex $\mathcal{C}(\theta)$, where $\theta$ is a generic, transversal parameter, then $x$ is a parametrically smooth point for $\theta$.  
\end{lemma}

\begin{proof}
Since $\theta$ is generic and transversal, the bent hyperplane arrangement is codimension $1$ in $\mathbb{R}^{n_0}$ and coincides with the $(n_0-1)$-skeleton of $\mathcal{C}(\theta)$  (\cite{GrigsbyLindsey}). Thus, if $C$ is top-dimensional, $x \in \textrm{interior}(C)$ implies $x$ is not in the bent hyperplane arrangement for $\theta$, meaning the ternary labeling of $x$ with respect to every neuron of $\bar{\rho}(\theta)$ is nonzero.  Lemma \ref{l:sufficientconditionnondegenerate} implies that $x$ is a smooth point for $\theta$.  \end{proof}

We remark that the proof above shows that every generic, transversal parameter $\theta$ is extraordinary; the parametrically smooth points for $\theta$ comprise a full measure subset of the domain.

\begin{lemma} \label{l:opennbhd1}
Let $(n_0,\ldots,n_m)$ be an architecture and fix a parameter $\theta \in \mathcal{P}_{n_0,\ldots,n_m}$. Then, $x \in \mathbb{R}^{n_0}$ is a parametrically smooth point for $\theta$ if and only if $(\theta,x)$ is in the interior of a maximal piece on which $\mathcal{F}_{n_0,\ldots,n_m}$ is polynomial.  
\end{lemma}

\begin{proof}
By Theorem \ref{t:Fpiecewisepolynomial}, $\mathcal{F}_{n_0,\ldots,n_m}$ is finitely piecewise polynomial.  If $(\theta,x)$ is in the interior of a piece on which $\mathcal{F}_{n_0,\ldots,n_m}$ is a polynomial, we are done, because polynomials are smooth everywhere.  So suppose $(\theta,x)$ is not in the interior of a piece on which $\mathcal{F}_{n_0,\ldots,n_m}$ is a polynomial.  Then, without loss of generality, we may assume $(\theta,x)$ is on the boundary of two closed regions such that the restriction of $\mathcal{F}_{n_0,\ldots,n_m}$ to each of the two regions is a polynomial, and these polynomials are different.  Because $\mathcal{F}_{n_0,\ldots,n_m}$ is continuous, these two polynomials cannot differ only in their constant terms (while agreeing at all higher order terms).  Hence, these two polynomials differ at some non-constant term.  Consequently, the derivative of $\mathcal{F}$ corresponding to this term is not defined at $(\theta,x)$, i.e. $x$ is not a parametrically smooth point for $\theta$, a contradiction.  
\end{proof}

\begin{remark} \label{rmk:Cinfty}
Lemma \ref{l:opennbhd1} requires $\mathcal{F}$ to be \emph{smooth} ($C^{\infty}$), and not just continuously differentiable ($C^1$), at $(\theta,x)$.  This is why, even though $\mathcal{F}$ being $C^1$ at $(\theta,z)$ is enough to define $\mathbf{J}E_z(\theta)$ (defined in \eqref{d:evaluationMapPtDef}), we require in the definitions of stochastic, batch and full functional dimension that the points $z \in Z$ satisfy the stronger property of being  smooth ($C^{\infty}$) for $\mathcal{F}$.  
\end{remark}

The following is an immediate corollary of Lemma \ref{l:opennbhd1}.

\begin{corollary} \label{c:ordinaryparamsopen}
For any architecture $(n_0,\ldots,n_m)$, the set 
$$\left\{(\theta,x) \in \mathcal{P}_{n_0,\ldots,n_m} \times \mathbb{R}^{n_0} \mid \mathcal{F}_{n_0,\ldots,n_m} \textrm{ is smooth at } (\theta,x)\right\}$$ 
is an open set.  
\end{corollary}

\begin{theorem} \label{t:ordinaryparamsfull}
Fix any architecture $(n_0,\ldots,n_m)$.  The set of ordinary parameters in $\mathcal{P}_{n_0,\ldots,n_m}$ is open, dense, and has full (Lebesgue) measure. 
\end{theorem}

\begin{proof}
This follows immediately from corollary  \ref{c:ordinaryparamsopen} and Lemma \ref{l:ordinaryparamsdense}.
\end{proof}

\section{Definitions and examples of functional dimension}

The (full) functional dimension of a parameter $s$ is a nonnegative integer that measures the number of linearly-independent ways the associated neural network map $\rho(\theta)$ can be varied by perturbing $\theta$.  Equivalently, we wish to measure the maximal possible dimension of a submanifold $S$ of $\mathcal{P}_{n_0,\ldots,n_m}$ passing through $\theta$ such that the restriction of $\rho$ to $S$ is locally injective at $s$.  

An easy way to see that (assuming the network has more than one layer) this number is always less than the dimension of $\mathcal{P}_{n_0,\dots,n_m}$ is the well-known (cf. \cite{RolnickKording}) fact that $\rho$ has a scaling invariance -- meaning that multiplying all parameters of a neuron by a nonzero constant $c$ while simultaneously multiplying all the weights of the next layer map associated to that neuron by $1/c$ does not change the function.  Consequently, there is an (at least) $1$-dimensional set $V \subset \mathcal{P}_{n_0,\ldots,n_m}$ passing through $\theta$ consisting of parameters which all determine the function $\rho(\theta)$.  

To assess how perturbing $\theta$, say changing  the $i$th coordinate of $\theta$, impacts the function $\rho(\theta)$, we will consider, for any point $z \in \mathbb{R}^{n_0}$, the directional derivative $D_{\theta_i}\rho(\theta)(z) = \frac{\partial}{\partial \theta_i} \rho(\theta)(z)$.  For example, if $v$ is a vector that is tangent at $\theta$ to the submanifold $V$ consisting of scaling invariant parameters, then the directional derivative $D_v \rho(\theta)(z) = \vec{0}$ for every point $z \in \mathbb{R}^{n_0}$.  In general, given an arbitrary perturbation direction $v \in \mathcal{P}_{n_0,\ldots,n_m}$, the value of  $D_v \rho(\theta)(z)$ may depend on the point $z$ -- so we may wish to consider the set of derivatives $\{D_v \rho(\theta)(z_j)\}_{j=1}^N$ for some collection of points $z_1,\ldots,z_N$ in $\mathbb{R}^{n_0}$.  The notions of stochastic, batch, and (full) functional dimension defined below correspond to using a single point $z$, a finite set of points $z_i$, or the supremum of all such finite sets to measure the number of linearly independent ways $\rho(\theta)$ can be changed by perturbing $\theta$.

\subsection{Stochastic functional dimension}
Fix an architecture $(n_0,\ldots,n_m)$ and a point $z \in \mathbb{R}^{n_0}$ in the domain.  
Let $\rho_i$ denote the $i$th coordinate function of the output of the unmarked realization map $\rho$, and define the \emph{evaluation map} at $z$, \[E_z : \mathcal{P}_{n_0,\ldots,n_m} \to \mathbb{R}^{n_m},\] by 
\begin{equation} \label{d:evaluationMapPtDef}
E_z(\theta) = \left (\rho_1(\theta)(z),\ldots, \rho_{n_m}(\theta)(z)\right).
\end{equation}
Informally, $E_z(\theta)$ records the coordinates of $\rho(\theta)$ when evaluated at $z$. Note that $z$ is fixed in advance, and $E_z$ is viewed as a function from parameter space to output space.

Now suppose $\theta \in \mathcal{P}_{n_0,\ldots,n_m}$ is an ordinary parameter, and choose $z \in \mathbb{R}^{n_0}$ to be a parametrically smooth point for $\theta$. Then the Jacobian matrix of $E_z$ evaluated at $\theta$, \[{\bf J}E_z \vert_\theta = \left. \left[\frac{\partial(E_z)_i}{\partial{\theta_j}}\right] \right \vert_\theta,\] is the $n_m \times D$ matrix whose entry in the $i$th row and $j$th column records the partial derivative of the $i$th coordinate of $E_z$ at $\theta$ with respect to the $j$th parameter ($j$th coordinate of $\theta$). 

\begin{definition} \label{d:localfundim}  Let $\theta \in \mathcal{P}_{n_0,\ldots,n_m}$ be an ordinary parameter and let $z \in \mathbb{R}^{n_0}$ be a parametrically smooth point for $\theta$.  The {\em stochastic functional dimension} of the parameter $\theta$ for a point $z$ is the rank of $E_z$ (equivalently of ${\bf J}E_z$) at $\theta$:
$$\textrm{dim}_{st.fun}(\theta,z) \coloneqq \textrm{rank }{\bf J}E_z \vert_\theta.$$
\end{definition}

\begin{remark}
Since the rank of a matrix is bounded above by the number of columns or rows, stochastic functional dimension is bounded above by $n_m$. \end{remark}

\begin{example} For the architecture $(1,2)$, the parameterized family is the map $\mathcal{F}_{1,2}:\mathcal{P}_{1,2} \times \mathbb{R}^1 \to \mathbb{R}^2$ given by 
$$\mathcal{F}_{1,2} ((a,b,c,d),z) = \left( \sigma(ax+b), \sigma(cx+d) \right).$$ We consider $\theta = (1,0,1,-1)$ and various values of $z$.  

\emph{Case $z < 0$:} There is a neighborhood $U \subset \mathcal{P}$ of $\theta$ on which $\rho(u,z) = (0,0)$ for all $u \in U$.  Hence, 
$$\textrm{dim}_{st.fun}(\theta,z) = \textrm{rank} \left( \begin{matrix} 0 & 0 & 0 & 0 \\ 0 & 0 & 0 & 0 \end{matrix} \right) = 0.$$

\emph{Case $0 < z < 1$:} There is a neighborhood $U \subset \mathcal{P}$ of $\theta$ on which 
$$\rho((a_u,b_u,c_u,d_u),z) = (a_uz+b_u,0)$$ for all $u=(a_u,b_u,c_u,d_u) \in U$. 
Hence, 
$$\textrm{dim}_{st.fun}(\theta,z) = \textrm{rank} \left( \begin{matrix} z & 1 & 0 & 0 \\ 0 & 0 & 0 & 0 \end{matrix} \right) = 1.$$

\emph{Case $1 < z$:}  There is a neighborhood $U \subset \mathcal{P}$ of $\theta$ on which 
$$\rho((a_u,b_u,c_u,d_u),z) = (a_uz+b_u,c_uz+d_u)$$ for all $u=(a_u,b_u,c_u,d_u) \in U$. 
Hence, 
$$\textrm{dim}_{st.fun}(\theta,z) = \textrm{rank} \left( \begin{matrix} z & 1 & 0 & 0 \\ 0 & 0 & z & 1 \end{matrix} \right) = 2.$$
\end{example}

\begin{remark} \label{r:sfdintepretation}
 In stochastic gradient descent (SGD), the data set is partitioned randomly into smaller subsets, called mini-batches, and each calculation of the gradient of the loss function is computed utilizing the data in a mini-batch. In the extreme case that a mini-batch consists of a single sample point $z$, the corresponding gradient update in parameter space is limited to the tangent directions that change the associated function at $z$.  Informally, the {\em stochastic functional dimension of $s$ at $z$} therefore measures the number of linearly independent directions in parameter space that can be chosen to perturb $\rho(\theta)$ in order to impact its value at $z$. \end{remark}
 
\begin{lemma} Let $z \in \mathbb{R}^{n_0}$, and suppose $\theta \in \mathcal{P}_{n_0,\ldots,n_m} \cong \mathbb{R}^D$ is an ordinary parameter for which $k^{(i)}$ neurons in the $i$th layer are switched off at $z$, and $z$ is parametrically smooth for $\theta$. Then,  setting $n_i' := n_i - k^{(i)}$, the stochastic functional dimension of $\rho(\theta)$ at $z$ is at most
 \[D(n_0',n_1', \ldots, n_m') = \sum_{i=1}^{m} (n_{i-1}'+1)n_{i}'.\]
\end{lemma}

\begin{proof}  Informally, when viewed from $z$, the neural network has architecture $(n_0', n_1', \ldots, n_{m}')$ because of the neurons switched off at $z$. 

More formally, since $s$ is ordinary and $z$ is parametrically smooth for $s$, an application of Lemma \ref{lem:neuronoff} to the $k^{(i)}$ neurons switched off in the $i$th layer at $z$ tells us that the partial derivatives of the function $F = \rho(s)$ with respect to the parameters in the corresponding $k^{(i)}$ rows of $\widehat{A}^{\ell}$ and $k^{(i)}$ columns of $\widehat{A}^{\ell + 1}$ are $0$. This is true for each layer $i$. It follows that ${\bf J}E_z (s)$ has at most \[D(n_0',n_1', \ldots, n_m') = \sum_{i=0}^{m-1} (n_i'+1)n_{i+1}'\] nonzero columns.

\end{proof}

\subsection{Batch functional dimension}
\begin{definition} \label{def:evaluationmap}
Let $Z = \{z_1, \ldots, z_k\}$ be an ordered set of $k < \infty$ points in $\mathbb{R}^{n_0}$ and define the \emph{evaluation map} at $Z$ to be the map  $E_Z:  \mathcal{P}_{n_0,\ldots,n_m} \to \mathbb{R}^{k \cdot n_m}$ by:
\begin{equation} \label{d:evaluationMapDef}
E_Z(\theta) = \left (\rho_1(\theta)(z_1),\ldots, \rho_{n_m}(\theta)(z_1),  \ldots ,  \rho_1(\theta)(z_k),\ldots, \rho_{n_m}(\theta)(z_k)\right).
\end{equation}
\end{definition}

\begin{remark}
It is standard in the machine learning literature to organize a set, $Z=\{z_1, \ldots, z_k\}$ of $k$ fixed data points in $\mathbb{R}^{n_0}$ as the rows of a design matrix \[M_Z = \left[\begin{array}{c}z_1\\ \vdots \\ z_k\end{array}\right].\] For any given parameter $\theta \in \mathbb{R}^{D}$, the image of $M_Z$ under $\rho(\theta)$ is then most naturally a $k \times n_m$ matrix \[\rho(\theta)\left(M_Z\right) := \left[\begin{array}{c} \rho(\theta)(z_1)\\ \vdots \\ \rho(\theta)(z_k)\end{array}\right].\] $E_Z(\theta)$ is the associated unrolled vector in $\mathbb{R}^{k\cdot n_m}$.
\end{remark}

\begin{definition} \label{def:batchfundim} Let $\theta \in \mathcal{P}_{n_0,\ldots,n_m}$ be an ordinary parameter.   The {\em batch functional dimension} of $\theta$ for a batch $Z \subset \mathbb{R}^{n_0}$ of parametrically smooth points for $\theta$ is 
$$\textrm{dim}_{\textrm{ba.fun}} \coloneqq \textrm{rank }{\bf J}E_Z \vert_\theta.$$
\end{definition}

\begin{remark}
If $Z=\{z_1,\ldots,z_k\}$ is a set of $k$ parametrically smooth points for $\theta \in \mathcal{P}_{n_0,\ldots,n_m}$, then $ {\bf J}E_Z(\theta)$ is a  $kn_m \times D_{n_0,\ldots,n_m}$ matrix, formed by stacking the $k$  different $n_m \times D$ matrices $JE_{z_i}(\theta)$: 
\[{\bf J}E_Z(\theta) = \begin{bmatrix}
{\bf J}E_{z_1}(\theta)\\
\vdots\\
{\bf J}E_{z_k}(\theta)\end{bmatrix}\]

\end{remark}

\begin{remark} The batch functional dimension of $\theta$ may be thought of as the number of linearly independent directions in parameter space in which we can perturb $\rho(\theta)$ in order to impact the value of the function $F$ on the set (batch) $Z$  (compare Remark \ref{r:sfdintepretation}).
\end{remark}

 \begin{example}Consider the architecture $(1,1)$. $\mathcal{P}_{1,1} = \mathbb{R}^2$.  For any parameter  $\theta = (w^{1}, b^{1}) \in \mathcal{P}_{1,1}$, $\rho(\theta) :\mathbb{R}^1 \to \mathbb{R}^1$ is defined by 
 \begin{eqnarray*} x &\mapsto& \sigma(w^{1}x+b^{1}).\end{eqnarray*} 
 Let $Z = (z_1,\ldots,z_k)$ be a batch of $k$ fixed points in $\mathbb{R}^1$.  Then, for arbitrary $\theta = (w^{1},b^{1}) \in \mathcal{P}_{1,1}$ and $Z$ parametrically smooth (which ensures partial derivatives exist), we have: 
 
\[{\bf J}E_Z(\theta) = \begin{bmatrix}
{\bf J}E_{z_1}(\theta)\\
\vdots\\
{\bf J}E_{z_k}(\theta)\end{bmatrix}\]
 where ${\bf J}E_{z_i}(\theta) =  [z_i  \,\,\, 1]$ if the ternary label $s_{z_i}^{1}$ at $z_i$ is $1$ (i.e., if the lone neuron in layer $1$ is switched on at $z_i$) and $=[0 \,\,\, 0]$ if the ternary label $s_{z_i}^{1}$ is $-1$ or $0$ (i.e., if the lone neuron in layer $1$ is switched off at $z_i$).

The rank of $\boldsymbol{J}E_Z(\theta)$ can therefore assume the values $0$, $1$, or $2$, depending on the set $Z$ and parameters $(w^{1}, b^{1})$. Indeed, the rank of $\boldsymbol{J}E_Z(\theta)$: 
\[
\begin{cases}
	=0 & \mbox{ if the ternary label $s^{1}_{z_i}$ associated to $s$ is $\leq 0$ for every points $z_i$ in $Z$},\\
	=1 & \mbox{ if $s^{1}_{z_i} \leq 0$ for all but one of the points $z_i$ in $Z$, and}\\ 
	=2 & \mbox{ if $s^{1}_{z_i} = 1$ for at least two distinct points $z_i$ of $Z$.}
\end{cases}
\]
 \end{example}

 A natural question is how the selection of the points in the batch $Z$ affects the batch functional dimension. We give a partial answer to this question in Proposition \ref{p:reformulateoriginaldef}. For most parameters $\theta$, if $Z$ contains $n_0+1$ points arranged in a geometric simplex in a top-dimensional cell of $\mathcal{C}(\theta)$, adding more points to $Z$ in that cell will not increase the batch functional dimension.

\subsection{Functional dimension} \label{ss:globalfunctionaldimension}

We are now ready to define the (full) functional dimension of an ordinary parameter $\theta \in \mathcal{P}_{n_0,\ldots,n_m}$.

 \begin{definition}\label{def:functionaldimension} Fix an architecture $(n_0,\ldots,n_m)$. For any ordinary parameter $\theta \in \mathcal{P}_{n_0,\ldots,n_m}$, define the (full) \emph{functional dimension} at $\theta$ to be

 $$\textrm{dim}_{\textrm{fun}}(\theta) := \sup_{Z \subset \mathbb{R}^{n_0} \textrm{ is finite and parametrically smooth for }\theta} \textrm{rank } \boldsymbol{J}E_Z\vert_\theta.$$
 \end{definition}

\begin{remark} The assumption in Definition \ref{def:functionaldimension} that $\theta$ is ordinary ensures that the supremum is taken over a nonempty set. 
\end{remark}

\begin{remark}
Since the matrix  $\boldsymbol{J}E_Z\vert_\theta$ has $D(n_0,\ldots,n_m)$ columns, and the rank of a matrix is bounded above by its number of columns, it would be equivalent to take the supremum in Definition \ref{def:functionaldimension} over only subsets $Z \subset \mathbb{R}^{n_0}$  that are parametrically smooth for $\theta$ and satisfy $|Z| \leq D(n_0,\ldots,n_m)$.    For the same reason, $ \textrm{dim}_{\textrm{fun}}(\theta) \leq D(n_0,\ldots,n_m)$ for all $\theta \in \mathcal{P}_{n_0,\ldots,n_m}$. 
\end{remark}

\begin{definition} 
Fix an architecture $(n_0,\ldots,n_m)$. 
Define the \emph{functional dimension of the parameter space} $\mathcal{P}_{n_0,\ldots,n_m}$ to be 
$$\textrm{dim}_{\textrm{fun}} (\mathcal{P}_{n_0,\ldots,n_m}) : = \sup_{\textrm{ordinary } \theta \in \mathcal{P}_{n_0,\ldots,n_m}} \textrm{dim}_{\textrm{fun}}(\theta).$$
\end{definition}

\begin{remark}
It is natural to ask which sets $Z \subset \mathbb{R}^{n_0}$ attain the functional dimension.  For a generic, transversal, combinatorially stable parameter $\theta$ (Definition \ref{def:combinatoriallystable}), Proposition  \ref{p:reformulateoriginaldef} shows that functional dimension is achieved on decisive sets (Definition \ref{d:decisive}).  A decisive set for a parameter $\theta$ consists of $k(n_0+1)$ points, where $k$ is the number of top-dimensional cells of the canonical polyhedral complex $\mathcal{C}(\theta)$.  A full measure set of parameters is generic and transversal, and we conjecture that combinatorially stable is also a full measure condition. 
\end{remark}

 \subsection{An illustrative example of functional dimension} \label{s:illustrativeexample}
 We give an example in which the dimension of the parameter space $\mathcal{P}$ is $7$, but the functional dimension at a specific point in $\mathcal{P}$ is $5$, which Theorem \ref{t:upperbound} guarantees is the maximum possible functional dimension.

Consider the architecture $(1,2,1)$.  If $\theta \in \mathcal{P}_{1,2,1}$ has coordinates $\theta=(w^{1}_{11},b^{1}_{1},w^{1}_{21},b^{1}_{2},w^{2}_{11},w^{2}_{12},b^{2}_1)$, then the associated matrices for the affine linear transformations are

  $$A^{1} = \begin{bmatrix} w^{1}_{11} & b^{1}_{1} \\ w^{1}_{21} & b^{1}_{2} \end{bmatrix}, \quad A^{2} = \begin{bmatrix} w^{2}_{11} & w^{2}_{12} & b^{2}_1 \end{bmatrix},$$ and $\bar{\rho}(\theta):\mathbb{R} \to \mathbb{R}$ is the function 
 \begin{eqnarray*}
 x &\mapsto & \sigma(w^{2}_{11}\sigma(w^{1}_{11}x+b^{1}_1) + w^{2}_{12}\sigma(w^{1}_{21}x+b^{1}_{2})+b^{2}_1)\\
 &=& A^{2}_{s^{2}_x}\widehat{A^{1}_{s^{1}_x}}\widehat{x}
 \end{eqnarray*}

Let $Z=(z_1,\ldots,z_k)$ be an ordered list of $k < \infty$ points in $\mathbb{R}^{1}$.   Then the general form of $\boldsymbol{J}E_Z$ is

   $$ \boldsymbol{J}E_Z(s) = \begin{bmatrix}   
 \frac{\partial}{\partial w_1^{1}}  \left(A^{2}_{s^{2}_{z_1}}\widehat{A^{1}_{s^{1}_{z_1}}}(\widehat{z}_1)\right) & \dots &   \frac{\partial}{\partial b^{2}_1}  \left(A^{2}_{s^{2}_{z_1}}\widehat{A^{1}_{s^{1}_{z_1}}}(\widehat{z}_1)\right)\\ 
  \vdots & \ddots &  \vdots \\
 \frac{\partial}{\partial w_1^{1}}  \left(A^{2}_{s^{2}_{z_k}}\widehat{A^{1}_{s^{1}_{z_k}}}(\widehat{z}_k)\right) & \dots & 
 \frac{\partial}{\partial b_1^{2}}  \left(A^{2}_{s^{2}_{z_k}}\widehat{A^{1}_{s^{1}_{z_k}}}(\widehat{z}_k)\right) 
 \end{bmatrix}$$
 
Consider the specific parameter point $\theta_0 = (2,-5,-1,4,1,1,1 )$. 
\begin{figure}
  \centering
    \includegraphics[width=0.5\textwidth]{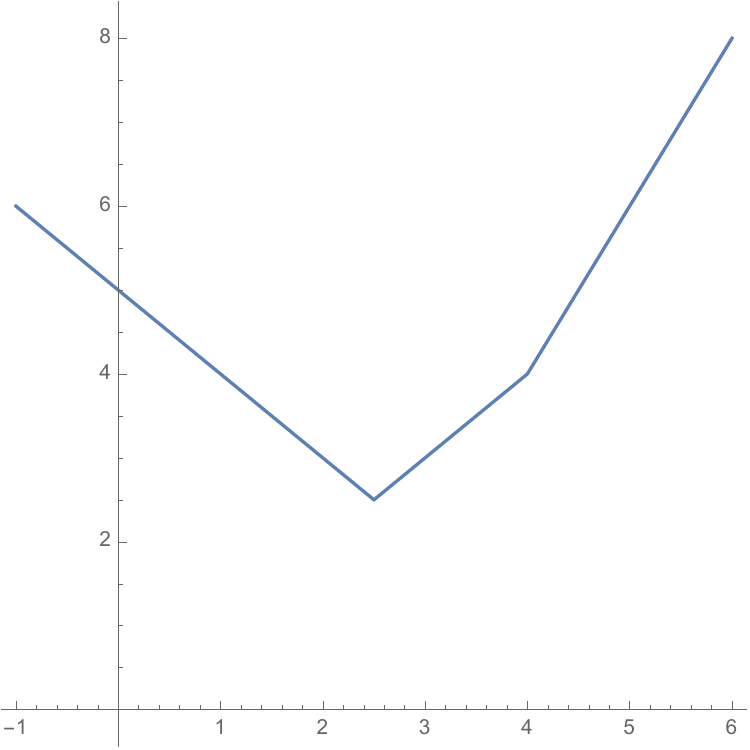}
      \caption{For $\theta_0 = (2,-5,-1,4,1,1,1)$, the function $\rho(\theta_0)$ has the form $\rho(\theta_0)(x) = \sigma(\sigma(2x-5) + \sigma(-x+4)+1)$.}
      \label{f:illustrativeexample}
\end{figure}
The function $\bar{\rho}(\theta_0)$ is piecewise linear, with $3$ pieces, and the points of nonlinearity of $\bar{\rho}(\theta_0)$ are $x=2.5$ and $x=4$.  (See Figure \ref{f:illustrativeexample}.)  The functions that describe the 3 pieces of $\bar{\rho}(\theta_0)$ are as follows:

\begin{itemize}
\item Left piece: 
		\begin{eqnarray*} 
			x &\mapsto&  A^{2}_{s^{2}_x}\widehat{A^{1}_{s^{1}_x}}\widehat{x}\\
			&=& \left[\begin{array}{ccc} w^{2}_{11} & w^{2}_{12} & b^{2}_1\end{array}\right]\left[\begin{array}{cc} 0 & 0\\w^{1}_{21} & b^{1}_2\\0 & 1\end{array}\right]\left[\begin{array}{c}x\\1\end{array}\right]\\
			&=& w^{2}_{12}w^{1}_{21}x+w^{2}_{12}b^{1}_2+b^{2}_1
		\end{eqnarray*}
\item Middle piece: 
		\begin{eqnarray*} 
			x &\mapsto&  \left[\begin{array}{ccc} w^{2}_{11} & w^{2}_{12} & b^{2}_1\end{array}\right]\left[\begin{array}{cc} w^{1}_{11} & b^{1}_1\\w^{1}_{21} & b^{1}_2\\0 & 1\end{array}\right]\left[\begin{array}{c}x\\1\end{array}\right]\\
			&=& w^{2}_{11}w^{1}_{11}x+w^{2}_{11}b^{1}_1+w^{2}_{12}w^{1}_{21}x+w^{2}_{12}b^{1}_2+b^{2}_1
		\end{eqnarray*}

\item Right piece: 
		\begin{eqnarray*} 
			x &\mapsto&  \left[\begin{array}{ccc} w^{2}_{11} & w^{2}_{12} & b^{2}_1\end{array}\right]\left[\begin{array}{cc} w^{1}_{11} & b^{1}_1\\0 & 0\\0 & 1\end{array}\right]\left[\begin{array}{c}x\\1\end{array}\right]\\
			&=& w^{2}_{11}w^{1}_{11}x+w^{2}_{11}b^{1}_1+b^{2}_1
		\end{eqnarray*}
\end{itemize}

For reference, here are the ternary labelings on the $1$--cells associated to the three pieces:

\begin{itemize}
	\item Left piece: $s = (s^{1},s^{2}) = ((-1,1),(1))$,
	\item Middle piece: $s = ((1,1),(1))$,
	\item Right piece: $s = ((1,-1),(1))$
\end{itemize}

 Thus, for a generic point $x$ in each of the three pieces, the $1$-row matrix $\boldsymbol{J}E_{\{x\}}$ has the forms:

 \begin{align*}
\left[\begin{array}{ccccccc} 0 & 0 & w^{2}_{12}x & w^{2}_{12} & 0 & w_{21}^{1}x+b_2^{1} & 1\end{array}\right]  & \qquad \textrm{ if } x \textrm{ is in the left piece,} \\
\left[\begin{array}{ccccccc} w^{2}_{11}x & w^{2}_{11} & w^{2}_{12}x & w^{2}_{12} & w^{1}_{11}x+b^{1}_1 & w^{1}_{21}x+b^{1}_2 & 1\end{array}\right] &  \qquad \textrm{ if } x \textrm{ is in the middle piece,} \\
\left[\begin{array}{ccccccc} w^{2}_{11}x & w^{2}_{11} & 0 & 0 & w^{1}_{11}x+b^{1}_1 & 0 & 1\end{array}\right]& \qquad  \textrm{ if } x \textrm{ is in the right piece}. \\
 \end{align*}

  Now suppose $Z$ is the set consisting of parametrically smooth (for $\theta_0$) points $x_1,\ldots,x_\ell$ in the domain of the left piece, $y_1,\ldots,y_m$ in the  domain of the middle piece, and $z_1,\ldots,z_n$ in the domain of the right piece.  Then $\textrm{rank } \boldsymbol{J}E_Z \vert_{\theta_0} = $

  $$\textrm{rank}\begin{bmatrix}
 0 & 0 & w^{2}_{12}x_1 &w^{2}_{12}&  0 & w^{1}_{21}x_1+b^{1}_2 & 1 \\
 \vdots & & & & & & \vdots \\
  0 & 0 & w^{2}_{12}x_\ell &w^{2}_{12}&  0 & w^{1}_{21}x_\ell+b^{1}_2 & 1 \\
  w^{2}_{11}y_1 &  w^{2}_{11} & w^{2}_{12}y_1 & w^{2}_{12} & w^{1}_{11}y_1+b^{1}_1 & w^{1}_{21}y_1 + b^{1}_2 & 1 \\
   \vdots & & & & & & \vdots \\ 
     w^{2}_{11}y_m & w^{2}_{11} & w^{2}_{12}y_m & w^{2}_{12} & w^{1}_{11}y_m+b^{1}_1 & w^{1}_{21}y_m + b^{1}_2& 1 \\
     w^{2}_{11}z_1 & w^{2}_{11} & 0 & 0 & w^{1}_{11}z_1+b^{1}_1 & 0 & 1\\
       \vdots & & & & & & \vdots \\      
        w^{2}_{11}z_n & w^{2}_{11} & 0 & 0 & w^{1}_{11}z_n+b^{1}_1 & 0 & 1\\
  \end{bmatrix}$$

We know from linear algebra that rank = row rank = column rank, and rank is invariant under elementary row operations and elementary column operations.  Applying some elementary column operations to the matrix above results in 
 
  $$\textrm{rank}\begin{bmatrix}
 0 & 0 & x_1 &1&  0 & 0 & 1 \\
 \vdots & & & & & & \vdots \\
  0 & 0 & x_\ell &1&  0 & 0 & 1 \\
  y_1 & 1 & y_1 & 1& 0 &0& 1 \\
   \vdots & & & & & & \vdots \\ 
     y_m & 1 & y_m & 1 & 0 & 0& 1 \\
     z_1 & 1 & 0 & 0 & 0  & 0 & 1\\
       \vdots & & & & & & \vdots \\      
        z_n & 1 & 0 & 0 &0 & 0 & 1\\
  \end{bmatrix}
  =\textrm{rank}
  \begin{bmatrix}
 0 & 0 & x_1 &1&   1 \\
 \vdots & & & & \vdots \\
  0 & 0 & x_\ell &1&   1 \\
  y_1 & 1 & y_1 & 1&  1 \\
   \vdots & & & &  \vdots \\ 
     y_m & 1 & y_m & 1 & 1 \\
     z_1 & 1 & 0 & 0 &  1\\
       \vdots & & & &  \vdots \\      
        z_n & 1 & 0 & 0  & 1\\
  \end{bmatrix}
  $$
  
Thus, as the reader can verify using the rightmost matrix above, \[\textrm{dim}_{\textrm{fun}}(\theta_0) = 5.\]

 \bigskip 
The parameter $\theta_0$ was chosen to be a combinatorially stable parameter (Definition \ref{def:combinatoriallystable}).  The above result, $\textrm{dim}_{\textrm{fun}}(\theta_0) = 5$, implies that there are exactly $5$ linearly-independent degrees of freedom available for varying parameters near $\theta_0$ while remaining in the class of functions realizable by networks of this architecture (see \S \ref{s:interpretation}). For example, one could perturb the slope of the left piece, the $x$- and $y$-coordinates of the left bend, the $x$-coordinate of the right bend, and the slope of the right piece. 

Note that a quick glance at Figure \ref{f:illustrativeexample} seems to indicate that there are 6 independent degrees of freedom, rather than 5.  The explanation from our analysis is that the slopes of two of the lines determine the slope of the third. 

\section{Combinatorial stability and functional dimension} \label{s:interpretation}

\subsection{Combinatorial stability}

\begin{definition}[Definition 1.26 of \cite{Grunert}]
Two polyhedral complexes $K$ and $L$ are \emph{combinatorially equivalent} if there is a bijection $\phi:K \to L$ between them that is a poset isomorphism with respect to the poset structure given by the face relation and maintains the dimensions of the cells.  
\end{definition}

Recall that the \emph{relative interior} of subset $S$ of $\mathbb{R}^n$ is the interior of $S$ when $S$ is viewed as a subset of the affine hull of $S$. If $S$ is a single point, we interpret its affine hull to be just that point, so that the relative interior of a $0$-cell is itself. 

\begin{definition}
Let $\phi:K \to L$ be a combinatorial equivalence of polyhedral complexes, and let $s_K$ (resp. $s_L$) denote the  ternary labeling functions of points in $K$ (resp. $L$).  We say that $\phi$ \emph{respects the ternary labelings} if
for every cell $C_K \in K$  and corresponding cell $C_L = \phi(C_K) \in L$ and every pair of points $x_K$ in the relative interior of $C_K$ and $x_L$ in the relative interior of $C_L$, $s_K(x_k) = s_L(x_L)$.
\end{definition}

\begin{definition} \label{def:combinatoriallystable}
A parameter $\theta \in \mathcal{P}_{n_0,\ldots,n_m}$ is  \emph{combinatorially stable} if there exists an open neighborhood $U$ of $\theta$ and a continuous map 
$\Phi: U \times \mathbb{R}^{n_0}  \to \mathbb{R}^{n_0}$ such that for every $u \in U$, the restriction $\Phi_u:= \Phi( u,\cdot) : \mathbb{R}^{n_0} \to \mathbb{R}^{n_0}$ is a homeomorphism that induces a combinatorial equivalence between $\mathcal{C}(\theta)$ and $\mathcal{C}(u)$ that respects the ternary labelings. 
\end{definition}

 One might wonder whether combinatorially stable is implied by generic and transversal, or vice versa.  As we show in the remarks below, neither implication is true.
 
 \begin{remark}
Generic and transversal does not imply combinatorially stable.  For a generic and transversal parameter, it is possible to have two bent hyperplanes (coming from different layers) in $\mathbb{R}^{n_0}$ that have parallel, unbounded pieces that do not intersect; then arbitrarily small perturbations of the slopes of these unbounded pieces would cause them to intersect, a change in combinatorial structure. 

For example, consider the marked network map $\mathbb{R}^2 \to \mathbb{R}^2 \to \mathbb{R}^1$ in which the first layer map is $\sigma \circ Id$ and the second layer map is $(x,y) \mapsto \sigma(x-1)$.  The bent hyperplanes from the first layer map are the standard axes in $\mathbb{R}^2$, and the bent hyperplane from the second layer map is the line $x=1$.  If the second layer map was perturbed to become $(x,y) \mapsto \sigma (x-1+\epsilon y)$, for $\epsilon > 0$, the associated bent hyperplane would intersect the $y$-axis at some positive $y$-value.  
 \end{remark}

  \begin{remark}
  Combinatorially stable does not imply generic.  Hyperplanes that do not intersect the positive orthant in hidden layers have no impact on combinatorial stability, and so can form a nongeneric arrangement.

 For example, consider the marked network map $\mathbb{R}^2 \to \mathbb{R}^2 \to \mathbb{R}^2$  in which the first layer map is $\sigma \circ Id$ and the second layer map is $(x,y) \mapsto \sigma((y+x+1, y+x-1))$.   The hyperplane arrangement associated to the second layer map is not generic (since the two hyperplanes are parallel).  However, the parameter is combinatorially stable.  
 
 This illustrates the fact that ``generic'' is a descriptor of parameters or marked network functions, not of unmarked network functions. 
  \end{remark}

\begin{lemma}  \label{l:5.6}
Let $\theta \in \mathcal{P}_{n_0,\ldots,n_m}$ be a parameter that is generic, transversal, and combinatorially stable.  
Let $U$ be an open neighborhood of $\theta$ as in Definition \ref{def:combinatoriallystable}.  Then for every $u \in U$, the union (in $\mathbb{R}^{n_0}$) of the bent hyperplanes associated to $\bar{\rho}(u)$ is the $(n_0-1)$-skeleton of $\mathcal{C}(u)$. 
\end{lemma}

We note that if the set of generic and transversal parameters was proven to be open, the lemma above would follow immediately.  We conjecture this is true, but since what follows does not require it, for the sake of brevity we prove Lemma \ref{l:5.6} directly. 

\begin{proof}
It suffices to prove that no top-dimensional cell of $\mathcal{C}(u)$ is contained in the union of the bent hyperplanes of $u$.  
It is immediate that for every parameter $u \in \mathcal{P}_{n_0,\ldots,n_m}$, the $(n_0-1)$-skeleton of $\mathcal{C}(u)$ is contained in the union of the bent hyperplanes associated to $u$. By \cite{GrigsbyLindsey}, because $\theta$ is generic and transversal parameter, the $(n_0-1)$-skeleton of $\mathcal{C}(\theta)$ is precisely the union of the bent hyperplanes in $\mathbb{R}^{n_0}$ associated to $\theta$, i.e. the set of all points $x \in \mathbb{R}^{n_0}$ such that that ternary coding of $x$ for $\theta$ has at least one $0$ entry.  Consequently, the ternary coding for $\theta$ of any point in the interior of a top-dimensional cell has no $0$ entries.   By assumption, the combinatorial equivalence $\Phi(\cdot, u)$ induces a bijection between the sets of top-dimensional cells of $\mathcal{C}(\theta)$ and $\mathcal{C}(u)$, and this combinatorial equivalence respects the ternary labelings.  Consequently, points in the interior of top-dimensional cells of $\mathcal{C}(u)$ have ternary labelings with respect to $u$ that have no $0$ entries.  Consequently, no top-dimensional cell is contained in the union of the bent hyperplanes of $u$.  
\end{proof}

\subsection{Functional dimension of combinatorially stable parameters}

In this section, we present an alternative characterization of functional dimension for generic, transversal, combinatorially stable parameters.  Namely, for such a parameter $\theta$, we record the value and total derivative of the function at one point in each top-dimensional cell of $\mathcal{C}(\theta)$; functional dimension is the degrees of freedom we have to vary this data (Theorem \ref{t:equivalentdef}). 

This alternative definition is somewhat more intuitive, but has the disadvantage that it is only defined for \emph{combinatorially stable} ordinary parameters.  We conjecture that a full measure (with respect to Lebesgue measure) set of parameters is combinatorially stable, but a proof of this conjecture is beyond the scope of this paper.

\begin{definition} \label{d:decisive} 
 We call a set $Z \subset \mathbb{R}^{n_0}$ \emph{decisive} for $\theta \in \mathcal{P}_{n_0,\ldots,n_m}$ if for every top-dimensional polyhedron $C \in \mathcal{C}(\theta)$, $Z$ contains precisely $n_0+1$ points in the interior of $C$ that form an $n_0$-dimensional geometric simplex (i.e. they determine $n_0$ linearly independent vectors).  
\end{definition}

\begin{remark}
Note that if $\mathcal{C}(\theta)$ has $k$ top-dimensional cells, then a decisive set for $\theta$ consists of $(n_0+1)k$ points in $\mathbb{R}^{n_0}$.  Proposition \ref{p:reformulateoriginaldef} will then imply that for a combinatorially stable parameter $\theta$,  $(n_0+1)k$ is an upper bound on the functional dimension. Since it is difficult in practice to compute the value of $k$, we cannot say exactly how the upper bound, $k(n_0+1)$, on the functional dimension compares to the upper bound in Theorem \ref{t:upperbound}, but we expect that it is typically much worse. The importance of Definition \ref{d:decisive} and Proposition \ref{p:reformulateoriginaldef} is to give a concrete theoretical procedure for computing the full functional dimension of a combinatorially stable parameter.
\end{remark}

Lemma \ref{l:generictransversalordinary}  immediately implies the following result.

\begin{corollary}
If $\theta$ is generic and transversal, a decisive set $Z$ for $\theta$ consists of parametrically smooth points for $\theta$.
 \end{corollary}

 \begin{definition} \label{d:cellularmembership}
  Let $\theta \in \mathcal{P}_{n_0,\ldots,n_m}$ be a combinatorially stable parameter and let $X$ be a set of points in $\mathbb{R}^{n_0}$.  We say that a neighborhood $U \subset \mathcal{P}_{n_0,\ldots,n_m}$ of $\theta$ \emph{preserves cellular membership of points in $X$} if for every point $x \in X$, the set of ternary labelings of all cells of $\mathcal{C}(u)$ that contain $x$  does not depend on $u \in U$.  
 \end{definition}
 
 \begin{lemma} \label{l:preservingcellularnbhd}
  Let $\theta$ be a generic, transversal, and combinatorially stable parameter.  Let $X$ be a finite set of points in $\mathbb{R}^{n_0}$ that is contained in the union of the interiors of the top-dimensional cells of $\mathcal{C}(s)$.  Then there exists an open neighborhood $U \subset \mathcal{P}_{n_0,\ldots,n_m}$ of $\theta$ that preserves cellular membership of points in $X$.  
 \end{lemma}

  \begin{proof}
 The proof of Lemma \ref{l:sufficientconditionnondegenerate} shows that for any point $x$ in the interior of a top-dimensional cell of $\mathcal{C}(\theta)$, there exists a neighborhood $U_2$ of $\theta$ such that the ternary labeling of $x$ with respect to $u \in U_2$ is constant (does not vary with $u$).  
 The result follows immediately.  
 \end{proof}
 
 The following Lemma is an immediate consequence of Lemma \ref{l:preservingcellularnbhd}.
 
 \begin{lemma} \label{l:decisivestable}
 Let $\theta$ be a generic, transversal, and combinatorially stable parameter.  Let $Z$ be a decisive set for $\theta$.  Then there exists an open neighborhood $V \subset \mathcal{P}$ of $\theta$ such that 
 \begin{enumerate}
 \item $V$ preserves cellular membership of points in $Z$,  
 \item $Z$ is a decisive set for $v$ for all $v \in V$.  
 \end{enumerate}
 \end{lemma}
 
\begin{definition} \label{d:simplexevaluationmap}
 Let $\theta \in \mathcal{P}_{n_0,\ldots,n_m}$ be a generic, transversal, and combinatorially stable parameter.
Let $Z$ be a decisive set for $\theta$ and let $V$ be as in Lemma \ref{l:decisivestable}. 
We define the \emph{simplex evaluation map} 
$$SE_{\theta,Z} : V \to \mathbb{R}^{k(n_0+1)n_m}$$
by $$SE_{\theta,Z}(v) = E_Z(v).$$
\end{definition} 

\begin{remark}
The map $E_Z$ is as defined in Definition \ref{def:evaluationmap}, but we note that the $k$ has a different meaning there.  
\end{remark}

A parametrically smooth point $x$ for a parameter $\theta$ does not necessarily lie in the interior of a top-dimensional cell of $\mathcal{C}(\theta)$, even in the case that $s$ is generic and transversal.  This is related to the fact that the locus of nondifferentiability of $\rho(\theta)$ may be a proper subset of the bent hyperplane arrangement.  See the proof of Corollary \ref{cor:NonDiffInBHA} for an example.  The following Lemma will be used in the proof of Proposition \ref{p:reformulateoriginaldef}.

\begin{lemma} \label{l:dimattainedbyinteriorpoints}
Let $\theta \in \mathcal{P}_{n_0,\ldots,n_m}$ be an ordinary parameter. Then functional dimension for $\theta$ can be realized on a set $Z$ of parametrically smooth points contained in the union of the interiors of the top-dimensional cells of $\mathcal{C}(\theta)$.  
\end{lemma}

\begin{proof}
Suppose $dim_{fun}(\theta)$ is attained on a set $Z\subset \mathbb{R}^{n_0}$ of parametrically smooth points for $\theta$ that contains a point $z$ that is in the $(n_0-1)$-skeleton of $\mathcal{C}(\theta)$.  We will argue that we can replace $z$ in $Z$ with a point $z'$ that is not in the $(n_0-1)$-skeleton of $\mathcal{C}(\theta)$, without changing the rank of $\mathbf{J}E_Z(\theta)$.

The restriction of $\bar{\rho}(\theta)$ to any cell of $\mathcal{C}(\theta)$ is an algebraic expression in the variables $\theta_1,\ldots,\theta_D$ and $x_1,\ldots,x_{n_0}$.
The $D$ entries of the row of $\mathbf{J}E_Z(s)$ corresponding to the point $z$ are the partial derivatives of this algebraic equation evaluated at $x=z$ with respect to each of the parameters $\theta_1,\ldots,\theta_D$.  Consequently, the entries of this row are also algebraic expressions in the variables $\theta_1,\ldots,\theta_D$ and $z_1,\ldots,z_{n_0}$.  If this row is not linearly independent from the other rows of the matrix, then (since $Z$ maximizes the possible rank of this matrix), we may freely perturb $z$ without changing the rank of the matrix.  Thus, we may assume without loss of generality that this row is linearly independent from the other rows.  Note that a point in Euclidean space that is not contained in a proper linear subspace is a positive distance from the subspace; consequently, sufficiently small perturbations of the point will remain outside of the linear subspace.  In our situation, this means that any sufficiently small perturbation of the row associated to the point $z$ will remain linearly independent from the other row vectors.  Consequently, if we replace $z$ with $z'=z+\epsilon$ for any sufficiently small perturbation $\epsilon$, the rank of the matrix will not change.  

\end{proof}

Since the supremum in the definition of (full) functional definition is not ideal from the point of view of computation, it is desirable to characterize the circumstances under which functional dimension equals batch functional dimension for some chosen batch.  Proposition \ref{p:reformulateoriginaldef} guarantees that for generic, transversal, combinatorially stable parameters, functional dimension equals batch functional dimension for any batch $Z$ that contains a decisive set.  The idea behind the proof is that the gradient at any point $x$ in a top-dimensional cell of the canonical polyhedral complex will always be in the linear span of the gradients at points defining a top-dimensional simplex in that cell.

\begin{proposition} \label{p:reformulateoriginaldef}
For a generic, transversal, combinatorially stable parameter, functional dimension is attained on any decisive set.  

That is, letting all notation be as in Definition \ref{d:simplexevaluationmap}, the functional dimension of $\theta$ is the rank of the map $SE_{\theta,Z}: V \to \textrm{SE}_{\theta,Z}(V)$ at the point $\theta$, i.e. 
$$\textrm{dim}_{\textrm{fun}}(\theta) = \textrm{rank }\textrm{SE}_{\theta,Z}\vert_\theta.$$
\end{proposition}

\begin{proof}
By Lemma \ref{l:dimattainedbyinteriorpoints} $\dim_{fun}(\theta)$ is attained on some set $Z'$ such that every point of $Z'$ lies in the interior of a top-dimensional cell $C$ of the canonical polyhedral complex $\mathcal{C}(\theta)$.  Thus, the set $Z^* \coloneqq Z \cup Z'$
\begin{enumerate}
\item contains the decisive set $Z$,
\item is a set on which $\textrm{dim}_{fun}(\theta)$ is achieved, 
\item is contained in the union of the interiors of the top-dimensional cells of $\mathcal{C}(\theta)$.  
\end{enumerate}
Without loss of generality (by decreasing the size of $V$ and appealing to Lemma \ref{l:preservingcellularnbhd}), we may assume that $V$ preserves cellular membership of points of $Z^*$.  Thus $$\textrm{dim}_{\textrm{fun}}(\theta) = \textrm{rank } \mathbf{J}E_{Z^*}\vert_\theta = \textrm{rank }\textrm{SE}_{\theta,Z^*} \vert_\theta.$$

First, consider the case that $n_m = 1$. We will show that for $z \in Z^* \setminus Z$ in cell $C \in \mathcal{C}(\theta)$, the row vector  $\mathbf{J}E_{z^*}(\theta)$ is in the linear span of the set of row vectors $\{\mathbf{J}E_{z^1}(\theta), \ldots, \mathbf{J}E_{z^{n_0 +1}}(\theta)\}$, where $z^1,\ldots,z^{n_0+1}$ are the elements of the decisive set $Z$ that are in $C$.   
Denote the coordinates of the parameter $\theta\in \mathcal{P}_{n_0,\ldots,n_m=1}$ by  $\theta=(\theta_1,\ldots,\theta_D)$ and the coordinates of an arbitrary point $x^i \in \mathbb{R}^{n_0}$ by $x=(x_1^i,\ldots,x_{n_0}^i)$.  
 For $x \in C$, $\rho(\theta)(x)$ can be written as
\begin{equation}\label{eq:algform} 
\rho(\theta)(x) = S(\theta) + \sum_{i=1}^{n_0} \theta_1^{\tau_{i,1}}\theta_2^{\tau_{i,2}}\dots \theta_D^{\tau_{i,D}} \tau_ix_i,
\end{equation}
where $S(\theta)$ is an algebraic expression in the variables $\theta_1,\ldots,\theta_D$ that does not depend on any $x_1,\ldots,x_{n_0}$ and $\tau_i,\tau_{i,j} \in \{0,1\}$ for all $i,j$.   
Then, for $x \in C$, we have
\begin{multline*} \mathbf{J}E_x(\theta)= \left (\frac{\partial}{\partial \theta_1} \rho(\theta)(x),\ldots ,\frac{\partial}{\partial \theta_D} \rho(\theta)(x) \right)
 \\
 = \left (\frac{\partial}{\partial \theta_1}S(\theta) + \sum_{i=1}^{n_0}x_i ( \frac{\partial}{\partial \theta_1}\theta_1^{\tau_{i,1}}\theta_2^{\tau_{i,2}}\dots \theta_D^{\tau_{i,D}} \tau_i ) , \ldots, \right. \\ 
 \left. \frac{\partial}{\partial \theta_D}S(\theta) + \sum_{i=1}^{n_0}x_i ( \frac{\partial}{\partial \theta_D}\theta_1^{\tau_{i,1}}\theta_2^{\tau_{i,2}}\dots \theta_D^{\tau_{i,D}} \tau_i ) \right )
\end{multline*}
Consequently, for any points $x,y \in C$, we have
  \begin{multline} \label{eq:Jdifference}
  \mathbf{J}E_x(\theta) - \mathbf{J}E_{z_1}(\theta) \\
  =  \left (\sum_{i=1}^{n_0}(x_i-y_i) ( \frac{\partial}{\partial \theta_1}\theta_1^{\tau_{i,1}}\theta_2^{\tau_{i,2}}\dots \theta_D^{\tau_{i,D}} \tau_i ) , \ldots, \sum_{i=1}^{n_0}(x_i-y_i) ( \frac{\partial}{\partial \theta_D}\theta_1^{\tau_{i,1}}\theta_2^{\tau_{i,2}}\dots \theta_D^{\tau_{i,D}} \tau_i ) \right ) 
  \end{multline}

So consider any point  $z^* \in Z^* \setminus Z$ in $C$.   Because the points $ \{z^1,\ldots,z^{n_0+1}\}$ are the vertices of a top-dimensional simplex,  there exist real numbers $c^i$ such that 
\[z^* = z^1 + \sum_{i=2}^{n_0+1}c^i(z^i-z^1).\] 
Then applying \eqref{eq:Jdifference} yields 

  \begin{multline}  \label{eq:linearsum}
  \mathbf{J}E_{z^*}(\theta) - \mathbf{J}E_{z^1}(\theta)\\
  =  \left (\sum_{i=1}^{n_0}( {\textstyle \sum_{j=2}^{n_0+1}}c^j(z^j_i-z^1_i)  \left( \frac{\partial}{\partial \theta_1}\theta_1^{\tau_{i,1}}\theta_2^{\tau_{i,2}}\dots \theta_D^{\tau_{i,D}} \tau_i \right)  \right. ,
   \ldots, \\
\left. \sum_{i=1}^{n_0}( {\textstyle \sum_{j=2}^{n_0+1}}c^j(z^j_i-z^1_i)  \left( \frac{\partial}{\partial \theta_D}\theta_1^{\tau_{i,1}}\theta_2^{\tau_{i,2}}\dots \theta_D^{\tau_{i,D}} \tau_i \right) \right) \\
= \sum_{j=2}^{n_0+1} \left (\sum_{i=1}^{n_0} c^j(z^j_i-z^1_i)  \left( \frac{\partial}{\partial \theta_1}\theta_1^{\tau_{i,1}}\theta_2^{\tau_{i,2}}\dots \theta_D^{\tau_{i,D}} \tau_i \right)  \right. ,
   \ldots, \\ 
\left. \sum_{i=1}^{n_0}c^j(z^j_i-z^1_i)  \left( \frac{\partial}{\partial \theta_D}\theta_1^{\tau_{i,1}}\theta_2^{\tau_{i,2}}\dots v_D^{\tau_{i,D}} \tau_i \right) \right) \\
= \sum_{j=2}^{n_0+1} c^j(\mathbf{J}E_{z^j}(\theta) - \mathbf{J}E_{z^1}(\theta)).\\
  \end{multline}

But
 $$\textrm{row rank } \begin{bmatrix} \mathbf{J}E_{z_1}(\theta) \\ \mathbf{J}E_{z_2}(\theta)  \\ \vdots \\ \mathbf{J}E_{z_{n_0+1}}(\theta) \\  \mathbf{J}E_{z*}(\theta) \end{bmatrix} 
 =\textrm{row rank } \begin{bmatrix} \mathbf{J}E_{z_1}(\theta) \\ \mathbf{J}E_{z_2}(\theta) -  \mathbf{J}E_{z_1}(\theta)  \\  \vdots \\ \mathbf{J}E_{z_{n_0+1}}(\theta) - \mathbf{J}E_{z_1}(\theta)\\  \mathbf{J}E_{z*}(\theta) - \mathbf{J}E_{z_1}(\theta) \end{bmatrix} 
   $$
 and equation \eqref{eq:linearsum}  tells us that the last row of the rightmost matrix above is a linear combination of its other rows.  Therefore 
 \begin{equation} \label{eq:onelesspoint}
 \textrm{dim}_{\textrm{fun}}(\theta) = \textrm{rank }  \mathbf{J}E_{Z'}(\theta) =  \textrm{rank }\mathbf{J}E_{Z' \setminus \{z^*\}}(\theta).
 \end{equation}
 Inductively iterating the conclusion of \eqref{eq:onelesspoint} over all the points $z^* \in Z^* \setminus Z$ yields 
$$\textrm{dim}_{\textrm{fun}}(\theta)  = \textrm{row rank} \ \mathbf{J}E_{Z}(\theta).$$ 

The case $n_m > 1$ is similar; apply the argument above to each coordinate function of $\rho$. 
\end{proof}

\subsection{Equivalent definition of functional dimension for combinatorially stable parameters}

Equipped with a constructive (i.e. not containing a supremum) definition of functional dimension for generic, transversal and combinatorially stable parameters (Proposition \ref{p:reformulateoriginaldef}), we next prove this definition coincides with an alternative characterization of functional dimension for such parameters.

 \begin{definition} \label{def:SVmap}
 Let $\theta \in \mathcal{P}_{n_0,\ldots,n_m}$ be a generic, transversal, and combinatorially stable parameter.
    Fix any ordering $C_1,\ldots,C_k$ of the top-dimensional cells of $\mathcal{C}(\theta)$.  For each $1 \leq i \leq k$, choose one point $z_i$ in the interior of $C_i$, and set $Z^1 = \{z_1,\ldots,z_k\}$.  
  Let $V$ be a neighborhood of $\theta$ in $\mathcal{P}$ that preserves cellular membership of points in $Z^1$ (as guaranteed by Lemma \ref{l:preservingcellularnbhd}). 
  We define the \emph{slopes and values map} 
 $$SV_{\theta,Z^1}:V \to \left( \mathbb{R}^{n_m} \times \mathbb{R}^{n_0n_m} \right)^k,$$ as follows.  
For $v \in V$, for each $1 \leq i \leq k$, we define the $i$th element in $\left( \mathbb{R}^{n_m} \times \mathbb{R}^{n_0n_m} \right)^k$ of $SV(v)$ 
to be $$\rho(v)(z_i) \times \boldsymbol{J}(\rho(v))\vert_{z_i}^T,$$ where we interpret $\boldsymbol{J}(\rho(v))\vert_{z_i}^T$ as a vector listing its entries. 
 \end{definition}

\begin{remark}
In the definition above, the vector $\boldsymbol{J}(\rho(v))\vert_{z_i}^T$ amounts to a listing of the $n_0$ directional derivatives (in the directions of the axes in $\mathbb{R}^{n_0}$) of the $\mathbb{R}^{n_m}$-valued function $\rho(v)$ at the point $z_i$. Alternatively, these are the coefficients of the corresponding local multi-affine-linear function.
\end{remark}

The content of Lemma \ref{l:qrealizable} is that, because neural networks maps are continuous and PL, if we know the affine-linear functions that are the restrictions of the neural network map to each top-dimensional cell, we can figure out the loci of nondifferentiability of the neural network map. 

\begin{lemma} \label{l:qrealizable}
Let all notation be as in Definition \ref{def:SVmap}.
Let $q \in \textrm{Image}(SV_{\theta,Z^1})$  Then the unmarked function that realizes $q$ is uniquely determined.  
\end{lemma}

\begin{proof}
Let $f$ be a function that realizes $q$.  At one fixed point $z_C$ in each top-dimensional cell $C$, $q$ provides the value of $f(z_C)$ and the total derivative $\boldsymbol{J}(f)\vert_{z_C}$. Since $f$ is linear on cells of $\mathcal{C}(\theta)$, the restriction of $f$ to each top-dimensional cell is therefore determined.  
Implicit in the definition of the map $SV_{\theta,Z^1}$ is a recording of the stable combinatorial structure of $\mathcal{C}(\theta)$ -- including, in particular, the list of top-dimensional cells that are cofaces of any non-top-dimensional cell.  Consequently, because $f$ is continuous by assumption, the set of points that comprise any non-top-dimensional cell $D$ can be determined by solving for the intersection loci of the affine-linear functions that are the restrictions of $f$ to the top-dimensional cells that are co-faces of $D$.  
\end{proof}

Lemma \ref{l:qrealizable} justifies the following definition.

\begin{definition} \label{def:commutingdiagram}
Let $\theta$ be a generic, transversal, combinatorially stable parameter, let $Z^1$ be as in Definition \ref{def:SVmap}, let $Z$ be a decisive set for $\theta$, and let $V$ be a neighborhood of $\theta$ in $\mathcal{P}$ on which cellular membership of points in $Z^1$ and $Z$ is preserved. 
Define the map $\Phi:\textrm{SV}_{\theta,Z^1}(V)  \to \textrm{SE}_{\theta,Z}(V)$ so that the following diagram commutes: 
\begin{equation} \label{commutativediagram} \begin{tikzcd}[column sep=small]
V \arrow[rr] \arrow[dr] &  &  \textrm{SE}_{\theta,Z}(V)\\
&\textrm{SV}_{\theta,Z^1}(V) \arrow[ur,"\Phi" '] & 
\end{tikzcd}\end{equation}

\end{definition}

\begin{remark}
The map $\Phi$ depends on $Z^1$, $Z$, and the cellular membership of points in these sets to top-dimensional cells of $\mathcal{C}(\theta)$.  While the notation $\Phi_{\theta,Z^1,Z}$ might be more accurate, we omit the subscripts to lighten notation. 
\end{remark}

Equipped with Proposition \ref{p:reformulateoriginaldef}, which says that functional dimension is the rank of the simplex evaluation map -- which is the map at the top of the commutative diagram \eqref{commutativediagram} -- our next goal is to use this commutative diagram to show that functional dimension is also the rank of the SV (slopes and values) map.  Our strategy to do this is to argue that the map $\Phi$ has full rank -- but in order to discuss the rank of $\Phi$, we need $\Phi$ to be a differentiable map between smooth manifolds. To this end, since the subset $\textrm{SV}_{\theta,Z^1}(V) \left( \mathbb{R}^{n_m} \times \mathbb{R}^{n_0n_m} \right)^k $ may not be a manifold, we will construct an extension $\widetilde{\Phi}$ of $\Phi$ defined on all of $\left( \mathbb{R}^{n_m} \times \mathbb{R}^{n_0n_m} \right)^k$.

\begin{equation} \label{generalizedcommutativediagram}
\begin{tikzcd}[column sep=small]
V \arrow[rr,"\textrm{SE}_{\theta,Z}"] \arrow[dr, "\textrm{SV}_{\theta,Z^1}" '] &  &  \left( (\mathbb{R}^{n_m})^{n_0 + 1}\right)^k\\
& \left( \mathbb{R}^{n_m} \times \mathbb{R}^{n_0n_m} \right)^k \arrow[ur,"\widetilde{\Phi}" '] & 
\end{tikzcd}
\end{equation} 

Here, $k$ is the number of top-dimensional cells in $\mathcal{C}(\theta)$ (where $\theta$ is assumed to be a generic, transversal, combinatorially stable parameter), as in Definitions \ref{d:simplexevaluationmap} and \ref{def:SVmap}.

Note that the maps $\Phi:\textrm{SV}_{\theta, Z^1}(V) \to \textrm{SE}_{\theta,Z}(V)$ and $\Phi^{-1}$ can be defined regardless of whether or not the data that is the input of the map is realizable by a neural network function.  Specifically, for each top-dimensional cell, $\Phi$ acts by computing, based on the values assigned to points forming a geometric simplex, the affine-linear function that realizes those values; conversely, for each top-dimensional cell, $\Phi^{-1}$ acts by taking the total derivative and value at a prescribed point in the cell and computing the values of the corresponding affine-linear function at the vertices of a simplex.  These computations can be done regardless of whether the input point is realizable, and hence there is a \emph{natural extension $\widetilde{\Phi}$} of $\Phi$ (resp. $\widetilde{\Phi}^{-1}$ of $\Phi^{-1}$) to all of $\left( \mathbb{R}^{n_m} \times \mathbb{R}^{n_0n_m} \right)^k$ (resp. $\left( (\mathbb{R}^{n_m})^{n_0 + 1}\right)^k$).  Furthermore, with $\widetilde{\Phi}$ defined in this way, commutativity as in the diagram \eqref{generalizedcommutativediagram} holds, and the coordinate functions of $\widetilde{\Phi}$ and $\widetilde{\Phi}^{-1}$ can be written as polynomials (over $\mathbb{R}$) in terms of the coordinates of the inputs.  Consequently, we have the following characterization: 

\begin{lemma} \label{l:extendingPhi}
The maps $\widetilde{\Phi}$ and $\widetilde{\Phi}^{-1}$ are smooth maps between differentiable manifolds, and are extensions of $\Phi$ and $\Phi^{-1}$, respectively.  
\end{lemma}

Theorem \ref{t:equivalentdef} may be taken as the definition of functional dimension at generic, transversal, combinatorially stable parameters $\theta$.    It says that for each top-dimensional cell $C$ of $\mathcal{C}(\theta)$, one records
\begin{enumerate}
 \item the value of the function $\rho(\theta)$ at some chosen point $z_C$ in the interior of $C$, and 
 \item  the total derivative of the (affine-linear) function $\rho(\theta)$ at $z_C$.  
 \end{enumerate} 
 The functional dimension at $\theta$ is then the number of degrees of freedom you can achieve in this data by perturbing the parameter $\theta$.

\begin{theorem} \label{t:equivalentdef}
Let $\theta$ be a generic, transversal, combinatorially stable parameter, let $Z^1$ be as in Definition \ref{def:SVmap}, let $Z$ be a decisive set for $\theta$, and let $V$ be a neighborhood of $\theta$ in $\mathcal{P}$ on which cellular membership of points in $Z^1$ and $Z$ is preserved. Then 
$$\textrm{dim}_{\textrm{fun}} (\theta) = \textrm{rank } \textrm{SV}_{\theta,Z^1} \vert_{\theta}.$$

\end{theorem}

\begin{proof}[Proof of Theorem \ref{t:equivalentdef}]
 We have the functional equation $$\textrm{SE}_{\theta,Z} =  \widetilde{\Phi} \circ \textrm{SV}_{\theta,Z^1}$$ of Definition  \ref{def:commutingdiagram}, and 
$$\textrm{dim}_{\textrm{fun}}(\theta) = \textrm{rank } \textrm{SE}_{\theta,Z} \vert_s$$ by Proposition \ref{p:reformulateoriginaldef}. 
The inverse map $\widetilde{\Phi}^{-1}$ exists and is smooth by Lemma \ref{l:extendingPhi}.  This implies that $\widetilde{\Phi}$ has full rank everywhere on its domain, and hence 
$$\textrm{rank } \textrm{SE}_{\theta,Z} \vert_\theta = \textrm{rank } \textrm{SV}_{\theta,Z^1} \vert_{\theta}.$$ 
\end{proof}

\begin{example}
We make Theorem \ref{t:equivalentdef} more transparent by applying it to compute the functional dimension for the example from \S \ref{s:illustrativeexample} at the parameter  $\theta_0 = (2,-5,-1,4,1,1,1)$.

First, we will explain why the parameter $\theta_0 = (2,-5,-1,4,1,1,1)$ is \emph{combinatorially stable}.  The canonical polyhedral complex $\mathcal{C}(\theta_0)$ consists of three top-dimensional cells -- the intervals $(-\infty, 2.5]$, $[2.5,4]$ and $[4, \infty)$ -- and the lower-dimensional faces of these cells.  Following \S \ref{s:illustrativeexample}, we call these top-dimensional cells the ``left piece,'' ``middle piece'' and ``right piece'' of $\mathcal{C}(\theta_0)$.  
For any sufficiently small perturbation $\theta_0' = \theta_0 + \epsilon$ of the parameter, the canonical polyhedral complex $\mathcal{C}(\theta_0')$ still  has precisely three top-dimensional cells (intervals) -- which we will call the ``left/right/middle piece'' of $\mathcal{C}(\theta_0')$ -- and their lower-dimensional faces.  Moreover, the activation pattern on the left (resp. middle or right) piece of $\mathcal{C}(\theta_0')$ is the same as the activation pattern on the left (resp. middle or right) piece of $\mathcal{C}(\theta_0)$.  

Next, we select a set $Z^1 = \{z_L, z_M, z_R\}$ by picking one point in the interior of each top-dimensional cell of $\mathcal{C}(\theta_0)$.  Specifically, we pick $z_L = 1$, $z_M = 3$, $z_R = 6$.  Note that for every parameter $\theta_0'$ sufficiently close to $\theta_0$, the point $z_L = 1$ (resp. $z_M =3$, $z_R = 6$) remains in the left  (resp. middle, right) piece of $\mathcal{C}(\theta_0')$.  That is, any sufficiently small neighborhood of $\theta_0$ \emph{preserves the cellular membership} (Def. \ref{d:cellularmembership}) of the set $Z^1$. 

Fix a small neighborhood $V$ of $\theta_0$ that preserves cellular membership of $Z^1$. We now compute the map $SV_{\theta_0,Z^1}\colon V\to(\mathbb{R}^1\times\mathbb{R}^1)^3$. In \S \ref{s:illustrativeexample}, for any choice of parameter 
$$\theta_0' = (w_{11}^1, w_{21}^1, b_1^1, b_2^1, w^2_{11}, w^2_{12}, b^2_1) \in V$$ 
 we found the expression for the function $\rho(\theta_0')$ on each of the three pieces of $\mathcal{C}(\theta_0')$:  
\begin{align*}
\textrm{Left piece: } &  z \mapsto w^{2}_{12}w^{1}_{21}z+w^{2}_{12}b^{1}_2+b^{2}_1, \\
\textrm{ Middle piece: } &  z \mapsto (w^{2}_{11}w^{1}_{11} +w^{2}_{12}w^{1}_{21})z +w^{2}_{11}b^{1}_1 +w^{2}_{12}b^{1}_2+b^{2}_1, \\
\textrm{Right piece: } & z \mapsto w^{2}_{11}w^{1}_{11}z+w^{2}_{11}b^{1}_1+b^{2}_1. \\
\end{align*} 
Hence, the slopes and values for each input $z_i$ are 
\begin{align*}
z_L = 1 &: &  \textrm{ slope }  & w^{2}_{12}w^{1}_{21},   & \textrm{ value }& w^{2}_{12}w^{1}_{21}+w^{2}_{12}b^{1}_2+b^{2}_1, \\
 z_M = 3 &: &  \textrm{ slope }  & w^{2}_{11}w^{1}_{11} +w^{2}_{12}w^{1}_{21}, & \textrm{ value }  &3(w^{2}_{11}w^{1}_{11} +w^{2}_{12}w^{1}_{21}) +w^{2}_{11}b^{1}_1 +w^{2}_{12}b^{1}_2+b^{2}_1 \\
 z_R = 6 &: & \textrm{ slope } &  w^{2}_{11}w^{1}_{11},   & \textrm{ value } & 6w^{2}_{11}w^{1}_{11}+w^{2}_{11}b^{1}_1+b^{2}_1. \\
\end{align*}
Thus, the map $SV_{\theta_0,Z^1}:V \to (\mathbb{R}^1 \times \mathbb{R}^1)^3 \cong \mathbb{R}^6$ is given by 

\begin{multline*}
 (w_{11}^1, w_{21}^1, b_1^1, b_2^1, w^2_{11}, w^2_{12}, b^2_1) \mapsto \\ 
\big( w^{2}_{12}w^{1}_{21},  \, w^{2}_{12}w^{1}_{21}+w^{2}_{12}b^{1}_2+b^{2}_1, \\
 w^{2}_{11}w^{1}_{11} +w^{2}_{12}w^{1}_{21}, \, 3(w^{2}_{11}w^{1}_{11} +w^{2}_{12}w^{1}_{21}) +w^{2}_{11}b^{1}_1 +w^{2}_{12}b^{1}_2+b^{2}_1, \\
w^{2}_{11}w^{1}_{11},  \, 6w^{2}_{11}w^{1}_{11}+w^{2}_{11}b^{1}_1+b^{2}_1 \big). 
\end{multline*}

We will now compute $\textrm{rank } \textrm{SV}_{\theta_0,Z^1} \vert_{\theta_0}$.

\[
\begingroup
\setlength{\arraycolsep}{3pt}
\begin{aligned}
\mathrm{rank}\,\mathrm{SV}_{\theta_0,Z^1}\big|_{\theta_0}
&=
\mathrm{rank}
\left.
\begin{bmatrix}
0 & w^{2}_{12} & 0 & 0 & 0 & w^{1}_{21} & 0\\
0 & w^{2}_{12} & 0 & w^{2}_{12} & 0 & w^{1}_{21}+b^{1}_{2} & 1\\
w^{2}_{11} & w^{2}_{12} & 0 & 0 & w^{1}_{11} & w^{1}_{21} & 0\\
3w^{2}_{11} & 3w^{2}_{12} & w^{2}_{11} & w^{2}_{12} &
  3w^{1}_{11}+b^{1}_{1} & 3w^{1}_{21}+b^{1}_{2} & 1\\
w^{2}_{11} & 0 & 0 & 0 & w^{1}_{11} & 0 & 0\\
6w^{2}_{11} & 0 & w^{2}_{11} & 0 & 6w^{1}_{11}+b^{1}_{1} & 0 & 1
\end{bmatrix}
\right|_{s_0} \\[6pt]
&=
\mathrm{rank}
\begin{bmatrix}
0 & 1 & 0 & 0 & 0 & -5 & 0\\
0 & 1 & 0 & 1 & 0 & -1 & 1\\
1 & 1 & 0 & 0 & 2 & -5 & 0\\
3 & 3 & 1 & 1 & 5 & -11 & 1\\
1 & 0 & 0 & 0 & 2 & 0 & 0\\
6 & 0 & 1 & 0 & 11 & 0 & 1
\end{bmatrix}
= 5.
\end{aligned}
\endgroup
\]
This agrees with the computation in §\ref{s:illustrativeexample}, where $\dim_{\mathrm{fun}}(\theta_0)=5$ was obtained using the original definition, confirming that the alternative definition (Theorem \ref{t:equivalentdef}) yields the same value at $\theta_0$. 
\end{example}


\section{Functional dimension and the Neural Tangent Kernel} \label{s:neuraltangentkernel}

\subsection{Neural Tangent Kernel}
The Neural Tangent Kernel (NTK) of a parameterized function space $\mathcal{F} = \{f: \mathbb{R}^{n_0} \rightarrow \mathbb{R}^{n_m}\}$ was defined in \cite{JacotGabrielHongler} in order to investigate the relationship between ordinary gradient descent of an empirical cost function on parameter space and kernel gradient descent of the cost function on function space. In the time since the paper appeared, the NTK has become a central object of study for those interested in understanding the convergence and generalization properties of parameterized function classes such as feedforward neural networks.

Note that for the NTK to be well-defined, one needs to restrict the functions in $\mathcal{F}$ to be continuously differentiable ($C^1$) with respect to the parameters. Since the ReLU activation function introduces non-differentiable points, we must exercise care when defining the NTK for the class of ReLU neural network functions.

\begin{definition} \label{d:ntk}
Fix an architecture  $(n_0,\ldots,n_m)$, and let 
 $$\mathcal{W} = \left \{(\theta, x,y) \in  \mathcal{P}_{n_0,\ldots,n_m}  \times \mathbb{R}^{n_0} \times \mathbb{R}^{n_0} \,\,\vert\,\,  \frac{\partial \rho(\theta)(x)}{\partial s_i} \textrm{ and }   \frac{\partial \rho(\theta)(y)}{\partial \theta_i} \textrm{ exist  for all } i   \right \}.$$
The \emph{Neural Tangent Kernel} is the map 
$$\textrm{NTK}: \mathcal{W}  \to \{(n_m \times n_m) \textrm{ real-valued matrices}\}$$ 
where the $(k,l)$th entry of $\textrm{NTK}(\theta,x,y)$ is 
$$\textrm{NTK}(\theta,x,y)_{k,l} := \sum_{i=1}^{D_{n_0,\ldots,n_m}} \frac{\partial \rho_k(\theta)(x)}{\partial \theta_i} \frac{\partial \rho_l(\theta)(y)}{\partial \theta_i}$$
where $\rho_k(\theta)$ denotes the $k$th coordinate function of $\rho(\theta)$. 
\end{definition}

Recall that for $z \in \mathbb{R}^{n_0}$ we defined the map $E_z:\mathcal{P}_{n_0,\ldots,n_m} \to \mathbb{R}^{n_m}$ by $E_z(\theta) = (\rho_1(\theta)(z),\ldots,\rho_{n_m}(\theta)(z))$, and--for a parametrically smooth point $z$--the Jacobian matrix of $E_z$ evaluated at $\theta$,
 \[{\bf J}E_z (\theta) = \left[\frac{\partial(E_z(\theta))_i}{\partial{\theta_j}}\right] = \left[  \frac{\partial \rho_i(\theta)(z) }{\partial \theta_j} \right]\] 
 is the $n_m \times D$ matrix whose entry in the $i$th row and $j$th column records the partial derivative of the $i$th coordinate of $E_z(\theta) \in \mathbb{R}^{n_m}$ with respect to the $j$th parameter.  Then for $(\theta,x,y) \in \mathcal{W}$, it is immediate that 
 \begin{equation} \label{eq:NTKJacobian}
 \textrm{NTK}(\theta,x,y) = {\bf J}E_x (\theta) \cdot ( {\bf J}E_y (\theta))^T.
 \end{equation}

\begin{remark} Note that as defined above NTK is bilinear with respect to the latter two inputs, and $\textrm{NTK}(\theta,x,y) = (\textrm{NTK}(\theta,y,x))^T$. Moreover, it follows immediately from the definitions that if $\mathcal{S}_\theta \subseteq \mathbb{R}^{n_0}$ is the subset of parametrically smooth points for a fixed parameter $\theta$, then $ \{\theta\} \times \mathcal{S}_\theta \times \mathcal{S}_\theta$ is contained in $\mathcal{W}$. Recalling (cf. \cite{JacotGabrielHongler}) that an $n_m$--dimensional kernel $K$ on a subset $\mathcal{S} \subseteq \mathbb{R}^{n_0}$ is a bilinear map \[K: \mathcal{S} \times \mathcal{S} \rightarrow \mathbb{R}^{n_m \times n_m},\] each ReLU neural network function $\rho(\theta)$ endows the subset $\mathcal{S}_\theta \subseteq \mathbb{R}^{n_0}$ with an $n_m$--dimensional kernel.
\end{remark}

What is the full set of points in the domain of the NTK --  the set $\mathcal{W}$ of Definition \ref{d:ntk}?  

\begin{lemma}
For any architecture, the set $\mathcal{W}$ that is the domain of the NTK has full measure (with respect to Lebesgue measure on $ \mathcal{P} \times \mathbb{R}^{n_0} \times \mathbb{R}^{n_0} $). 
\end{lemma}

\begin{proof}
This is a consequence of the fact that the parameterized family is finitely piecewise polynomial (Theorem \ref{t:Fpiecewisepolynomial}). In particular, $\mathcal{W}$ contains the set of points $(\theta,x,y)$ such that $\theta$ is an ordinary parameter and $x,y$ are parametrically smooth points for $\theta$, and this set has full measure in $\mathbb{R}^{n_0} \times \mathbb{R}^{n_0} \times \mathcal{P}$.
\end{proof}

We will now briefly review the connection between the NTK and parametric gradient descent. For more details, see \cite[Sec. 3]{JacotGabrielHongler}.
Given input-output data pairs $\{x_i,y_i\}_{i=1}^N \in \mathbb{R}^{n_0} \times \mathbb{R}^{n_m}$, and a pointwise cost function $c:\mathbb{R}^{n_m} \to \mathbb{R}$ (which we will assume is differentiable at all points where we evaluate it), gradient flow decreases the total cost function $C:\mathcal{P}_{n_0,\ldots,n_m} \to \mathbb{R}$, 
$$C(\theta) \coloneqq \sum_{i=1}^N c \left(\rho(\theta)(x_i) - y_i\right).$$
Consider a fixed starting parameter $\theta_0$ for gradient flow.  Write $u_i^\theta = \rho(\theta)(x_i)$ for a generic $\theta$ and $u_i^0 = \rho(s_0)(x_i)$.  Applying the chain rule gives
\begin{multline} \label{eq:Cderiv}
 \frac{\partial C(\theta)}{\partial \theta} \Big \vert_{\theta=\theta_0} = \sum_{i=1}^N \frac{ \partial c(u_i^\theta - y_i)}{ \partial u^\theta_i} \Big \vert_{u_i^\theta= u_i^0}  \frac{\partial (u_i^\theta - y_i)}{\partial s}\Big \vert_{\theta=\theta_0} \\
 =  \sum_{i=1}^N \frac{ \partial c(u_i^\theta - y_i)}{ \partial u^\theta_i} \Big \vert_{u_i^\theta= u_i^0}  \frac{\partial u_i^\theta }{\partial \theta}\Big \vert_{\theta=\theta_0}
\end{multline} 
Here, we use the notation $\frac{\partial}{\partial}$ for the total derivative or Jacobian matrix; this notation emphasizes which variables are dependent and independent.  The left side of \eqref{eq:Cderiv} is a $1 \times D$ row vector; the left term inside the sum on the right is a  $1 \times n_m$ row vector and the right term inside the sum on the right is a $n_m \times D$ matrix. 
Since gradient flow lines for $\theta=\theta_t$ follow the negative gradient of $C$, we have (letting $t_0$ be such that $\theta_{t_0} = \theta_0$)
\begin{multline} \label{eq:partialst}
\frac{\partial \theta_t}{\partial t} \Big \vert_{t=t_0} = -  \left( \frac{\partial C(\theta)}{\partial \theta} \Big \vert_{\theta=\theta_0} \right)^T =
-   \sum_{i=1}^N  \left(  \frac{ \partial c(u_i^\theta - y_i)}{ \partial u^\theta_i} \Big \vert_{u_i^\theta= u_i^0}  \frac{\partial u_i^\theta}{\partial \theta}\Big \vert_{\theta=\theta_0} \right)^T \\
= - \sum_{i=1}^N  \left(
 \frac{\partial u_i^\theta}{\partial \theta}\Big \vert_{\theta=\theta_0}
   \right)^T \left(
    \frac{ \partial c(u_i^\theta - y_i)}{ \partial u^\theta_i} \Big \vert_{u_i^\theta= u_i^0} 
   \right)^T
\end{multline}
Hence for any $z \in \mathbb{R}^{n_0}$, 
\begin{multline} \label{eq:NTKconnectionToGD}
\frac{\partial \rho(\theta_t)(z)} {\partial t} \Big \vert_{t=t_0}= \frac{\partial \rho(\theta)(z)}{\partial \theta} \Big \vert_{\theta = \theta_0}  \frac{\partial \theta_t}{\partial t} \Big \vert_{t = t_0}\\
 =  -  \sum_{i=1}^N  \frac{\partial \rho(\theta)(z)}{\partial \theta} \Big \vert_{\theta= \theta_0}\left( \frac{\partial u_i^\theta}{\partial \theta}\Big \vert_{\theta=\theta_0} \right)^T \left( \frac{ \partial c(u_i^\theta - y_i)}{ \partial u^\theta_i} \Big \vert_{u_i^\theta= u_i^0} \right)^T \\
 = -  \sum_{i=1}^N  \frac{\partial \rho(\theta)(z)}{\partial \theta} \Big \vert_{\theta = \theta_0} \left( \frac{\partial \rho(\theta)(x_i)}{\partial \theta}\Big \vert_{\theta=\theta_0}\right)^T \left( \frac{ \partial c(u_i^\theta - y_i)}{ \partial u^\theta_i} \Big \vert_{u_i^\theta= u_i^0}\right)^T  \\ 
 = - \sum_{i=1}^N \textrm{NTK}(\theta_0,z,x_i) \left( \frac{ \partial c(u_i^\theta - y_i)}{ \partial u^\theta_i} \Big \vert_{u_i^\theta= u_i^0}\right)^T
\end{multline}

Equation \eqref{eq:NTKconnectionToGD} shows that the derivative of $\rho(\theta)(z)$ with respect to time during gradient descent is a weighted sum of the neural tangent kernels $\textrm{NTK}(\theta,z,x_i)$, where $x_i$ ranges over the input points of the sample data set.

\bigskip 

Equations \eqref{eq:NTKJacobian} and \eqref{eq:NTKconnectionToGD} suggest that the number of linearly independent ways the restriction of a neural network function $\rho(\theta)$ to a set $Z$ can change under perturbations of $\theta$ is determined by the collection of matrices $\{ \textrm{NTK}(\theta,z_i,z_j) \mid z_i,z_j \in Z\}$.  This motivates the following definition.

\begin{definition} \label{def:batchNTK}
 Let $\theta \in \mathcal{P}_{n_0,\ldots,n_m}$ be an ordinary parameter.  Let $Z = \{z_1,\ldots,z_N\}$ be a finite set of parametrically smooth points for $\theta$.  We define the \emph{batch neural tangent kernel} for the set $Z$ and parameter $\theta$, denoted $\mathcal{K}_Z(\theta)$, 
 to be the  $Nn_m \times Nn_m$ matrix that is a block matrix consisting of $N \times N$ blocks each of size $n_m \times n_m$, and whose $(i,j)^{\textrm{th}}$ $n_m \times n_m$ block  is $\textrm{NTK}(\theta,z_i,z_j)$. \end{definition}

The following interpretation of the batch neural tangent kernel is immediate from the definitions, as the reader may verify.

\begin{lemma} \label{l:ntkmatrix} Let $\theta \in \mathcal{P}_{n_0,\ldots,n_m}$ be an ordinary parameter and let $Z = \{z_1,\ldots,z_N\}$ be a finite set of parametrically smooth points for $\theta$.  Then 
$$\mathcal{K}_Z(\theta)  =\mathbf{J}E_Z(\theta)   \mathbf{J}E_Z(\theta)^T.$$
\end{lemma}

\begin{theorem} \label{t:ntkDim}
Let $\theta \in \mathcal{P}_{n_0,\ldots,n_m,1}$ be an ordinary parameter.  
\begin{enumerate}
\item Batch functional dimension coincides with the rank of the batch neural tangent kernel, i.e. if $Z$ is a finite set of parametrically smooth points for $\theta$ then 
$$\textrm{dim}_{\textrm{ba.fun}}(\theta,Z) = \textrm{rank }\mathcal{K}_Z(\theta).$$ 
\item Functional dimension is the sup of the rank of the batch neural tangent kernels, i.e. 
$$\textrm{dim}_{fun}(\theta) = \sup_{Z \subset \mathbb{R}^{n_0} \textrm{ finite and parametrically smooth for }\theta} \textrm{rank} (\mathcal{K}_Z(\theta)).$$
\end{enumerate}
\end{theorem}

\begin{proof}
A standard fact from linear algebra is that for any matrix $M$ with real entries, $\textrm{rank}(M) = \textrm{rank}(M M^T)$. Thus, by Lemma \ref{l:ntkmatrix},  $$\textrm{rank}(\mathcal{K}_Z(\theta)) = \textrm{rank } \mathbf{J}E_Z(\theta),$$ and the result follows. 
\end{proof}


Proposition \ref{p:gradientrelationship} shows that the total derivative of the cost function at $s_0$ may be expressed in terms of $\mathbf{J}E_{\{z_1,\ldots,z_N\}}(\theta_0)$, providing a link between functional dimension and gradient descent. 

\begin{proposition}  \label{p:gradientrelationship}
Let $A$ be the $1 \times Nn_m$ row vector formed by concatenating the $N$ $1\times n_m$ row vectors $\frac{ \partial c(u_i^\theta - y_i)}{ \partial u^\theta_i} \Big \vert_{u_i^\theta= u_i^0}$.  Then 
$$ \frac{\partial C(\theta)}{\partial \theta} \Big \vert_{\theta=\theta_0} = A \cdot \mathbf{J}E_{\{z_1,\ldots,z_N\}}(\theta_0)$$
and hence, letting $t_0$ be such that $\theta_{t_0} = \theta_0$, 
$$\frac{\partial \theta_t}{\partial t} \Big \vert_{t=t_0}  = -   \left( \frac{\partial C(\theta)}{\partial \theta} \Big \vert_{\theta=\theta_0} \right)^T  = - \left(A \cdot \mathbf{J}E_{\{z_1,\ldots,z_N\}}(\theta_0)\right)^T.$$

\end{proposition}

\begin{proof}
Notice that 
$$\frac{\partial u_i^\theta }{\partial \theta}\Big \vert_{\theta=\theta_0} = \frac{\partial \rho(\theta)(z_i)}{\partial \theta} \vert_{\theta=\theta_0}$$ is the $n_m \times D$ matrix whose $j$th row is 
$$ \begin{bmatrix}  \frac{\partial \rho_j(\theta_0)(z_i)}{\partial \theta_1}  \quad \cdots \quad  \frac{\partial \rho_j(\theta_0)(z_i)}{\partial \theta_D}  \end{bmatrix},$$
where $\rho_j(\theta)$ denotes the $j$th coordinate function of $\rho(\theta)$.  That is, it is the $j$th row of  $\mathbf{J}E_{z_i}(\theta_0)$.
Hence 
$$ \frac{ \partial c(u_i^\theta - y_i)}{ \partial u^\theta_i} \Big \vert_{u_i^\theta= u_i^0}  \frac{\partial u_i^\theta }{\partial \theta}\Big \vert_{\theta=\theta_0}  = \frac{ \partial c(u_i^\theta - y_i)}{ \partial u^\theta_i} \Big \vert_{u_i^\theta= u_i^0} \mathbf{J}E_{z_i}(\theta_0)$$

So equation \eqref{eq:Cderiv} can be rewritten as 
\begin{equation} \label{eq:Cgrad}
 \frac{\partial C(\theta)}{\partial \theta} \Big \vert_{\theta=\theta_0}  = 
  \sum_{i=1}^N \frac{ \partial c(u_i^\theta - y_i)}{ \partial u^\theta_i} \Big \vert_{u_i^\theta= u_i^0}  \frac{\partial u_i^\theta }{\partial \theta}\Big \vert_{\theta=\theta_0} 
  =  \sum_{i=1}^N \frac{ \partial c(u_i^\theta - y_i)}{ \partial u^\theta_i} \Big \vert_{u_i^\theta= u_i^0}  \mathbf{J}E_{z_i}(\theta_0).
\end{equation} 

\end{proof}

An important consequence of Proposition \ref{p:gradientrelationship} is that flow lines of gradient descent at $\theta_0$ (using batch inputs $\{z_1,\ldots,z_N\}$) are constrained to directions that lie in a space whose dimension is at most $\textrm{rank } \boldsymbol{J}E_{\{z_1,\ldots,z_N \}}(\theta_0)$.

\section{Upper bound on functional dimension} \label{s:upperbound}

The scaling invariance of the realization map $\rho$ yields the following upper bound on functional dimension. 

\begin{theorem} \label{t:upperbound}
  For any ordinary parameter $\theta \in \mathcal{P}_{n_0, \ldots, n_m}$ in any architecture $(n_0,\ldots,n_m)$,
  $$\textrm{dim}_{\textrm{fun}}(\theta) \leq n_m  +\sum_{i=0}^{m-1}n_in_{i+1}.$$
\end{theorem}

\begin{proof} For any parameter $\theta \in \mathcal{P}_{n_0,\ldots,n_m}$, suppose the  associated marked neural network map, $\bar{\rho}(\theta)$, is given by 
 $$\mathbb{R}^{n_0}  \xrightarrow{ \sigma \circ A^{1}} \mathbb{R}^{n_1}  \xrightarrow{\sigma \circ A^{2}}\dots \xrightarrow{ \sigma \circ A^{m}} \mathbb{R}^{n_m}.$$
 For any layer index $0 < i<m$, let $h$ be any linear self-map of $\mathbb{R}^{n_i}$ that is represented by a diagonal matrix with all positive entries.  Note that $h$ commutes with the component-wise ReLU map $\sigma: \mathbb{R}^{n_i}  \to  \mathbb{R}^{n_i}$.  Therefore 
 \begin{equation} \label{eq:scaling}
 \sigma \circ A^{i+1}\circ h^{-1} \circ \sigma \circ h\circ A^{i}=\sigma \circ A^{i+1}\circ \sigma \circ A^{i},
 \end{equation}
 i.e.  replacing $A^{i}$ with $h\circ A^{i}$ and $A^{i+1}$ with $A^{i+1}\circ h^{-1}$ will not change the  unmarked function $\rho(s)$.  Indeed, cf. \cite{ZhaoGWYD23}, there is a global action on $\mathcal{P}_{n_0, \ldots, n_m}$ of the Lie group $D_{n_1} \times \ldots \times D_{n_{m-1}} \subseteq GL_{n_1} \times \ldots \times GL_{n_{m-1}}$  of positive diagonal matrices whose dimensions are equal to the dimensions of the hidden layers. Therefore, parameter space decomposes into orbits for this action, and by the classical orbit-stabilizer theorem for Lie group actions (cf. \cite{Lee}), any stabilizer-free orbit will have dimension equal to the dimension of $D_{n_1} \times \ldots \times D_{n_{m-1}}$, which is $\sum_{i=1}^{m-1}n_i$. Moreover, it is straightforward to check that the parameters with nontrivial stabilizer are contained in the set of non-generic parameters (indeed, they are all {\em degenerate}, in the terminology of \cite[Sec. 2]{GrigsbyLindsey}). This implies that the set of points with trivial stabilizer is full measure. Now for such a point $\theta$ and any batch $Z$ of parametrically smooth points for $\theta$, an application of the rank-nullity theorem for $\boldsymbol{J}E_Z(s)$ viewed as a map from the tangent space of $\theta$ to itself then tells us:
\begin{multline*}
 \textrm{rank}(\boldsymbol{J}E_Z(\theta)) \leq D(n_0,\ldots,n_m) - \sum_{i=0}^{m-1} n_i \\
 = \left(\sum_{i=0}^{m-1} n_{i+1}n_i + n_{i+1} \right) - \left(  \sum_{i=1}^{m-1} n_i \right)  = \left(\sum_{i=0}^{m-1} n_{i+1}{n_i} \right) + n_m.
 \end{multline*}
 
The upper bound on  $\textrm{rank}(\boldsymbol{J}E_Z(\theta))$ above applies also to ordinary parameters with nontrivial stabilizer, since the rank of the derivative of any piecewise polynomial map is lower semi-continuous, and the set of points with trivial stabilizer is full measure.
\end{proof}

\begin{corollary} \label{c:dimdifference}
For any architecture $(n_0,\ldots,n_m)$,
$$\textrm{dim}(\mathcal{P}_{n_0,\ldots,n_m}) - \textrm{dim}_{\textrm{fun}}(\mathcal{P}_{n_0,\ldots,n_m}) \geq n_1 + \ldots + n_{m-1}.$$
\end{corollary}

\begin{theorem}
If an ordinary parameter $s$ has a stably unactivated neuron, then $\textrm{dim}_{\textrm{fun}}(\theta)$ does not achieve the upper bound of Theorem \ref{t:upperbound}.
\end{theorem}

\begin{proof}
Suppose $\theta \in \mathcal{P}_{n_0,\ldots,n_m}$ has a stably unactivated neuron $N$.  Let $\mathcal{P}_0$ be the network architecture that omits the neuron $N$ (and only that neuron), and let $\theta_0 \in \mathcal{P}_0$ be the parameter that coincides with $\theta$ in all the coordinates of $s$ that do not determine the neuron $N$.  Let $Z \subset \mathbb{R}^{n_0}$ be any finite set of parametrically smooth points for $\theta$. Then $Z$ is also a parametrically smooth set for $s_0$. Because $\theta$ is stably unactivated, the columns of $\boldsymbol{J}E_Z(s)$ that correspond to taking derivatives with respect to parameters that determine $N$ are all $0$.  Also, the other columns of $\boldsymbol{J} (E_Z(\theta))$ coincide with the columns of $\boldsymbol{J}(E_Z(\theta_0))$.  Since padding a matrix with columns of zeroes does not change its rank, we have that $\textrm{rank} \left( \boldsymbol{J}E_Z(\theta) \right)= \textrm{rank} \left( \boldsymbol{J}E_Z(\theta_0)\right)$.   Hence $\textrm{dim}_{\textrm{fun}}(\theta) \leq \textrm{dim}_{\textrm{fun}}(\theta_0)$, and the upper bound from Theorem \ref{t:upperbound} for the reduced complexity architecture $\mathcal{P}_0$ applies.  
\end{proof}

\begin{lemma} \label{lem:lowbdunactive}
Suppose, for any fixed architecture $(n_0,\ldots,n_m)$, the coordinates of a parameter that determines a neural network are selected randomly from a probability distribution on $\mathbb{R}$ that is symmetric about $0$.  Then for any neuron in any layer $i \geq 2$, the probability that the neuron is stably unactivated is at least $$\frac{1}{2^{1+n_{i-1}}}.$$
  
\noindent Consequently, the probability that a randomly selected network of architecture $(n_0,\ldots,n_m)$ has at least one stably unactivated neuron is  at least 
 $$ 1 - \prod_{i=2}^{m} \left(1 - \frac{1}{2^{1+n_{i-1}}} \right)^{n_i}.$$
\end{lemma}

\begin{proof}
A sufficient condition for a neuron from layer $i \geq 2$ to be stably unactivated is that, denoting the associated hyperplane in $\mathbb{R}^{n_{i-1}}$ by $H$, the closed positive orthant of $\mathbb{R}^{n_{i-1}}$ is entirely contained in one of the two connected components of $\mathbb{R}^{n_{i-1}} \setminus H$ and the oriented normal vector to $H$ points into the connected component that does not contain the positive orthant.  This is equivalent to the oriented normal vector lying in the all-negative orthant of $\mathbb{R}^{n_{i-1}}$ and the bias of $H$ being negative. There is a $\tfrac{1}{2^{n_{i-1}}}$ chance of the oriented normal lying in the all-negative orthant, and a $\tfrac{1}{2}$ probability that the bias is negative; furthermore, these probabilities are independent.  Hence, the probability that any given neuron from layer $i \geq 2$ is stably unactivated is $\geq \tfrac{1}{2^{1+n_{i-1}}}$, proving the first claim.  To see the second claim, note that the probability that {\em no} neurons in layers $2$ and beyond are stably unactivated is \[\leq \prod_{i=2}^{m} \left(1 - \frac{1}{2^{1+n_{i-1}}} \right)^{n_i},\] so the probability that at least one neuron from layer $2$ and beyond is stably unactivated is \[\geq 1 - \prod_{i=2}^{m} \left(1 - \frac{1}{2^{1+n_{i-1}}} \right)^{n_i},\] as desired.
\end{proof}

\begin{remark}
We note that the lower bound in Lemma \ref{lem:lowbdunactive} is extremely weak unless the width is low or the depth far exceeds the width. For example, when $m = 10$ and $n_i = 10$ for all $i$, the lower bound above is about $0.05$. By the time $m = n_i = 20$, it is less than $0.0002$. However, it is significant that the lower bound is always nonzero. 
\end{remark}


\section{Precomposing with a layer and attaining the upper bound}

The primary goal of this section is to prove that for narrowing networks the upper bound on functional dimension from Theorem \ref{t:upperbound} is tight.  Our proof strategy is to use induction on the number of layers.  Subsection \ref{s:inductive} proves the inductive step (Theorem \ref{t:inductiverevised}) -- that, under certain assumptions, for $\theta \in \mathcal{P}_{n_1,\ldots,n_m}$ and $n_0 > n_1$, we have $\textrm{dim}_{\textrm{fun}}(A,\theta) = n_0n_1+\textrm{dim}_{\textrm{fun}}(\theta)$, where $(A,\theta)$ denotes the parameter in $\mathcal{P}_{n_0,\ldots,n_m}$ corresponding to a marked neural network map whose first layer is given by $A$ and whose later layers are given by $\theta$.  This inductive step is applied in subsection \ref{s:realizingbound} to prove Theorem \ref{t:upperbound}.

\subsection{Effect on functional dimension of precomposing with an additional layer map} \label{s:inductive}

We preface this subsection with a reminder that almost every parameter is ordinary and admits a full-measure set of parametrically smooth points (see \S \ref{ss:parametricallysmoothpoints}). Consequently, while the proofs of many of the results in this subsection contain careful justifications that the parameters (resp. points) under consideration are ordinary (resp. parametrically smooth), when first reading this subsection, the reader may wish to disregard these technical smoothness considerations.  

 \begin{definition}\label{d:dimlocpos} 
For any parameter $\theta \in \mathcal{P}_{n_0,\ldots,n_m}$ that has at least one parametrically smooth point in the open positive orthant $(\mathbb{R}^{> 0})^{n_0}$, define the \emph{functional dimension restricted to the positive orthant} of $\theta$  to be 
  $$\textrm{dim}_{\textrm{fun+}}(\theta) := \sup_{Z \subset (\mathbb{R}^{> 0})^{n_0} \textrm{ is parametrically smooth for }\theta,\newline |Z| \leq D(n_0,\ldots,n_m)} rank (\boldsymbol{J}E_Z(\theta)).$$
 \end{definition}

First, Lemma \ref{l:precompositionbound} shows that precomposing with a narrowing layer from $\mathbb{R}^{n_0}$ to $\mathbb{R}^{n_1}$ can increase functional dimension by at most $n_0n_1$. 

\begin{lemma} \label{l:precompositionbound} Fix any architecture $(n_0,\ldots,n_m)$.  
Let $\theta \in \mathcal{P}_{n_1,\ldots,n_m}$ be a parameter that admits a parametrically smooth point in  the positive orthant $(\mathbb{R}^{>0})^{n_1}$. Let $A:\mathbb{R}^{n_0} \to \mathbb{R}^{n_1}$ be an affine-linear map such that every row of the associated matrix has at least one nonzero entry.  Denote by $(A,\theta)$ the parameter in $ \mathcal{P}_{n_0,\ldots,n_m}$ that corresponds to precomposing with $A$, i.e. $$\bar{\rho}(A,\theta) = \bar{\rho}(\theta) \circ \sigma \circ A.$$
  If $(A,\theta)$ is an ordinary parameter,  then 
\begin{equation} \label{eq:precompineq}
\textrm{dim}_{\textrm{fun}}(A,\theta) \leq n_0n_1 + \textrm{dim}_{\textrm{fun+}}(\theta).
\end{equation}
Furthermore, a necessary condition for equality in \eqref{eq:precompineq} is that $\textrm{dim}_{\textrm{fun+}}(\theta)$ be realized on a parametrically smooth set $Z^* \subset \textrm{Image}(\sigma \circ A)$. 
\end{lemma}

\begin{proof}
 Assume $(A,\theta)$ is an ordinary parameter. 
By assumption, every row of $A$ has at least one nonzero element,  say $a_{ij_i}$. 
To prove \eqref{eq:precompineq}, we will exploit the scaling invariance used in the proof of Theorem \ref{t:upperbound} -- that if $h$ is any linear self-map of $\mathbb{R}^{n_1}$ represented by a diagonal matrix with all positive entries, then 
\begin{equation} \label{eq:scalinginvariance}
\rho(\theta) \circ \sigma \circ A = \rho(\theta) \circ h^{-1} \circ \sigma \circ h \circ A.
\end{equation}
Define $$H \coloneqq (\mathbb{R}^{>0})^{n_1} \times (\mathbb{R}^{n_0})^{n_1} \times \mathcal{P}_{n_1,\ldots,n_m}.$$ 
$H$ may be thought of as the space of parameters for $\mathcal{P}_{n_0,\ldots,n_m}$ ``up to rescaling parameters from the first layer map,'' together with the scaling factors.  

For nonzero real numbers $x$, define $\textrm{sgn}(x) = 1$ if $x>0$ and $\textrm{sgn}(x) = -1$ if $x < 0$.  
Let $U \subset \mathcal{P}_{n_0,n_1}$ be an open neighborhood of $A$ such that for any $B^{(1)} \in U$, $\textrm{sgn}(b_{ij_i}) = \textrm{sgn}(a_{ij_i})$.  Uniformly rescaling the $i$th row of $B^{(1)}=[b_{ij}]$ so that the rescaled version of $b_{ij_i}$ equals $\textrm{sgn}(a_{ij_i})$ and recording the scaling factors provides a natural encoding
$$e:U \times \mathcal{P}_{n_1,\ldots,n_m} \to H.$$  Namely, if we represent a point $B  \in U \times \mathcal{P}_{n_1,\ldots,n_m}$ as 
$B=(B^{(1)},B^{(2)},\ldots,B^{(m)})$, where $B^{(i)}$ is the matrix that describes the affine part of the $i$th layer map, we define 
$$
e(B) :=
\left( \textrm{diag}( h_B) ,  (h_{B}  B^{(1)})^* , B^{(2)} h_B^{-1}, B^{(3)},\ldots ,B^{(m)} \right)
$$
where $h_B$ is the linear self-map of $\mathbb{R}^{n_1}$ represented by a diagonal matrix whose main-diagonal entries are $|b_{ij_i}|^{-1}$ and $(h_{B} \circ B_1)^*$ is formed from the product matrix $h_{B}  B^{(1)}$ by omitting the $(i,j_i)$th entries (which are $\textrm{sgn}(a_{ij_i})$).

To distinguish evaluation maps on different spaces, we will write $E_Z^1$ for the evaluation map on $\mathcal{P}_{n_1,\ldots,n_m}$ (for $Z \subset \mathbb{R}^{n_1}$), and $E_Z^0$ for the evaluation map on $\mathcal{P}_{n_0,\ldots,n_m}$ (for $Z \subset \mathbb{R}^{n_0}$).
It is easy to see that $e$ is invertible, and hence $e$ and $e^{-1}$ have full rank everywhere.  
Hence,
$$\textrm{rank }(E_Z^0 \circ e^{-1}) \vert_{e(A,\theta)} = \textrm{rank }E_Z^0 \vert_{(A,s)}.$$

Recall that columns of  $\boldsymbol{J}( E_Z^0 \circ e^{-1})$ correspond to derivatives with respect to the coordinates of $ (\mathbb{R}^{>0})^{n_1}\times (\mathbb{R}^{n_0})^{n_1}  \times \mathcal{P}_{n_1,\ldots,n_m}$; we consider the three types of columns (corresponding to parameters in the three factors of the product space) separately. 

\medskip 
\emph{Parameters in $(\mathbb{R}^{>0})^{n_1}$: } Because $ E_Z^0 \circ e^{-1}$ is invariant under changing coordinates in $(\mathbb{R}^{>0})^{n_1}$ (by \eqref{eq:scalinginvariance}), 
 the $n_1$ columns of $\boldsymbol{J}( E_Z^0 \circ e^{-1})$ that correspond to taking derivatives with respect to coordinates of  $(\mathbb{R}^{>0})^{n_1}$ are all $\vec{0}$.  
 
 \medskip
 \emph{Parameters in $(\mathbb{R}^{n_0})^{n_1}$}:
  There are $n_0n_1$ columns that correspond to taking derivatives with respect to coordinates of $(\mathbb{R}^{n_0})^{n_1}$. Each of these columns may contribute at most $1$ to the rank of $\mathbf{J}(E_Z^0 \circ e^{-1})$. 
  
  \medskip 
  \emph{Parameters in $\mathcal{P}_{n_1,\ldots,n_m}$:} Note that the columns of $$\boldsymbol{J}( E_Z^0 \circ e^{-1} ) \vert_{e(A,\theta)}$$ that correspond to taking derivatives with respect to the coordinates of $\theta$  coincide with the columns of  $\boldsymbol{J}E^0_Z \vert_{(A,s)}$ that correspond to taking derivatives with respect to the coordinates of $s$.  Furthermore, these columns  coincide with the columns of $$\boldsymbol{J}E^1_{\sigma \circ A (Z)} \vert_s$$

Therefore, for any set $Z \subset \mathbb{R}^{n_0}$ that is parametrically smooth for $(A,\theta)$, 
 \begin{equation} \label{eq:onedirection}
  \textrm{rank } E_Z^0 \vert_{(A,\theta)} = \textrm{rank }(E_Z^0 \circ e^{-1}) \vert_{e((A,\theta))}  \leq n_0n_1 + \textrm{rank } E_{\sigma \circ A (Z)}^1 \vert_\theta.
  \end{equation}

 Taking the supremum over all finite sets $Z \subset \mathbb{R}^{n_0}$ that are parametrically smooth for $(A,\theta)$ thus yields 
 \begin{multline*} \textrm{dim}_{\textrm{fun}}(A,\theta) \leq n_0n_1 + \\ \textrm{sup} \left\{\textrm{rank }  E_{Z^*}\vert_\theta : Z^* \subset \textrm{Image}(\sigma \circ A) \textrm{ is parametrically smooth for } \theta\right\}. \end{multline*}
\end{proof}

As was shown in Lemma \ref{l:qrealizable}, data from the evaluation map $E_Z$ (for suitable $Z$) can be used to detect the hyperplane $H \subset \mathbb{R}^{n_0}$ that contains a shared codimension 1 face of two top-dimensional cells with different slopes.  Because functional dimension involves parameters  (i.e. marked - as opposed to unmarked - functions), we will need to compute not only the geometric set $H$ but the coefficients of the affine-linear equations for which $H$ is the solution set -- i.e. if $H = \{x \in \mathbb{R}^{n_0} :  A \widehat{x} = \vec{0} \}$ for some matrix $A$ of weights and biases, we will need to express the entries of the matrix $A$ (up to rescaling rows) in terms of $E_Z(\theta)$. Furthermore, we want each entry (up to scaling) of $A$ to be a differentiable (e.g. algebraic) function of the coordinates of the vector $E_Z(\theta)$.

\begin{lemma} \label{l:candetecthyperplanes}
Let $M$ be a polyhedral complex embedded in $\mathbb{R}^d$, $d \geq 1$, and let $F:\mathbb{R}^{d} \to \mathbb{R}^{n_m}$ be a continuous map that is affine-linear on cells of $M$.  Let $X,Y$ be two 
 $d$-dimensional cells of $M$ that share a $(d-1)$-dimensional facet, and denote the hyperplane that contains this shared facet by $H$.  Assume $\boldsymbol{J}F\vert_X \neq \boldsymbol{J}F \vert_Y$.
Then, for any decisive sets $S_X \subset X$ for $F\vert_X$ and $S_Y \subset Y$ for $F\vert_Y$,  $H$ is the solution set to an affine-linear equation, 
$$H = \{\vec{x} : \vec{c} + A \vec{x} = \vec{0} \}$$
where every entry of the matrix $A$ is an affine-linear expression in the variables that are the coordinates of 
$E_{S_X \cup S_Y}(F)$. Furthermore, the matrix $A$ is unique up to rescaling rows by constants. 
\end{lemma}

\begin{proof}
Pick $S_X$ to be a decisive set in $X$ for $F\vert_X$ (meaning the points of $S_X$ are the vertices of a $d$-dimensional simplex in $X$).  Denote the points of $S_X$ by $z_0,\ldots,z_d$.  Then, each of the directional derivatives $D_{\overline{z_0z_i}}F(z_0)$, for $i = 1,\ldots,d$, is the difference between the two coordinates of $E_{S_X}(F)$ corresponding to $z_0$ and $z_i$, scaled by $1/|z_0 -z_i|$.  Since the vectors $\overline{z_0z_i}$ span $\mathbb{R}^d$, each of the partial derivatives in the directions of the standard coordinate axes, $\frac{\partial F}{\partial x_i}$, can be written as a linear combination of the directional derivatives  $D_{\overline{z_0z_i}}F(z_0)$.  Thus, each entry of the matrix $\boldsymbol{J}F\vert_x$ for $x \in X$ is given by a linear combination of the coordinates of the vector $E_{S_X}(F)$.  Similarly, picking $S_Y$ to be a decisive set in $Y$ for $F \vert_Y$, each entry of the matrix $\boldsymbol{J}F\vert_y$ for $y \in Y$ is given by a linear combination of the coordinates of the vector $E_{S_Y}(F)$.

Denote by $\textrm{ex}(F\vert_X):\mathbb{R}^d \to \mathbb{R}^{n_m}$ the affine-linear extension of the restricted map $F\vert_X : X \to \mathbb{R}^{n_m}$ to all of $\mathbb{R}^d$, and define $\textrm{ex}(F\vert_Y)$ similarly.  Then $\textrm{ex}(F\vert_X)$ (resp. $\textrm{ex}(F\vert_Y)$) can be written as $\vec{x} \mapsto \vec{c}_X + \boldsymbol{J}F\vert_X \vec{x} $ for some constant vector $\vec{c}_X$ (resp. $\vec{c}_Y$). 
The hyperplane $H$ containing the shared $(d-1)$-dimensional facet is given by 
 $$H = \{\vec{x} : \textrm{ex}(F\vert_X)\vec{ x}= \textrm{ex}(F\vert_Y)\vec{x} \}= \{ \vec{x} : \vec{c}_X - \vec{c}_Y + (\boldsymbol{J}F\vert_X - \boldsymbol{J}F\vert_Y)\vec{x} = \vec{0} \}.$$
 That is, $H$ is the set of points $(x_1,\ldots,x_d) \in \mathbb{R}^d$  such that for every coordinate $1 \leq i \leq n_m$, some affine-linear expression in the variables that are the coordinates of the vector $E_{S_X \cup S_Y}(F)$ assumes the value $0$. 
   \end{proof}
   
   \begin{remark}
Lemma \ref{l:candetecthyperplanes} may be interpreted as a description of a dual space to the space of functions in a neighborhood of a combinatorially stable 
function $\rho(\theta)$.  Namely, a point in the dual space consists of an assignment of a value in $\mathbb{R}^{n_m}$ and a total derivative (in $\mathbb{R}^{n_m} \times \mathbb{R}^{n_0}$) to each top-dimensional cell of $\mathcal{C}(\theta)$.   
   \end{remark}

Next, Lemma \ref{l:Aassumptions} shows that if various assumptions are satisfied, then there exists a batch $Z \subset \mathbb{R}^{n_0}$ that is suitable for detecting all relevant hyperplanes from the data $E_Z(A,\theta)$.  Note that (by Lemma \ref{l:opennbhd1}) if ternary labeling for $x$ contains no $0$s, then  $x$ is a parametrically smooth point for $\theta$.  

\begin{lemma}  \label{l:Aassumptions}
Fix an ordinary parameter $\theta \in \mathcal{P}_{n_1,\ldots,n_m}$.
 Let $A$ be an affine-linear map $ A: \mathbb{R}^{n_0} \to \mathbb{R}^{n_1}$ (equivalently an $n_1 \times (n_0+1)$ matrix).  For any point $y \in \mathbb{R}^{n_0}$, denote the ternary label of $y$ with respect to the $k$th row of $A$ by $s^A_k(y)$; for any point $y \in \mathbb{R}^{n_1}$, denote the ternary label of $y$ with respect to the $j$th neuron of the $i$th layer of $\bar{\rho}(\theta)$ by $s^i_j(y)$.  
Assume that $\{x \in \mathbb{R}^{n_0} \mid s^A_k(x) = 0\}$ is a hyperplane for all $1 \leq k \leq n_1$.  
Suppose that for each $1 \leq k \leq n_1$, there exists a point $y_k \in \mathbb{R}^{n_0}$ such that 
\begin{enumerate}
\item $s_k^A(y_k) = 0$, 
\item $s_j^A(y_k) \neq 0$ for all $j \neq k$, 
\item the full ternary label of the point $\sigma \circ A(y_k)$ with respect to $s$ has no $0$s,

\item \label{i:clarifyJmeaning} the $k$th column of $\boldsymbol{J}\rho(\theta)\vert_{\sigma \circ A (y_k)}$ is nonzero (as a column vector),
\end{enumerate}
Then there exists a finite set $Z \subset \mathbb{R}^{n_0}$ such that, 

\begin{enumerate}
 \setcounter{enumi}{4}
\item up to scaling rows of $A$ by positive numbers, each entry of $A$ is given by a unique affine-linear combination of the coordinates of the vector $E_Z(A,\theta)$.  
\item the ternary label of every point in $Z$ with respect to every neuron of $(A,\theta)$ is nonzero.  
\end{enumerate}

\noindent If, additionally,  $y_k \in \left(\mathbb{R}^{>0} \right)^{n_0}$ for every $k$, then 
\begin{enumerate} 
 \setcounter{enumi}{6}
\item the set $Z$ can be chosen to be a subset of $(\mathbb{R}^{>0})^{n_0}$. 
\end{enumerate}
\end{lemma}

\begin{remark} Note that $\rho(s): \mathbb{R}^{n_1} \to \mathbb{R}^{n_m}$, so $\mathbf{J}\rho(\theta)$ in item \eqref{i:clarifyJmeaning} is a matrix of partial derivatives with respect to the coordinates of $x \in \mathbb{R}^{n_1}$, not $\theta$.  Since $\rho(\theta)$ is affine-linear in $x$ on $C_k$, $\mathbf{J}\rho(\theta)$ is constant on $C_k$. \end{remark}

\begin{proof}
Conditions (i) and (ii) imply that $y_k$ is in the interior of a codimension $1$ cell of the canonical polyhedral decomposition $\mathcal{C}(A,\theta)$ of $\mathbb{R}^{n_0}$, and there is an open neighborhood $U_k$ of $y_k \subset \mathbb{R}^{n_0}$ such that no other hyperplane associated to $A$ intersects $U_k$.  Consequently, the ternary labels $s_j^A$ are constant on $U_k$ for all $j \neq k$. 
  Condition (iii) implies that $\sigma \circ A(y_k)$ is in the interior of some top-dimensional cell $C_k$ of  $\mathcal{C}(\theta)$. By condition (iii), by making $U_k$ small enough, we may assume that $\sigma \circ A (U_k) \subset C_k$, so that the ternary labeling with respect to each neuron of $\bar{\rho}(A,\theta)$ except the one corresponding to the $k$th row of $A$ is constant on $U_k$.   

The set $U_k \setminus \{x \in \mathbb{R}^{n_0} \mid \theta^A_k(x) =0 \}$ consists of two connected components; denote these components by $U_k^+$ and $U_k^-$ according to the sign of $\theta^A_k$ on these sets. Then, the ternary label with respect to every neuron of $(A,\theta)$ is constant and nonzero on each of $U^+_k$ and $U_k^-$.

 Since $\rho(\theta)$ is affine-linear on $C_k$, the restriction of $\rho(\theta) \circ \sigma \circ A$ to each of $U_k^+$, $U_k^-$ is an affine-linear map.  
We wish to use Lemma \ref{l:candetecthyperplanes} to obtain the set $Z$;  in order to satisfy the assumptions of the lemma, we must show that for each $k$
\begin{equation} \label{eq:differentJacs}
\boldsymbol{J}(\rho(\theta) \circ \sigma \circ A) \vert_{U_k^+} \neq \boldsymbol{J}(\rho(\theta) \circ \sigma \circ A) \vert_{U_k^-}.
\end{equation}  

  By the chain rule, for any point $z \in \mathbb{R}^{n_0}$,
\begin{equation} \label{eq:chainrule}
\boldsymbol{J}(\rho(\theta) \circ \sigma \circ A)\vert_z  = \boldsymbol{J} (\rho(\theta)) \vert_{\sigma \circ A (z)}    \boldsymbol{J} (\sigma \circ A) \vert_z.
\end{equation}
Since $\sigma \circ A (U_k) \subset C_k$ and $\rho(\theta)$ is affine-linear on $C_k$, we have that 
$ \boldsymbol{J}(\rho(\theta)) \vert_{\sigma \circ A (z)}$  does not depend on $z \in U_k$.  Set

$$\widetilde{A}_k^+ \coloneqq \boldsymbol{J}(\sigma \circ A \vert_{U_k^+}), \quad \widetilde{A}_k^- \coloneqq \boldsymbol{J}(\sigma \circ A \vert_{U_k^-}).$$
Observe that $\widetilde{A}_k^+$ is formed from the matrix $A$ by omitting the last bias column, and possibly zeroing out some rows (corresponding to neurons that are off on $U_k^+$).  The assumption that each row of $A$ determines a hyperplane implies that the $k$th row of $\widetilde{A}_k^+$ contains a nonzero element.  Also $\widetilde{A}_k^-$ agrees with $\widetilde{A}_k^+$ except in the $k$th row, in which every entry of $\widetilde{A}_k^-$ is $0$. 

Rewriting \eqref{eq:chainrule}, we have 
\begin{align*}  
\boldsymbol{J}(\rho(\theta) \circ \sigma \circ A) \vert_{U_k^+}  & =  \boldsymbol{J}(\rho(\theta)) \vert_{C_k} \cdot \widetilde{A}_k^+ \\
\boldsymbol{J}(\rho(\theta) \circ \sigma \circ A) \vert_{U_k^-} &   =  \boldsymbol{J}(\rho(\theta)) \vert_{C_k} \cdot \widetilde{A}_k^-  \\
\end{align*}
By condition (iv), the $k$th column of $\boldsymbol{J}(\rho(\theta)) \vert_{C_k}$ has a nonzero element.  Since the $k$th row of $\widetilde{A}_k^+$ also has a nonzero element, while the $k$th row of $\widetilde{A}_k^-$ is uniformly $0$, the two matrix products above are not equal, establishing \eqref{eq:differentJacs}.

Now, Lemma \ref{l:candetecthyperplanes} guarantees that for any decisive sets $S_k^+ \subset U_k^+$ and $S_k^- \subset U_k^-$, up to rescaling rows, there is a unique matrix that defines $H$, and its entries are affine-linear expressions in the variables that are the 
the coordinates of the vector $E_{S_k^+ \cup S_k^-}(A,\theta)$.  Set $Z = S_k^+ \cup S_k^-$.  The set $Z$ satisfies conclusions (v) and (vi) by construction.

If  $\cup_k \{y_k\} \subset (\mathbb{R}^{>0})^{n_0}$, then since $Z$ is contained in an arbitrarily small neighborhood of $\cup_k \{y_k\}$, conclusion (vii) follows. 
\end{proof}

Theorem \ref{t:inductiverevised} will serve as the inductive step in a  constructive proof showing that for narrowing architectures, functional dimension attains the upper bound from Theorem \ref{t:UpperBdTightForNarrowingNetworks}.

\begin{theorem} \label{t:inductiverevised}
Fix an ordinary parameter $\theta \in \mathcal{P}_{n_1,\ldots,n_m}$.   Suppose $Z_1 \subset (\mathbb{R}^{>0})^ {n_1}$ is a finite set of points
whose ternary labels with respect to every neuron of $s$ are nonzero and 
$$\textrm{dim}_{fun}(\theta) =  \textrm{rank}   \ E_{Z_1} \vert_\theta.$$
Suppose  $A:\mathbb{R}^{n_0} \to \mathbb{R}^{n_1}$ is a surjective affine-linear map that satisfies all assumptions of Lemma \ref{l:Aassumptions} (including that $\cup_k\{y_k\} \subset (\mathbb{R}^{>0})^{n_0}$).  
Then there exists a finite set $Z \subset (\mathbb{R}^{>0})^{n_0}$ such that 
the ternary label of every point in $Z$ with respect to every neuron of $(A,\theta)$ is nonzero and 
$$\textrm{dim}_{fun}(A,\theta) = \textrm{rank }E_Z\vert_{(A,\theta)} = n_0n_1 + \textrm{dim}_{\textrm{fun}}(\theta).$$
\end{theorem}

\begin{proof}
Because $A$ is surjective and $Z_1 \subset (\mathbb{R}^{>0})^{n_1}$, we may pick a set $Z_0 \subset (\mathbb{R}^{>0})^{n_0}$ that consists of one preimage under $\sigma \circ A$ of each point in $Z_1$.  Then the ternary label of each point in $Z_0$ with respect to each neuron of $A$ is $+1$; moreover, the ternary label of each point of $Z_0$ with respect to every neuron of $(A,\theta)$ is nonzero.

By Lemma \ref{l:Aassumptions}, there exists a finite set $Z_* \subset (\mathbb{R}^{>0})^{n_0}$ such that the ternary label of every point of $Z_*$ with respect to every neuron of $(A,\theta)$ is nonzero, and 
 each element of $A$ (up to scaling rows) is given by an affine-linear function in the variables that are the coordinates of $E_{Z_*}(A,\theta)$. 
 
  We will show that the set $Z = Z_0 \cup Z_*$ satisfies $$\textrm{dim}_{fun}(A,\theta) = \textrm{rank }E_Z\vert_{(A,\theta)} = n_0n_1 + \textrm{dim}_{\textrm{fun}}(\theta).$$
By Lemma \ref{l:precompositionbound}, $\textrm{dim}_{fun}(A,s) \leq n_0n_1 + \textrm{dim}_{\textrm{fun}}(s)$, so it suffices to prove that 
\begin{equation} \label{eq:inductivegoal}
\textrm{rank } E_{Z_0 \cup Z_*} \vert_{(\theta, A)}   \geq \textrm{rank } E_{Z_1}\vert_\theta + n_0n_1.
\end{equation}

Let $U$ be a small open neighborhood of $A$ in $\mathcal{P}_{n_0n_1}$. Without loss of generality (by making $U$ smaller, if necessary), by assumption  \eqref{i:hitsimportantpoints}, we may assume $Z_1$ is in the interior of $\textrm{Image}(\sigma \circ  u)$ for all $u \in U$.   By assumption, each row of $A$ has a nonzero element -- say in the $j_i$th element of row $i$ of $A$.  Again without loss of generality, we may assume that each $u \in U$  has nonzero entries in the $j_i$th element of row $i$, for each $i$.

 For any $u \in U$ and $\theta' \in \mathcal{P}_{n_1,\ldots,n_m}$, the first  $n_m|Z_0|$ coordinates of the evaluation data $E_{Z_0 \cup Z_*}(u,\theta') \in \mathbb{R}^{n_m(|Z_0| + |Z_*|)}$ represent $E_{Z_0}(u,\theta')$ and the last $n_m|Z_*|$ coordinates represent $E_{Z_*}(u,\theta')$. We will define the map 
$$\tau: E_{Z_0 \cup Z_*}(U \times \mathcal{P}_{n_1,\ldots,n_m}) \to \mathbb{R}^{n_m|Z_0|} \times (\mathbb{R}^{n_0})^{n_1}$$
to be the map that ``keeps" the data $E_{Z_0 \cup Z_*}(u,\theta')$, 
computes the entries of $A$ under the choice (i.e. scaling rows of $A$) that the $(i,j_i)$th entry of $A$ equals $1$ for each $i$, and then records the non-$(i,j_i)$-indexed entries of $A$.  Rigorously:
\begin{itemize}
\item  Define $\tau$ to be the identity on the first $n_m|Z_0|$ coordinates --  that is, for $u \in U$ and $\theta' \in \mathcal{P}_{n_1,\ldots,n_m}$, define the  $\mathbb{R}^{n_m|Z_0|}$ coordinate of $\tau(E_{Z_0 \cup Z_*}(u,\theta))$ to be $E_{Z_0}(u,\theta)$.
\item By Lemma \ref{l:candetecthyperplanes}, for each index $(i,j)$, there is a real-valued, affine-linear function $a_{i,j}$ such that, if $A'$ is any matrix such that 
$$\mathbf{J}E_{Z_*}(A,\theta) = \mathbf{J}E_{Z_*}(A',\theta),$$
then for each $i$ the $i$th row of $A'$ is given by 
$$\left [ r_i a_{i,1}(\mathbf{J}E_{Z_*}(A,\theta)), \ldots, r_i a_{i,n_0}(\mathbf{J}E_{Z_*}(A,\theta)), a_{i,n_0+1}(\mathbf{J}E_{Z_*}(A,\theta))
 \right ] $$
for some nonzero row-scaling-factor $r_i \in \mathbb{R}$. Define $B(E_{Z_*}(u,\theta'))$ to be the matrix computed in this way using the row-scaling-factors $r_i = r_{i,j_i}(\mathbf{J}E_{Z_*}(u,\theta'))^{-1}$, so that $j_i$th coordinate of the $i$th row of $B$ equals $1$, for every row index $i$. Define the  $(\mathbb{R}^{n_0})^{n_1}$ coordinate of $\tau: E_{Z_0 \cup Z_*}(u,\theta')$ to be the vector formed by unrolling the matrix $B \left(\mathbf{J}E_{Z_*}(u,\theta')\right)$ and then dropping the coordinates corresponding to all the indices $(i,j_i)$ (i.e. dropping all the entries that we explicitly forced to be $1$ by scaling rows).

\end{itemize}

We now consider the composition $$\tau \circ E_{Z_0 \cup Z_*}:   U \times \mathcal{P}_{n_1,\ldots,n_m}    \to   \mathbb{R}^{n_m|Z_0|} \times (\mathbb{R}^{n_0})^{n_1}.$$
Because all the maps $a_{i,j}$ are affine-linear, and the $a_{i,j_i}$s (the denominators of the scaling factors $r_i$) are nonzero on $U$,  $\tau$ is (at least locally at $(A,\theta)$) a smooth 
 map between differentiable manifolds. We will show that 
$$\textrm{rank } \tau  \circ E_{Z_0 \cup Z_*}\vert_{(\theta , A)} = \textrm{rank } E_{Z_1}\vert_\theta +n_0n_1.$$
After suitably permuting its rows and columns, the matrix $\boldsymbol{J}(\tau \circ E_{Z_0 \cup Z_*})\vert_{(\theta,A)}$ can be written as a block matrix 
 $$
    \left[
    \begin{array}{c|c}
\alpha & \beta \\
\hline
\gamma & \delta
\end{array}
\right]
$$
where 
\begin{itemize}
\item rows of $\alpha$ and $\beta$ correspond to coordinates in  $\mathbb{R}^{n_m|Z_0|}$,
\item rows of $\gamma$ and $\delta$ correspond to coordinates in $(\mathbb{R}^{n_0})^{n_1}$, 
\item columns of $\alpha$ and $ \gamma$ correspond to taking  partial derivatives with respect to parameters in $\mathcal{P}_{n_1,\ldots,n_m}$ (i.e. those that are coordinates of $\theta$), and
\item  columns of $\beta$ and $\delta$ correspond to taking partial derivatives with respect to parameters in $\mathcal{P}_{n_0n_1}$ (coordinates of $u$). 
\end{itemize}
By construction, 
$$\alpha =\boldsymbol{J} E_{Z_0}\vert_{(\theta, A)} = \boldsymbol{J}  E_{Z_1}\vert_s$$
 and 
$\gamma =\boldsymbol{ 0}.$
Since $\gamma = \boldsymbol{0}$ means 
$$\boldsymbol{J} (\tau \circ E_{Z_0 \cup Z_*})\vert_{(\theta, A)}$$
 is an upper triangular block matrix, it follows that

\begin{equation} \label{eq:intermedstep}
\textrm{rank } (\tau \circ E_{Z_0 \cup Z_*}) \vert_{(\theta, A)}= 
\textrm{rank }(\alpha) + \textrm{rank }(\delta) = \textrm{rank } E_{Z_1}\vert_\theta + \textrm{rank}(\delta).\end{equation}

But what is $\delta$? $\delta$ is the derivative of the map which is scaling the entries of $A$ so that the privileged element (the $j_i$-th element of the $i$-th row) has norm $1$.  More precisely,
a row of $\delta$ corresponds to looking at one of the $n_0n_1$-many non-$(i,j_i)$ indices of elements of $u$ -- say $(s,t)$ -- and recording (as the entries of the row vector) the  partial derivatives (with respect to each of the $n_1(n_0+1)$-many entries of $u$) of the map $[u_{i,j}] \mapsto \frac{u_{(s,t)} }{|u_{s,j_s}| }$.  It follows that the rank of $\delta$ is $n_0n_1$. 
Thus, from \eqref{eq:intermedstep}, we have 
\begin{equation}
\label{eq:factorrankachieves}
\textrm{rank }  (\tau  \circ E_{Z_0 \cup Z_*}) \vert_{(\theta , A)} = \textrm{rank } E_{Z_1}\vert_\theta +n_0n_1.
\end{equation}

By the chain rule,
\begin{equation} \label{eq:chainrule} 
\boldsymbol{J}(\tau \circ E_{Z_0 \cup Z_*})\vert_{(\theta, A)} = \boldsymbol{J} \tau \vert_{E_{Z_0 \cup Z_*}(\theta, A)} \boldsymbol{J} E_{Z_0 \cup Z_*}\vert_{(\theta, A)}).  
\end{equation}
Since the rank of a composition of linear maps is at most the minimum of the ranks of its factors,  equations \eqref{eq:factorrankachieves} and \eqref{eq:chainrule} together imply
$$
\textrm{rank } E_{Z_0 \cup Z_*}\vert_{(\theta,A)} \geq  \\
 \textrm{rank } (\tau \circ E_{Z_0 \cup Z_*}) \vert_{(\theta,A)} = \textrm{rank } E_{Z_1}\vert_\theta +n_0n_1.
$$
\end{proof}

\subsection{Tightness of the bound for narrowing architectures} \label{s:realizingbound}

In this section, we use Theorem \ref{t:inductiverevised} to prove that for narrowing networks, the upper bound on functional dimension from Theorem \ref{t:upperbound}  is tight.

\begin{definition} An architecture $(n_0,n_1,\ldots,n_m)$ is \emph{narrowing} if $n_i > n_{i+1}$ for all $i = 0,\ldots,m-1$. 
\end{definition}

Restricting our attention to narrowing architectures enables us to construct layer maps that send $(\mathbb{R}^{>0})^{n_i}$ surjectively onto $\mathbb{R}^{n_{i+1}}$ (Lemma  \ref{l:AExists}). 

 \begin{remark}
An example of a surjective affine map $A:(\mathbb{R}^{>0})^{n_0} \to \mathbb{R}^{n_1}$ is the following.  Let $n_0= 2$, $n_1 = 1$, and let $A$ be the map that is the dot product with the vector $(-1,1)$, i.e. the map that measures signed distance from a point $(x,y) \in (\mathbb{R}^{>0})^2$ to the line $y=x$. 
\end{remark}

The point of Lemma \ref{l:AExists} is to guarantee that a layer map $A$ satisfying the assumptions of the inductive step (Theorem \ref{t:inductiverevised}) exists. 

\begin{lemma} \label{l:AExists}
Fix natural numbers $n_0 > n_1$.  Fix an architecture $n_1,\ldots,n_m$ and an ordinary parameter $\theta \in \mathcal{P}_{n_1,\ldots,n_m}$ such that 
\begin{enumerate}[label=(\alph*)]
\item the ternary label of $\vec{0} \in \mathbb{R}^{n_1}$ with respect to $\theta$ has no $0$s,
\item $\textrm{dim}_{\textrm{fun}}(\theta) = \textrm{dim}_{\textrm{fun,+}}(\theta)$.
\item each column of $\boldsymbol{J} \rho(\theta)\vert_{\vec{0}}$ is nonzero (as a column vector),
\end{enumerate}
Then there exists a surjective affine-linear map $A:\mathbb{R}^{n_0} \to \mathbb{R}^{n_1}$ such that 
\begin{enumerate} 
\item the ternary label of $\vec{0} \in \mathbb{R}^{n_0}$ with respect to $(A,\theta)$ has no $0$s,
\end{enumerate}
 and $A$  satisfies the assumptions of Lemma \ref{l:Aassumptions} (equivalently, of Theorem \ref{t:inductiverevised}), i.e. 
\begin{enumerate}
 \setcounter{enumi}{1}
\item \label{i:nonzerorows} each row of $A$ determines a hyperplane in $\mathbb{R}^{n_0}$.
\item  \label{i:hitsimportantpoints} $(\mathbb{R}^{\geq 0})^{n_1} \subseteq \textrm{Image}(\sigma \circ A)$,
\end{enumerate}
and for each $1 \leq k \leq n_1$, there exists a point $x_k \in (\mathbb{R}^{> 0})^{n_0}$ such that 
\begin{enumerate}
\setcounter{enumi}{3}
\item \label{i:ptexists} $s^A_k(x_k) = 0$, 
\item \label{i:trueface} $s^A_j(x_k) \neq 0$ for all $j \neq k$. 
\item the ternary label of  $\sigma \circ A(x_k) \neq 0$ with respect to $\theta$ has no $0$s. 
\end{enumerate}

 \end{lemma}

 \begin{proof}
 Conditions \eqref{i:nonzerorows} and  \eqref{i:hitsimportantpoints}  hold for Lebesgue almost-every matrix $A$. If we ignore the requirement that the points $x_k$ belong to the positive orthant, conditions \eqref{i:ptexists} and \eqref{i:trueface} also hold for a full measure set of matrices; with the requirement that $x_k$ belong to the positive orthant, conditions \eqref{i:ptexists} and \eqref{i:trueface} are satisfied by a nonempty open set of matrices.  Hence, a positive measure set of matrices satisfies conditions \eqref{i:nonzerorows}-\eqref{i:trueface}.  
If a matrix $A$ satisfies conditions \eqref{i:nonzerorows}-\eqref{i:trueface}, a matrix $A'$ obtained by rescaling a row of $A$ will also satisfy  conditions \eqref{i:nonzerorows}-\eqref{i:trueface}.  Since the ternary label of  $\vec{0} \in \mathbb{R}^{n_1}$ has no $0$s, there is an open neighborhood $V \subset \mathbb{R}^{n_1}$ of $\vec{0}$ on which the ternary labels with respect to $s$ also all have no $0$s.  
We may rescale rows of $A$ so that all the points 
 $A(x_k)$ and $\sigma \circ A(\vec{0})$ are in $V$.  Hence, there exists a positive measure set of matrices $A$ satisfying all six conditions.  \end{proof}

 \begin{theorem} \label{t:UpperBdTightForNarrowingNetworks}
  For any narrowing architecture $(n_0,\ldots,n_m)$, the upper bound from Theorem \ref{t:upperbound} is tight, 
  i.e. $$\textrm{dim}_{\textrm{fun}} \left(\mathcal{P}_{n_0,\ldots,n_m}\right) =  n_m+\sum_{i=0}^{m-1}n_in_{i+1}.$$
\end{theorem}

We later (Lemma \ref{l:upperboundnottight}) exhibit a specific family of architectures for which the bound is not tight; so far we have not found any other architectures that do not realize the upper bound.

 \begin{proof}[Proof of Theorem \ref{t:UpperBdTightForNarrowingNetworks} ]

The base case of the inductive argument is an architecture consisting of a single layer.  Such architectures are treated in subsection \ref{ss:arcdep1}; in particular, by Proposition \ref{p:depth1}, there exists a parameter $\theta \in \mathcal{P}_{n_{m-1},n_m}$ such that $\textrm{dim}_{fun}(\theta) = n_m(n_{m-1}+1)$ and a set $Z \subset \mathbb{R}^{n_1}$ that admits an open neighborhood of parametrically smooth points for $\theta$ and such that $\textrm{dim}_{\textrm{fun}}(\theta)$ is realized on $Z$.  Furthermore, it is clear from the proof of Proposition \ref{p:depth1} that set $Z$ can be chosen in $(\mathbb{R}^{>0})^{n_1}$ and $\theta$ can be chosen so the ternary label of $\vec{0}$ for $\theta$ has no $0$s.
 
 The result then follows via induction on the number of layers, using Theorem \ref{t:inductiverevised} and Lemma \ref{l:AExists} for the inductive step.   
 \end{proof}

\section{Other architectures} \label{s:otherarchitectures}

\subsection{Architecture $(1,1,\ldots,1)$} 

\begin{theorem} \label{t:nottight} \ 
\begin{enumerate} 
 \item The upper bound on $\textrm{dim}_{\textrm{fun}}(\mathcal{P}_{n_0,\ldots,n_m})$ given by Theorem \ref{t:upperbound} is tight for architectures
  $(1,1)$, $(1,1,1)$ and $(1,1,1,1)$. 
\item Let $(1,\ldots,1)$ be any sequence of at least 5 consecutive $1$s.  Then the upper bound on $\textrm{dim}_{\textrm{fun}}(\mathcal{P}_{1,\ldots,1})$ given by Theorem \ref{t:upperbound} is not tight.  Specifically, 
 $$\textrm{dim}_{\textrm{fun}} (\mathcal{P}_{1,\ldots,1}) = 4.$$
\end{enumerate}
\end{theorem}

In this section, we will use the following classifications to describe certain continuous, piecewise affine-linear functions $f:\mathbb{R} \to \mathbb{R}$:

\begin{enumerate}
\item[Type 1:] $f$ is a constant function.
\item[Type 2:] $f$ has two pieces, say $(-\infty,\alpha]$ and $[\alpha,\infty)$. $f$ is constant on  $(-\infty,\alpha]$ and has positive slope on $[\alpha,\infty)$.
\item[Type 3:] $f$ has two pieces, say $(-\infty,\alpha]$ and $[\alpha,\infty)$. $f$ has negative slope on $(-\infty,\alpha]$ and is constant on $[\alpha,\infty)$. 
\item[Type 4:] $f$ has three pieces, say $(-\infty,\alpha]$, $[\alpha,\beta]$, and $[\beta,\infty)$, for $\alpha < \beta$.  $f$ is constant on $(-\infty,\alpha]$, has positive slope on $[\alpha,\beta]$, and is constant on $[\beta,\infty)$. 
\item[Type 5:] $f$ has three pieces, say $(-\infty,\alpha]$, $[\alpha,\beta]$, and $[\beta,\infty)$, for $\alpha < \beta$. $f$ is constant on $(-\infty,\alpha]$, has negative slope on $[\alpha,\beta]$, and is constant on $[\beta,\infty)$. 
\end{enumerate}

To prove that the bound is tight for architectures $(1,1)$, $(1,1,1)$ and $(1,1,1,1)$, we begin by observing that if $f:\mathbb{R} \to \mathbb{R}$ is given by a single layer map (i.e. is of architecture $(1,1)$), then $f$ is necessarily of type $1$, $2$ or $3$.    In particular, if $f$ is given by $$f(x) = \sigma(ax + b),$$ any choice of $a > 0$ and $b > 0$ makes $f$ of type $2$.  Such an $f$ has two degrees of freedom:  the bend point $-\tfrac{b}{a}$ and the slope $a$. The slope $a$ can take on any positive value (given our constraint) and then the bend point $-\tfrac b a$ can take on any negative value by judicious choice of $b$.  By Theorem \ref{t:equivalentdef}, this shows $\textrm{dim}_{\textrm{fun}} (\mathcal{P}_{1,1}) \geq 2$.  Since $2$ is also the upper bound on $\textrm{dim}_{\textrm{fun}} (\mathcal{P}_{1,1})$ by Theorem \ref{t:upperbound}, this proves: 
\begin{lemma} \label{l:11}
$\textrm{dim}_{\textrm{fun}} (\mathcal{P}_{1,1}) = 2$. 
\end{lemma}

Now, consider architecture $(1, 1, 1)$, with the parametrization
$$f(x) = \sigma (c \sigma (ax + b)  + d).$$
We may choose any $a,b > 0$ as above.  Then, choosing any $d > 0$ and $c < 0$ ensures $f$ has type $5$. Furthermore, the graph of $f$ is constantly $d$ from $-\infty$ to $- \tfrac b a$, then a line of slope 
$ca < 0$ from $-\tfrac b a$ to $- \tfrac b a - \tfrac{d}{ca}$, then constantly $0$ from $-\tfrac{b}{a} - \tfrac{d}{ca}$ to $\infty$.
 There are three degrees of freedom for $f$ subject to these constraints: the two bend points and the height $d$, each of which can be realized by judicious choice of $a, b, c, d$.  The first bend point can be any negative number, the second bend point can be any point greater than the first, and the height $d$ can be any positive number.
By Theorems \ref{t:equivalentdef} and \ref{t:upperbound}, this shows:
\begin{lemma} \label{l:111}
$\textrm{dim}_{\textrm{fun}} (\mathcal{P}_{1,1,1}) = 3$. 
\end{lemma} 

Now, consider architecture $(1, 1, 1, 1)$, as parametrized by  
$$F(x) = \sigma (e \sigma (c \sigma (ax + b)  + d) +f).$$  For any $a,b,d > 0$ and $c < 0$, the composition up until the final layer map (call this $F_0$) is a map of type $5$, as above.  We will specify $e, f >0$.  Any such map  $F$ is also of type 5.  $F$ has bend points at the same two points as $F_0$ does.  The graph of $F$ is constant on the piece from $-\infty$ to the left-hand (lesser) bend point, and also constant from $\infty$ to the right-hand bend point.  The height of the left segment is $de + f$ and the height of the right segment is $f$.  Thus there are  four degrees of freedom, so Theorem \ref{t:equivalentdef} and \ref{t:upperbound} imply:

\begin{lemma} \label{l:1111}
$\textrm{dim}_{\textrm{fun}} (\mathcal{P}_{1,1,1,1}) = 4$. 
\end{lemma}

\begin{lemma} \label{l:typespossible}

Let $f:\mathbb{R} \to \mathbb{R}$ be a piecewise affine-linear function.  If $f$ is of one of types $1$-$5$, then for any affine-linear map $A:\mathbb{R} \to \mathbb{R}$, the composition $\sigma \circ A \circ f$ is also of one of types $1$-$5$.  More specifically, if $f$ is of type $X$, then the possible types of the composition $\sigma \circ A \circ f$ is as follows:
\begin{center}
 \begin{tabular}{|l | l |} 
 \hline
 Type of $f$ & Possible type(s)  of \  $\sigma \circ A \circ f$  \\
 \hline
1 &  1 \\
2 & 2, 1 or 5 \\
3 & 3,1 or 4 \\
4 & 4,1 or 5 \\ 
5 & 5,1 or 4 \\
 \hline
\end{tabular}
\end{center}

\end{lemma}

The proof of Lemma \ref{l:typespossible}  follows by specific analysis of the various types.  

\begin{lemma} \label{l:upperboundnottight}
Let $(1,\ldots,1)$ be a sequence of at least $5$ $1$s. 
Then $$\textrm{dim}_{\textrm{fun}}(\mathcal{P}_{1,\ldots,1}) = 4.$$ 
\end{lemma}

\begin{proof}
By Lemma \ref{l:typespossible}, any function $f \in \mathcal{M}_{1,\ldots,1}$ is a function of one of types $1$-$5$.   As the reader may verify, any function of type $1$-$5$ has at most $4$ degrees of freedom.  Hence $dim_{\textrm{fun}}(\mathcal{P}_{1,\ldots,1}) \leq 4$. 

To see that  $dim_{\textrm{fun}}(\mathcal{P}_{1,\ldots,1}) \geq 4$, consider a map $F_0$ of type 5 such that $F_0(x) > 0$ for all $x$. Then, whenever $A$ is an affine-linear map $A:\mathbb{R}^1 \to \mathbb{R}^1$ that has a positive slope and positive constant, the composition $\sigma \circ A \circ F_0$ is also a type 5 map with positive image. 
Furthermore, the composition $\sigma \circ A \circ F_0$ has $4$ degrees of freedom (for a similar reason as does the map $F$ in the justification of Lemma \ref{l:1111}).  
\end{proof}


\subsection{Architectures of depth 1} \label{ss:arcdep1}

\begin{proposition} \label{p:depth1}
For any architecture of depth $1$ (i.e. a single layer map), the parameter space and moduli space have the same dimension.  
That is, for any $(n_1,n_2) \in \mathbb{N}^2$, there exists a parameter $\theta \in \mathcal{P}_{n_1,n_2}$ such that $$\textrm{dim}_{\textrm{fun}}(s) =  n_2(n_1+1).$$
Furthermore, there exists $\theta$ as above for which there exists a set $Z \subset \mathbb{R}^{n_1}$ so that the ternary labels of all points in $Z$ for $\theta$ have no $0$s, and 
 $\textrm{dim}_{\textrm{fun}}(\theta)$ is realized on $Z$, i.e. $\textrm{dim}_{\textrm{fun}}(\theta) = \textrm{rank }E_Z \vert_\theta$.
\end{proposition}

\begin{proof}
Fix any convex polytope in $\mathbb{R}^{n_1}$ that has at least $n_2$ codimension 1 faces.  Let $\theta$ be a parameter that determines an arrangement of $n_2$ distinct hyperplanes that are the affine hulls of codimension 1 faces of the polytope, oriented outwards.   Then for each $i$, there exists an open set $U_i \subset \mathbb{R}^n$ so that $U_i$ is contained in the positive side of $H_i$ and in the negative side of $H_j$ for all $j \neq i$; furthermore, the sets $U_i$ can be chosen so that they retain this property for suitably small perturbations of $s$.  It is then clear that all points in $\bigcup_iU_i$ are parametrically smooth for $\theta$.  
For each $i=1,\ldots,m$, let $Z_i$ be a set consisting of $n+1$ points in $U_i$ that form the vertices of a geometric simplex contained in $U_i$.  Let $Z = \bigcup_i Z_i$.  
For any point $x=(x_1,\ldots,x_n) \in \cup_i U_i$, 
$\boldsymbol{J}E_{\{x\}}(\theta)$ is the $m(n+1) \times 1$ row vector 
\begin{multline}
[\delta_1(x)x_1, \ldots, \delta_1(x)x_n, \delta_1(x), \\ 
\delta_2(x)x_1, \ldots, \delta_2(x)x_n, \delta_2(x),  \\
\ldots,\\
 \delta_m(x)x_1,\ldots, \delta_m(x)x_n, \delta_m(x) ]
 \end{multline}
where $\delta_i(x) = 1$ if $x$ is on the positive side of the co-normed hyperplane associated to the $i$th row of $A$, and $\delta_i(x) = 0$ otherwise.  
 Then $\boldsymbol{J}E_Z(\theta)$ is (with a suitable ordering of points in $Z$) a $m(n+1) \times m(n+1)$ matrix with $(n+1) \times (n+1)$ blocks along the main diagonal and  $0$s in all other entries; the $i$th block along the main diagonal is precisely $\boldsymbol{J}E_{Z_i}(f_i)$ for the function $f_i:\mathbb{R}^{n} \to \mathbb{R}$ that is the $i$th coordinate function of $\rho(\theta)$.  Then $$\textrm{rank}(\boldsymbol{J}E_Z(\theta)) = \sum_{i=1}^m \textrm{rank} (\boldsymbol{J}E_{Z_i}(f_i)) = \sum_{i=1}^m (n+1) = m(n+1).$$

\end{proof}

\section{Continuity of functional dimension} \label{s:continuity}

While we have seen that functional dimension is nonconstant on parameter space, we have not thus far addressed to what extent functional dimension  depends continuously on the parameter.  Recall that we only defined functional dimension for ordinary parameters, and that the set of ordinary parameters is open and has full measure.

First, we address continuity of the map $u \mapsto \boldsymbol{J}E_Z(u)$, rather than the rank of this map. 

\begin{lemma} \label{l:opennbhd2}
Fix an ordinary parameter $\theta \in \mathcal{P}_{n_0,\ldots,n_m}$.  Suppose $Z \subset \mathbb{R}^{n_0}$ is a finite set that is parametrically smooth for all parameters $u$ in some open neighborhood $U$ of $\theta$.  Then the map $U \owns u \mapsto \boldsymbol{J}E_Z(u)$ is continuous. \end{lemma}

\begin{proof}
The statement that each $z \in Z$ is parametrically smooth for all $u \in U$ means that the map $u \mapsto \rho(u)(z)$ is a polynomial in the coordinates of the parameter (by Theorem \ref{t:Fpiecewisepolynomial}).  The row of $JE_Z(\cdot)$ corresponding to a point $z \in Z$ consists of the partial derivatives of this polynomial with respect to the various parameter coordinates.  Thus, the entries of $JE_Z(u)$ vary continuously with $u$.

\end{proof}

\begin{theorem} \label{t:identificationspace}
Let $(n_0,\ldots,n_m)$ be an architecture.  For any $k \in \mathbb{N}$, 
$$\left \{\theta \in \mathcal{P}_{n_0,\ldots,n_m} : s \textrm{ is ordinary and }\textrm{dim}_{\textrm{fun}}(\theta) \geq k\right \}$$
is an open subset of $\mathcal{P}_{n_0,\ldots,n_m}$.
\end{theorem}

\begin{proof}
 A standard result in linear algebra is that rank is lower-semicontinuous, that is,  for an integer rank $r \geq 0$, and $m, n \in \mathbb{N}$, the set of $(m \times n)$ matrices of rank $\geq r$ is an open subset of $\mathbb{R}^{m \times n}$.  
The result follows immediately by combining this fact with Lemmas \ref{l:opennbhd1} and \ref{l:opennbhd2}.
\end{proof}

\begin{corollary}
Let $(n_0,\ldots,n_m)$ be an architecture.  Let $(\theta_i)_{i \in \mathbb{N}}$ be a sequence of points in $\mathcal{P}_{n_0,\ldots,n_m}$ that converges to a point $\theta_{\infty} \in \mathcal{P}_{n_0,\ldots,n_m}$.  Then $$\limsup_{i \to \infty} \textrm{dim}_{\textrm{fun}}(\theta_i) \geq \textrm{dim}_{\textrm{fun}}(\theta_{\infty}).$$ 
\end{corollary}

Theorem \ref{t:identificationspace} tells us that we may think of parameter space as the result of gluing together regions on which 
 functional dimension is constant.  The structure of this decomposition of parameter space into regions of constant functional dimension may have implications for training.


\section{Symmetries and fibers} \label{s:fibers}

\begin{definition} \label{def:fiber}
For any function $f \in \mathcal{M}_{n_0,\ldots,n_m}$, the \emph{fiber} of $f$ in $\mathcal{P}_{n_0,\ldots,n_m}$ is the set 
$\rho^{-1}(f)$. 
\end{definition}

\begin{example}
We will give examples of two points in the fiber 
$$\rho^{-1}(x \mapsto \max\{0,x+1\})$$ in $\mathcal{P}_{1,2,2,1}$.
First, note that for $x \in \mathbb{R}$,
$$\sigma(- \sigma(-x)+1) = \begin{cases} 0 & \textrm{if }x \leq -1 \\
x+1 &  \textrm{if }-1 \leq x \leq 0 \\
1 & \textrm{if } x \geq 0.
\end{cases}$$
Hence, for all $x \in \mathbb{R}$, 

\begin{equation} \label{eq:twoways}
\sigma(\sigma(- \sigma(-x)+1) + \sigma(\sigma(x))) = \sigma(\sigma(\sigma(x+1)))
\end{equation}

The two sides of \eqref{eq:twoways} are the marked realizations of two different parameters in $\mathcal{P}_{1,2,2,1}$, both of which determine the function $x \mapsto \max\{0,x+1\}$. 

\end{example}

\begin{definition}
The \emph{symmetry group} of $\mathcal{P}_{n_0,\ldots,n_m}$ is the group $\textrm{Aut}(\mathcal{P}_{n_0,\ldots,n_m})$ consisting of all homeomorphisms $T$ of $\mathcal{P}_{n_0,\ldots,n_m}$ such that $\rho = \rho \circ T$. 
\end{definition} 

The symmetry group includes \emph{permutations} and \emph{rescalings}.  

\medskip 
\emph{Permutations:}  Consider any fixed architecture $\mathcal{P}_{n_0,\ldots,n_m}$ with $n \geq 2$.  Represent a parameter $\theta \in \mathcal{P}_{n_0,\ldots,n_m}$ as $\theta=(A_1,\ldots,A_m)$ where $A_i$ is the matrix associated to the $i$th layer map of $\bar{\rho}(\theta)$.  For $1 \leq i < m$, and distinct indices $1 \leq j,k \leq n_i$ define the \emph{permutation of the $j$th and $k$th neurons} of $\theta$ to be the map $\tau^i_{j,k} :\mathcal{P}_{n_0,\ldots,n_m} \to \mathcal{P}_{n_0,\ldots,n_m}$ defined by 
$$\tau^i_{j,k}: (A_1,\ldots,A_i, A_{i+1},\ldots A_m) \mapsto (A_1,\ldots,A_i', A_{i+1}',\ldots A_m)$$
where $A_i'$ is the matrix formed from $A_i$ by interchanging the $k$th and $j$th row, and $A_{i+1}'$ is the matrix formed from $A_{i+1}$ by interchanging the $k$th and $j$th columns.  The \emph{permutation symmetry group} of $\mathcal{P}_{n_0,\ldots,n_m}$ is the subgroup of $\textrm{Aut}(\mathcal{P}_{n_0,\ldots,n_m})$ generated by the set of all such maps $\tau^i_{j,k}$; clearly the permutation symmetry group is isomorphic to $S_{n_1} \times \ldots S_{n_{m-1}}$, where $S_i$ denotes the group of permutations of a set of cardinality $i$.

\medskip
\emph{Recalings:} Consider any fixed architecture $\mathcal{P}_{n_0,\ldots,n_m}$ with $n \geq 2$.
As in the proof of Theorem \ref{t:upperbound},  for any $1 \leq i<m$, if $h$ is a linear self-map of $\mathbb{R}^{n_i}$ that is represented by a diagonal matrix with all positive entries, then 
 \begin{equation} 
 \sigma \circ A_{i+1}\circ h^{-1} \circ \sigma \circ h\circ A_i=\sigma \circ A_{i+1}\circ \sigma \circ A_i.
 \end{equation}
 Thus the map $\mathcal{P}_{n_0,\ldots,n_m} \to \mathcal{P}_{n_0,\ldots,n_m}$ given by 
 $$(A_1,\ldots,A_i, A_{i+1}, \ldots, A_m) \mapsto (A_1,\ldots,A_i'', A_{i+1}'', \ldots, A_m)$$
 where $A_i'' \coloneqq h\circ A_i$ and $A_{i+1}'' \coloneqq A_{i+1}\circ h^{-1}$, is in $\textrm{Aut}(\mathcal{P}_{n_0,\ldots,n_m})$.  The \emph{rescalings symmetry group} is the subgroup of $\textrm{Aut}(\mathcal{P}_{n_0,\ldots,n_m})$ generated by all such maps.

\bigskip

Not all points in a fiber look the same from the point of view of what functions can be realized by parameters near them.  

\begin{example} \label{ex:nontransitive}
The realization map $\rho$ on $\mathcal{P}_{1,1} \cong \mathbb{R}^2$ is given by 
$$\rho(a,b):x \mapsto \sigma(ax+b).$$ 
Set $\theta_1 = (0,0)$ and $\theta_2 = (0,-1)$.  Then both $\rho(\theta_1)$ and $\rho(\theta_2)$ are the constant function $0$. 

First, we note that given any real number $\epsilon > 0$ and any open neighborhood $U_1 \subset \mathcal{P}_{1,1}$ of $s_1$, there exists a parameter $u \in U_1$ such that $\rho(u)(x) > 0$ for some $x \in \mathbb{R}^1$ with $|x| < \epsilon$.  For example, if $r>0$ is such that $U_1$ contains an open ball of radius $r$ centered at $s_1$, then the parameter $u=(0,\min\{r/2,\epsilon/2\})$ is in $U_1$ and $\rho(u)(x)$ is the constant function $\min\{r/2,\epsilon/2\}$. 

Next, note that this property does not hold for sufficiently small neighborhoods of $u_2$.  That is, for any fixed real number  $1> \epsilon > 0$, if $U_2 \subset \mathcal{P}_{1,1}$ is a sufficiently small neighborhood of $\theta_2$, every $u \in U_2$ satisfies $\rho(u)(x) = 0$ for all $|x| < \epsilon$.  \end{example}

\begin{proposition}
There exist fibers on which the symmetry group does not act transitively.  
\end{proposition}

\begin{proof}
Example \ref{ex:nontransitive} gives two parameters $\theta_1$ and $\theta_2$ in the fiber $\rho^{-1}(0)$ in $\mathcal{P}_{1,1}$ such that every neighborhood of $\theta_1$  contains (parameters whose corresponding) functions that cannot be realized by parameters in any sufficiently small neighborhood of $s_2$.  Consequently, no element of $\textrm{Aut}(\mathcal{P}_{1,1})$ can send $\theta_2$ to $\theta_1$. 
\end{proof}

\begin{example} \label{ex:fundimvaries}
We will give an example of two parameters in the fiber $\rho^{-1}(x \mapsto \sigma(x+1) )$ in $\mathcal{P}_{1,2,1}$ that have different functional dimensions.  The realization map is given by 
$$\rho(a,b,c,d,e,f,g) = x \mapsto \sigma \left( e \sigma(ax+b)+f \sigma (cx+d) + g \right).$$

\bigskip

Consider the parameter $\theta_1 = (1,0,-1,0,+1,-1,+1)$,
$$\rho(\theta_1):x \mapsto \sigma( \sigma(x+0) - \sigma(-x+0) +1).$$
Fix a small $\epsilon > 0$.  Then there exists a neighborhood $U$ of $\theta_1$ on which $\rho(u)$ has \emph{at least three} pieces, and satisfies the following:
\begin{itemize}
\item On the unbounded interval $(-\infty,-1-\epsilon)$, $\rho(u)$ is the constant function $0$.
\item On the interval $(-1+\epsilon, -\epsilon)$, $\rho(u)$ is given by 
$$x \mapsto f\sigma(cx+d) +g  \quad (\approx x +1).$$ 
\item On the interval  $(\epsilon, \infty)$, $\rho(u)$ is given by 
$$e(ax+b) + g \quad (\approx x+1).$$
\end{itemize}
By perturbing the parameter, we can independently vary the affine-linear functions that give $\rho(u)$ on the latter two intervals listed above.  Hence $\textrm{dim}_{\textrm{fun}}(\theta_1) \geq 4$. 

\bigskip
Now consider the parameter $\theta_2 = (1,1,-1,-2, 1,-1,0)$,
$$\rho(\theta_2):x \mapsto \sigma( \sigma(x+1) - \sigma(-x-2) +0).$$
Again, fix a small $\epsilon > 0$.  Then there exists a neighborhood $U$ of $\theta_2$ on which the function (ignoring the outer $\sigma$ in $\rho(u)$)
 $$x \mapsto e \sigma(ax+b)+f \sigma (cx+d) + g$$ has three pieces and satisfies the following:
 \begin{itemize}
 \item On the interval $(-\infty, -2-\epsilon)$, the function is given by 
 $$x \mapsto f \sigma (cx+d) + g \quad (\approx x +2).$$
 \item On the interval $(-2+\epsilon, -1-\epsilon)$, the function is given by 
 $$x \mapsto g \quad (\approx 0).$$
 \item On the interval $(-1+\epsilon, \infty)$, the function is given by 
 $$x \mapsto e(ax+b) + g \quad (\approx x+1).$$
 \end{itemize}
From this, we can see that 
\begin{itemize}
\item If $g \leq 0$, $\rho(u)$ has two intervals; $\rho(u)$ is the constant function $0$ on the unbounded interval $(-\infty, \approx -1)$ and is given by $e(ax+b)+g$ on the other unbounded interval $(\approx -1, \infty)$.  
\item If $g > 0$, $\rho(u)$ has four intervals; $\rho(u)$ is the constant function $0$ on the unbounded interval $(-\infty, \approx -2)$, has positive slope on a tiny interval around $-2$, is the constant function $g$ on the interval $(\approx -2, \approx -1)$, and has slope $ea$ on the interval $(\approx -1, \infty)$. 
\end{itemize}
It follows that points in $(-2,-1)$ are not parametrically smooth for $\theta_2$, while points in $(-\infty, -2)$ and $(-1,\infty)$ are, and that $\rho(u)$ is uniformly $0$ on points in $(-\infty, -2)$. Varying $u$ can change the affine-linear function giving $\rho(u)$ on the unbounded interval $(\approx -1,\infty)$. Hence $\textrm{dim}_{\textrm{fun}}(\theta_2) = 2$.  
  \end{example}

Example \ref{ex:fundimvaries} immediately gives us the following proposition. 

\begin{proposition} There exist fibers on which functional dimension is not constant. \end{proposition}

Recall that a subset $X$ of a topological space $Y$ is said to be \emph{disconnected} if there exist two disjoint, open subsets $A$ and $B$ in $Y$ such that $A \cap X \neq \emptyset$ and $B \cap X \neq \emptyset$; otherwise $X$ is said to be \emph{connected}.

\begin{proposition}
Disconnected fibers exist.  
\end{proposition}

\begin{proof}
Example \ref{ex:disconnectedfiber}.
\end{proof}

\begin{example} \label{ex:disconnectedfiber}
We will show that the fiber of the function $x \mapsto |x|$ in $\mathcal{P}_{1,2,1}$ is disconnected. 

The marked realization map $\bar{\rho}$ on $\mathcal{P}_{1,2,1}$ is the map that sends a point $(a,b,c,d,e,f,g) \in \mathcal{P}_{1,2,1}$ to the function 
\begin{equation} \label{eq:parametrizedmaps} x \mapsto \sigma(e \sigma(ax+b) + f \sigma(cx+d) + g).\end{equation}

Suppose $\rho(a,b,c,d,e,f,g)$ is the function $x \mapsto |x|$.  Then the outermost ReLU in \eqref{eq:parametrizedmaps} must act as the identity, i.e. $\rho(a,b,c,d,e,f,g)$ may be written as
\begin{equation} \label{eq:paramfun2} 
x \mapsto e \sigma(ax+b) + f \sigma(cx+d) + g.
\end{equation}

Since the function in \eqref{eq:paramfun2} is nondifferentiable at the points $-b/a$ and $-d/c$, while $x \mapsto |x|$ has a single point of nondifferentiability at $0$, we must have 
$-\frac{b}{a} = -\frac{d}{c} = 0$, implying $b = d = 0$.  Thus, $\rho(a,b,c,d,e,f,g)$ may be written as 
$$x \mapsto e \sigma(ax) + f \sigma(cx) + g.$$
So we must have $g = 0$, and hence $\rho(a,b,c,d,e,f,g)$ may be written as 
$$x \mapsto e\sigma(ax) +  f \sigma(cx).$$
Therefore, either i) $ea = 1$ and $fc = -1$, or ii) $ea = -1$ and $fc = +1$.  Because the outermost ReLU in \eqref{eq:parametrizedmaps} acts as the identity, we must have the $e> 0$ and $f> 0$. It is also not hard to see that these conditions are sufficient to ensure that $\rho(a,b,c,d,e,f,g) = x \mapsto |x|$. 
So $\rho^{-1}(x \mapsto |x|) = $
\begin{multline}  \label{eq:twosets}
\{(a,b,c,d,e,f,g) \in \mathcal{P}_{1,2,1} \mid  b = d= g = 0, e > 0, f > 0, ea=1, fc=-1 \} \bigcup \\
\{(a,b,c,d,e,f,g) \in \mathcal{P}_{1,2,1} \mid  b = d= g = 0, e > 0, f > 0, ea=-1, fc=1 \}. \\
\end{multline}
The two sets in \eqref{eq:twosets} are clearly nonempty but have empty intersection. 
\end{example}

\bibliographystyle{plain}
\bibliography{dimensionbibliography.bib}

\end{document}